%% file: NonunimodularPercolation_Revised5.tex
\documentclass[11pt, a4, reqno]{article}
\input{header}

\input{commands}

\usepackage[margin=0.85in]{geometry}
\usepackage{extarrows}
\usepackage[numbers,sort&compress]{natbib}
\usepackage{commath}
\usepackage{mathtools}
\newcommand{\eps}{\varepsilon}
\usepackage{bbm}
\usepackage{setspace}
\usepackage{enumitem}
\usepackage{tikz}
\usepackage{relsize}
\newcommand{\Aut}{\operatorname{Aut}}

\newcommand{\bP}{\mathbf P}
\newcommand{\bE}{\mathbf E}

\newcommand{\Peak}{\operatorname{peak}}
\newcommand{\peak}{\operatorname{peak}}
\newcommand{\stab}{\operatorname{Stab}}
\newcommand{\myfrac}[3][0pt]{\genfrac{}{}{}{}{\raisebox{#1}{$#2$}}{\raisebox{-#1}{$#3$}}}

\makeatletter
\pgfdeclareshape{slash underlined}
{
  \inheritsavedanchors[from=rectangle] % this is nearly a circle
  \inheritanchorborder[from=rectangle]
  \inheritanchor[from=rectangle]{north}
  \inheritanchor[from=rectangle]{north west}
  \inheritanchor[from=rectangle]{north east}
  \inheritanchor[from=rectangle]{center}
  \inheritanchor[from=rectangle]{west}
  \inheritanchor[from=rectangle]{east}
  \inheritanchor[from=rectangle]{mid}
  \inheritanchor[from=rectangle]{mid west}
  \inheritanchor[from=rectangle]{mid east}
  \inheritanchor[from=rectangle]{base}
  \inheritanchor[from=rectangle]{base west}
  \inheritanchor[from=rectangle]{base east}
  \inheritanchor[from=rectangle]{south}
  \inheritanchor[from=rectangle]{south west}
  \inheritanchor[from=rectangle]{south east}
  \inheritanchorborder[from=rectangle]
  \foregroundpath{
    \southwest \pgf@xa=\pgf@x \pgf@ya=\pgf@y
    \northeast \pgf@xb=\pgf@x \pgf@yb=\pgf@y
    \pgf@xc=\pgf@xa
    \advance\pgf@xc by .5\pgf@xb
    \pgf@yc=\pgf@ya
    \advance\pgf@xc by -1.3pt
    \advance\pgf@yc by -1.8pt
    \pgfpathmoveto{\pgfqpoint{\pgf@xc}{\pgf@yc}}
    \advance\pgf@xc by  2.6pt
    \advance\pgf@yc by  3.6pt
    \pgfpathlineto{\pgfqpoint{\pgf@xc}{\pgf@yc}}
    \pgfpathmoveto{\pgfqpoint{\pgf@xa}{\pgf@ya}}
    \pgfpathlineto{\pgfqpoint{\pgf@xb}{\pgf@ya}}
 }
}
\tikzset{nomorepostaction/.code=\let\tikz@postactions\pgfutil@empty}

\newcommand\nxleftrightarrow[2][]{%
  \mathrel{\tikz[baseline=-.7ex] \path node[slash underlined,draw,<->,anchor=south] {\(\scriptstyle #2\)} node[anchor=north] {\(\scriptstyle #1\)};}}
\makeatother

\newcommand{\pcl}[1]{p_c({#1})}

\author{Tom Hutchcroft\footnote{Statslab, DPMMS, University of Cambridge}}

 \title{
\textsc{Non-uniqueness
 and mean-field criticality
  for percolation on nonunimodular transitive graphs}
 }
\newcommand{\nexp}{\overline{ \exp \mkern-1mu}\mkern1mu }
\newcommand{\nlog}{\operatorname{l\overline{og\mkern-1mu}\mkern1mu}}

\renewcommand{\chi}{\mathcal{X}}

\begin{document}
\maketitle

\setstretch{1.1}
\begin{abstract}
We study Bernoulli bond percolation on nonunimodular quasi-transitive graphs, and more generally graphs whose automorphism group has a  nonunimodular quasi-transitive subgroup. 
We prove that  percolation on any such graph has a non-empty phase in which there are infinite \emph{light} clusters, which implies the existence of a non-empty phase in which there are \emph{infinitely many} infinite clusters. That is, we show that $p_c<p_h \leq p_u$ for any such graph. This answers a question of H\"aggstr\"om, Peres, and Schonmann (1999), and verifies the nonunimodular case of a well-known conjecture of Benjamini and Schramm (1996).
We also prove that the triangle condition holds at criticality
 on any such graph, 
which implies that various critical exponents exist and take their mean-field values.

All our results apply, for example, to the product $T_k\times \Z^d$ of a $k$-regular tree with $\Z^d$ for $k\geq 3$ and $d \geq 1$, for which these results were previously known only for large~$k$.
Furthermore, our methods also enable us to establish the basic topological features of the phase diagram for \emph{anisotropic} percolation on such products, in which tree edges and $\Z^d$ edges are given different retention probabilities. These features had  only previously been established for $d=1$, $k$ large.
\end{abstract}
\setstretch{1.1}

\newgeometry{margin=1.2in}

\newpage

\tableofcontents

\newgeometry{margin=0.8in}

\newpage

\section{Introduction}
\label{sec:Intro}

In \textbf{Bernoulli bond percolation}, the edges of a connected, locally finite graph $G$ are each either deleted or retained independently at random, with retention probability $p$, to obtain a random subgraph $G[p]$ of $G$. Retained edges are referred to as \textbf{open} and deleted edges are referred to as \textbf{closed}. The connected components of $G[p]$ are referred to as \textbf{clusters}. When $G$ is infinite, the \textbf{critical parameter} is defined to be
% \subsection{Multiplicity of phase transitions}
\[
p_c=p_c(G) = \inf\{p : G[p] \text{ contains an infinite cluster almost surely}\}
\]
and the \textbf{uniqueness threshold} is defined to be
\[
p_u=p_u(G) = \inf\{ p: G[p]
\text{ contains a unique infinite cluster almost surely}
\}.
\]
A principal question concerns the equality or inequality of these two critical parameters. For Euclidean lattices such as the hypercubic lattices $\Z^d$, this question has now been well-understood for thirty years: 
% In this context, 
% A natural class of graphs on which to study this question are the \textbf{quasi-transitive} graphs.
% An argument of Newman and Schulman \cite{newman1981infinite} implies that if $G$ is quasi-transitive then the number of infinite clusters of $G[p]$ is either $0,1,$ or $\infty$ almost surely.  
Aizenman, Kesten, and Newman \cite{MR901151} proved in 1987 that $\Z^d[p]$ has at most one infinite cluster almost surely for every $d\geq 1$ and $p\in[0,1]$, so that in particular $p_c(\Z^d)=p_u(\Z^d)$ for every $d\geq 1$. A beautiful alternative proof of the same result was obtained by Burton and Keane~\cite{burton1989density} in 1989.
In their influential paper \cite{bperc96}, Benjamini and Schramm proposed a systematic study of percolation on general \textbf{quasi-transitive} graphs, that is, graphs whose automorphism group has only finitely many orbits. They made the following conjecture.

\begin{conjecture}[Benjamini and Schramm 1996]
\label{conj:pcpu}
Let $G$ be a connected, locally finite, quasi-transitive graph. Then $p_c(G)<p_u(G)$ if and only if $G$ is nonamenable. 
\end{conjecture}

Gandolfi, Keane, and Newman \cite{gandolfi1992uniqueness} showed that Burton and Keane's proof
% \footnote{The Aizenman-Kesten-Newman proof can also be generalized to quasi-transitive graphs satisfying the stronger assumption of subexponential growth.}
 generalizes to all amenable quasi-transitive graphs, so that only the `if' direction of \cref{conj:pcpu} is open. 
H\"aggstr\"om, Peres, and Schonmann \cite{MR1676835,MR1676831,HPS99} proved that $G[p]$ has a unique infinite cluster almost surely whenever $G$ is quasi-transitive and $p>p_u$ (see also \cite{LS99}). For results on the complementary question of which graphs have $p_u<1$, see \cite{MR1622785,MR2286517}. For further background on percolation, we refer the reader to \cite{grimmett2010percolation,1707.00520,heydenreich2015progress} for the Euclidean case, to \cite{LP:book,MR2884871} for more general graphs, and to \cite{MR2280297} for a survey of work related specifically to \cref{conj:pcpu}.

% Despite sustained interest over a twenty year period, relatively little progress has been made on this conjecture. 
Most progress made on \cref{conj:pcpu} so far has taken a \emph{perturbative} approach. In these works, bounds on $p_c$ and $p_u$ are obtained via combinatorial methods, and these bounds are then shown to separate $p_c$ and $p_u$ when suitable parameters associated to a graph are made large or small as appropriate.
  In particular, it is known that $p_c(G)<p_u(G)$ if $G$ is a transitive nonamenable graph  with large Cheeger constant \cite{MR1833805}, with small spectral radius \cite{bperc96,MR1756965}, or with large girth \cite{MR3005730}. 
% is, it is assumed that some parameters associated to the graph are large (or small) 
The criterion concerning the spectral radius was used by Pak and Smirnova-Nagnibeda~\cite{MR1756965} to deduce that every nonamenable discrete group has at least one Cayley graph for which $p_c<p_u$.

The only previous \emph{non-perturbative} results on \cref{conj:pcpu} we are aware of (other than the infinitely ended case, which is trivial) are due to Benjamini and Schramm \cite{BS00}, who verified the conjecture for nonamenable planar quasi-transitive graphs (see also the earlier work of Lalley \cite{MR1614583}), and Gaboriau and Lyons \cite{MR1757952,MR3009109,MR2221157}, who verified the conjecture for all quasi-transitive graphs admitting non-constant harmonic Dirichlet functions (i.e., having positive first $\ell^2$-Betti number). 
% The latter work relies on the work of Lyons, Peres, and Schramm \cite{LPS06}, who in particular proved that $p_c<p_u$ if and only if the free and wired minimal spanning forests are distinct.
 These are all quite strong assumptions, and \cref{conj:pcpu} remains open for most examples. In particular, neither of these classes include any graph of the form $G \times H$ where $G$ and $H$ are both infinite. Let us also note that, in joint work with Angel \cite{1710.03003}, we constructed a counterexample demonstrating that the natural generalisation of \cref{conj:pcpu} to \emph{unimodular random rooted graphs} is \emph{false}. 
 % his class is rather restrictive.
 % The original analysis of percolation on $T\times \Z^d$ by Grimmett and Newman \cite{grimmett2010percolation} is also essentially perturbative in nature.
% 

Historically, the first example of a transitive graph with $p_c<p_u<1$ was given by Grimmett and Newman \cite{MR1064560}, who studied percolation on $T_k \times \Z^d$, the Cartesian  product (a.k.a.\ box product)  of a $k$-regular tree $T_k$ with the hypercubic lattice $\Z^d$. They proved in particular  that
 % $p_u(T_k\times \Z^d)<1$ for all $k\geq 3,d\geq 1$ and that 
 $p_c(T_k\times \Z^d)<p_u(T_k \times \Z^d)$ when $k$ is sufficiently large.
 % (this can also be deduced by the criteria of either Schonmann \cite{MR1833805} or Pak and Smirnova-Nagnibeda \cite{MR1756965}, discussed above). 
% 
In the special case of $d=1$, they also applied the well-understood theory of percolation on $\Z^2$ to prove a strong form of this result concerning \emph{anisotropic} percolation on $T_k\times \Z$, again under the assumption that $k$ is large. 
Since then, refinements of the perturbative methods have lead to smaller values of $k$ being treatable (in the isotropic case). In particular, recent work of 
Yamamoto \cite{yamamoto2017upper} has shown that $p_c(T_k \times \Z)<p_u(T_k\times \Z)$ for all $k\geq 4$.

In this paper, we verify \cref{conj:pcpu} for every  graph whose automorphism group has a \emph{quasi-transitive nonunimodular subgroup}. Our proof makes direct use of the recent result  that critical percolation on any such graph does not have any infinite clusters almost surely \cite{Hutchcroft2016944}, which built upon previous work of Benjamini, Lyons, Peres, and Schramm \cite{peres2006critical} and Tim\'ar \cite{timar2006percolation}.

\begin{thm}
\label{thm:pcpu}
Let $G$ be a connected, locally finite graph, and suppose that $\Aut(G)$ has a quasi-transitive nonunimodular subgroup. Then $p_c(G)<p_u(G)$.
\end{thm}

In particular, \cref{thm:pcpu} allows for a complete analysis of products with trees. 
Being non-perturbative, our methods can easily be adapted to analyze anisotropic bond percolation, enabling us to extend the full analysis of Grimmett and Newman \cite{MR1064560} to all $k\geq 3$ and $d\geq 1$. See \cref{subsec:inhomog} for details. 

\begin{corollary}
\label{cor:TxZ}
Let $T_k$ be the $k$-regular tree with $k\geq 3$ and let $G$ be a connected, locally finite,  quasi-transitive graph. Then $p_c(T_k\times G) < p_u(T_k \times G)$.
\end{corollary}

In light of the work of Lyons, Peres, and Schramm \cite{LPS06}, \cref{thm:pcpu} also has the following corollary. See that paper and \cite[Chapter 11]{LP:book} for background on minimal spanning forests.

\begin{corollary}
\label{cor:MSF}
Let $G$ be a connected, locally finite graph, and suppose that $\Aut(G)$ has a quasi-transitive nonunimodular subgroup. Then the free and wired minimal spanning forests of $G$ are distinct.
\end{corollary}

Here, a subgroup $\Gamma \subseteq \Aut(G)$ of the group of automorphisms of a connected, locally finite graph $G$ is said to be \textbf{transitive} if its action on $V$ is transitive, i.e.\ has exactly one orbit, and is said to be \textbf{quasi-transitive} if its action on $G$ has only finitely many orbits. $\Gamma$ is said to be \textbf{unimodular} if 
\[|\stab_v u|=|\stab_{u} v|\]
for every $u,v\in V$ in the same orbit of $\Gamma$, where $\stab_x$ is the stabilizer of $x$ in $\Gamma$ and $\stab_x y$ 
% = \{ \lambda y : \lambda \in \stab_x \}$
 is the orbit of $y$ under $\stab_x$. Otherwise $\Gamma$ is said to be \textbf{nonunimodular}.
  % A quasi-transitive graph $G$ is said to be unimodular or nonunimodular according to whether or not its full automorphism group $\Aut(G)$ is unimodular or nonunimodular.

A prototypical example of a graph with a transitive nonunimodular subgroup is the $d$-regular tree $T$ for $d\geq 3$, together with the subgroup of $\Gamma_\xi$ of $\Aut(T)$ fixing some specified end $\xi$ of $T$. This example allows us to build many others, including examples where the full automorphism group is nonunimodular such as the grandparent graph \cite{MR811571} and the Diestel-Leader graphs \cite{MR1856226}; see e.g.\ \cite{timar2006percolation} for further examples. 
Moreover, if $G$ has a quasi-transitive nonunimodular subgroup $\Gamma \subseteq \Aut(G)$ and $H$ is quasi-transitive, then the product $G\times H$ has a quasi-transitive nonunimodular subgroup isomorphic to $\Gamma \times \Aut(H)$. 
In particular, if $T_k$ is a $k$-regular tree with $k\geq 3$, $\Gamma_\xi$ is the group of automorphisms of $T$ fixing some specified end $\xi$ of $T$, and $G$ is an arbitrary quasi-transitive graph, then the Cartesian  product $T \times G$ has a quasi-transitive nonunimodular subgroup isomorphic to the direct product $\Gamma_\xi \times \Aut(G)$, so that \cref{cor:TxZ} does indeed follow from \cref{thm:pcpu}. 
Similar statements hold for the free product of $G$ and $H$, the wreath product of $G$ and $H$, and so on, so that \cref{thm:pcpu} also applies, for example, to the lamplighter on the tree (with lamps taking values in an arbitrary transitive graph) or the lamplighter on an arbitrary quasi-transitive graph with lamps taking values in the tree. 

Let us note however that having a quasi-transitive nonunimodular subgroup of automorphisms is not a very robust property, and \cref{thm:pcpu} does \emph{not} apply to every Cayley graph of the direct product of the free Abelian group $\Z^d$ with a non-Abelian free group. Indeed, there are even Cayley graphs of the free group on two generators whose automorphism groups are discrete and therefore do not have any nonunimodular subgroups. Moreover, it follows from the work of De La Salle and Tessera \cite{de2015characterizing} that for every infinite finitely generated group $\Gamma$, there is a Cayley graph of a finite extension of $\Gamma$ whose automorphism group is discrete and therefore does not have any nonunimodular subgroups.

% A further fundamental example of a nonunimodular transitive graph is the (1-skeleton of the) \emph{affine Bruhat-Tits building} $X_d(\mathbb{Q}_p)$ associated to the special linear group $\operatorname{SL}_d(\mathbb{Q}_p)$, where $d\geq2$, $p$ is prime, and $\mathbb{Q}_p$ is the field of $p$-adic numbers. The graph $X_d(\mathbb{Q}_p)$ has a transitive subgroup $\Gamma_p^d$ of automorphisms isomorphic to the group $\operatorname{Aff}_{d-1}(\mathbb{Q}_p)$ of affine transformations of $\mathbb{Q}_p^{d-1}$, which is nonunimodular.  
% We can think of $X_d(\mathbb{Q}_p)$ as the p-adic analogue of $d$-dimensional hyperbolic space $\mathbb{H}^d$, and of $\Gamma_p^d$ as the p-adic analogue of the group of isometries of hyperbolic space fixing a point in the ideal boundary.
% When $d=2$ the pair $(X_2(\mathbb{Q}_p),\Gamma_p^2)$ is equivalent to the $(p+1)$-regular tree together with the group of automorphisms fixing an end.  The pairs $(X_d(\mathbb{Q}_p),\Gamma_p^d)$ with $d\geq 3$ can be thought of as higher dimensional analogues of this example. Their properties are rather different to those enjoyed by the examples in the previous paragraph: for example, if $d\geq 3$ then $X_d(\mathbb{Q}_p)$ arises as a Cayley graph of a group with Kazhdan's property (T). 

% Thus, \cref{thm:pcpu} allows us to complete the work started by Grimmett and Newman \cite{MR1064560}.
 % It is a surprising feature of the present work that have been able to solve the nonunimodular case

% Grimmett and Newman \cite{MR1064560}

% \begin{thm}

In previous work on percolation in the nonamenable setting, it has often been required to treat the unimodular and nonunimodular cases separately.  Thus, it is likely that this paper will be a component of any eventual solution to \cref{conj:pcpu}, with the unimodular case being treated separately. Indeed, since it first appeared, the results of the present paper have been used as part of a case-analysis in several further works concerning percolation on nonamenable graphs \cite{Hutchcroft2019Hyperbolic,hutchcroft2018locality,1904.10448}, see also \cite{tang2018heavy}. It is worth noting, however, that in the past it has been the \emph{unimodular} case that has been solved first, with the nonunimodular case requiring greater effort. 
% Indeed, some problems that have been solved in the unimodular case remain open in the nonunimodular case. 
The basic reason for this disparity is that the mass-transport principle is a much more powerful tool in the unimodular case than in the nonunimodular case \emph{when it
 % its nonuimodular variant when it 
 comes to obtaining proofs by contradiction.}
 % (We call the nonunimodular case of the mass-transport principle the \emph{tilted mass-transport principle}.) 
% 
However, we shall see that the nonunimodular (a.k.a.\ tilted) mass-transport principle remains a powerful tool \emph{for performing calculations}. Moreover, the presence of a quasi-transitive nonunimodular subgroup $\Gamma \subseteq \Aut(G)$ endows our graph with an invariantly defined decomposition into `expanding layers' that is foundational to the entire strategy used to prove \cref{thm:pcpu}.

% Indeed, an attractive feature of nonunimodular graphs in the context of \cref{conj:pcpu} is that there is a third natural critical value, the \emph{heaviness threshold} $p_h$, which appears to have more `structure' associated to it than $p_u$ does. On the other hand, we have that $p_c\leq p_h \leq p_u$, and so to prove \cref{thm:pcpu} is suffices to prove that $p_c<p_h$ for every connected, locally finite graph $G$ and every nonunimodular quasi-transitive 

In \cite{1709.10515}, we apply methods  similar to (but substantially simpler than) those used in this paper to analyze self-avoiding walk on the same class of graphs.

\subsection{The heaviness transition}
\label{subsec:intro_heavy}

If $G=(V,E)$ is a connected, locally finite graph and $\Gamma\subseteq \Aut(G)$ is transitive, the \textbf{modular function} of $(G,\Gamma)$ is the function $\Delta=\Delta_\Gamma :V^2 \to (0,\infty)$ defined by
\[\Delta(x,y) = \frac{|\stab_y x|}{|\stab_x y|},\]
so that $\Gamma$ is unimodular if and only if $\Delta(x,y) \equiv 1$. The \textbf{tilted mass-transport principle} states that if $F:V^2\to[0,\infty]$ is invariant under the diagonal action of $\Gamma$, meaning that that $F(x,y)=F(\gamma x, \gamma y)$ for every $\gamma \in \Gamma$, then
\[\sum_{v\in V} F(x,v) = \sum_{v\in V} F(v,x) \Delta(x,v)\]
for every $x\in V$.
See \cref{subsec:MTP} for definitions of the modular function and the tilted mass-transport principle in the general quasi-transitive case. 
% See \cite{BC2011} for a probabilistic interpretation of the modular function.  
% We say that $\Gamma$ is \textbf{unimodular} if $\Delta(x,y) \equiv 1$ for all $x,y \in V$, and that $\Gamma$ is \textbf{nonunimodular} otherwise.

% In fact, the proof of \cref{thm:triangle} also implies that, under the same hypotheses, there exists $\eps>0$ such that $\nabla_{p_c+\eps}(v)<\infty$ for every $v\in V$.

We say that a set of vertices $K \subseteq V$ is \textbf{heavy} 
if $\sum_{y\in K} \Delta(x,y) =\infty$ for some $x \in V$ (and hence every $x\in V$ by \cref{lem:modularsymmetries}, a.k.a.\ the cocycle identity), saying that $K$ is \textbf{light} otherwise. 
% 
 % Usually the group $\Gamma$ will be fixed and we will omit it from the terminology. 
% 
% Let $\omega_p$ be Bernoulli bond percolation on $G$. 
We define the \textbf{heaviness transition} to be
\begin{align*}
p_h=p_h(G,\Gamma) &= \inf\{ p \in [0,1] : G[p] \text{ contains a heavy cluster almost surely}\}.
% \end{align*}
\end{align*}
The heaviness transition was introduced by H\"aggstr\"om, Peres, and Schonmann \cite{HPS99} in the context of their work on indistinguishability.
% , who asked whether $p_c(G)<p_h(G,\Gamma)$ whenever $\Gamma \subseteq \Aut(G)$ is quasi-transitive and nonunimodular. 

Tim\'ar \cite[Lemma 5.2]{timar2006percolation} showed that for percolation clusters (but not for arbitrary sets), being heavy is almost surely equivalent to having \textbf{unbounded height}, meaning that $\sup_{y\in K}\Delta(x,y)=\infty$, and is also almost surely equivalent to having infinite intersection with some set of the form $\{ u \in V : e^s \leq \Delta(v,u) \leq e^t \}$ for some $v\in V$ and $s<t$, which we call a \textbf{slab}. Thus, we can also write $p_h$ in either of the following equivalent forms.
\begin{align}
p_h=p_h(G,\Gamma)&=\inf\left\{\hspace{.21em} p \in [0,1] :\hspace{.4455em} G[p] \text{ contains a cluster of unbounded height almost surely}\right\}
\nonumber
\\
&= \inf\left\{ p \in [0,1] : \begin{array}{l} G[p] \text{ contains a cluster that has infinite}\\ \text{intersection with some slab almost surely}\end{array}\right\}.
\label{eq:Timarph}
\end{align}

It is clear that if a unique infinite cluster exists then it must be heavy, and hence that $p_h(G,\Gamma) \leq p_u(G)$ for every $G$ and $\Gamma$.
Thus, in order to prove \cref{thm:pcpu}, it suffices to prove the following, stronger result, 
which answers positively a question of H\"aggstr\"om, Peres, and Schonmann \cite{HPS99}. 
% See [ref] for a discussion of 

\begin{thm}
\label{thm:pcph}
Let $G$ be a connected, locally finite graph and suppose that $\Gamma \subseteq \operatorname{Aut}(G)$ is transitive and nonunimodular. Then
$p_c(G) < p_h(G,\Gamma)$.
\end{thm}

\subsection{Critical exponents and the triangle condition}
\label{subsec:intro_criticalexponents}
% The proof of \cref{thm:pcpu} also leads to several further results concerning  percolation at and near the critical value $p_c$. 

The proof of \cref{thm:pcpu,thm:pcph} also yields a great deal of information about \emph{critical} percolation. In particular, it allows us to prove that many critical exponents associated to percolation on graphs with nonunimodular quasi-transitive subgroups of automorphisms exist and take their mean-field values. We refer the reader to \cite[Chapters 9 and 10]{grimmett2010percolation} for detailed background on critical exponents in percolation. We write $\bP_p$ and $\bE_p$ for the law of $G[p]$ and the associated expectation operator, respectively. We also write $\asymp$ for an equality that holds up to positive multiplicative constants.

% \begin{align}
% \chi_p(v) &\asymp (p_c-p)^{-\lambda}  &p &\nearrow p_c
% \label{exponent:susceptibility}
% \\
% \chi_{p}^{(k+1)}(v)/\chi_{p}^{(k)}(v) &\asymp (p_c-p)^{-\Delta} &k\geq 1,\, p &\nearrow p_c
% \label{exponent:gap}
% \\
% \theta_p(v) &\asymp (p-p_c)^\beta &p &\searrow p_c
% \label{exponent:theta}
% \\
% % M_{p_c,h}(v) &\asymp h^{1/2} &h&\searrow 0\\
% \bP_{p_c}( |K_v| \geq n) & \asymp n^{-1/\delta} &n&\nearrow \infty 
% \label{exponent:volume}
% \\
% \bP_{p_c}( \operatorname{rad}(K_v) \geq n) & \asymp n^{-1-1/\rho} &n&\nearrow \infty
% \label{exponent:radius}
% \\
% \bP_{p_c}( \operatorname{rad}_\mathrm{int}(K_v) \geq n) & \asymp n^{-1-1/\rho'} &n&\nearrow \infty.
% \label{exponent:intradius}
% \end{align}

\begin{thm}[Mean-field critical exponents]
\label{thm:criticalexponents}
Let $G$ be a connected, locally finite graph, and suppose that $\Aut(G)$ has a quasi-transitive nonunimodular subgroup. Then the following hold for each $v\in V$.
% \vspace{-1.5em}
\begingroup
\addtolength{\jot}{0.5em}
\begin{align}
\chi_p(v) &\asymp (p_c-p)^{-1}  &p &\nearrow p_c
\label{exponent:susceptibility}
\\
\chi_{p}^{(k+1)}(v)/\chi_{p}^{(k)}(v) &\asymp_k (p_c-p)^{-2} &k\geq 1,\, p &\nearrow p_c
\label{exponent:gap}
\\
\bP_{p}\left(|K_v|=\infty\right) &\asymp p-p_c &p &\searrow p_c
\label{exponent:theta}
\\
% M_{p_c,h}(v) &\asymp h^{1/2} &h&\searrow 0\\
\bP_{p_c}\left( |K_v| \geq n\right) & \asymp n^{-1/2} &n&\nearrow \infty 
\label{exponent:volume}
\\
\bP_{p_c}\left( \operatorname{rad}(K_v) \geq n\right) & \asymp n^{-1} &n&\nearrow \infty
\label{exponent:radius}
\\
\bP_{p_c}\left( \operatorname{rad}_\mathrm{int}(K_v) \geq n\right) & \asymp n^{-1} &n&\nearrow \infty.
\label{exponent:intradius}
\end{align}
\endgroup
\end{thm}

Here, the \textbf{susceptibility} $\chi_p(v)$ is defined to be the expected volume of the cluster at $v$, and $\chi_p^{(k)}(v)$ is defined to be the $k$th moment of the volume of the cluster at $v$. The implicit constants in \eqref{exponent:gap} depend on $k$. We denote the cluster at $v$ by $K_v$, writing $|K_v|$ for its volume, rad$(K_v)$ for its radius (i.e., the maximum distance in $G$ between $v$ and another point in $K_v$) and rad$_\mathrm{int}(K_v)$ for its \textbf{intrinsic radius} (also known as the \textbf{chemical radius}, i.e., the maximum distance in $G[p]$ between $v$ and and another point in $K_v$). The exponent described by \eqref{exponent:gap} is known as the \textbf{gap exponent}. In the traditional notation for percolation critical exponents \cite[Chapter 9]{grimmett2010percolation}, \cref{thm:criticalexponents} states that $\gamma=1$, $\Delta=2$, $\beta=1$, $\delta=2$, and $\rho=1$. We also remark that \cref{thm:criticalexponents}  conclusively resolves \cite[Question 3.3]{MR1833805}. 
% ; other exponent names are  self-explanatory.
% 
% 

The \emph{lower bounds} of \eqref{exponent:susceptibility}, \eqref{exponent:theta}, and \eqref{exponent:volume} were shown to hold for \emph{all} transitive graphs by Aizenman and Barsky \cite{aizenman1987sharpness}, whose proof was generalised to the quasi-transitive case by Antunovi\'c and Veseli\'c~\cite{antunovic2008sharpness}. Beautiful new proofs of  these results in the transitive case have recently been obtained by Duminil-Copin and Tassion \cite{duminil2015new}.
% , parts of which have been generalised to the quasi-transitive case in the Master's thesis of Beekenkamp \cite{beekenkamp2016sharpness}. 

% Many of the lower bounds of \cref{thm:criticalexponents} are known to hold in great 
% With the exception of \eqref{exponent:radius}, all the \emph{lower bounds} of \cref{thm:criticalexponents} hold for all quasi-transitive graphs. 
% The lower bounds of \eqref{exponent:susceptibility}, \eqref{exponent:gap}, \eqref{exponent:theta}, and \eqref{exponent:volume} are known to hold for all quasi-transitive graphs: 

% Meanwhile, Durrett and Nguyen \cite{durrett1985thermodynamic} and Nguyen \cite{MR923855} showed that the lower bound of \eqref{exponent:susceptibility} implies the lower bound of \eqref{exponent:gap} for any bounded degree graph. 
The remaining bounds of \cref{thm:criticalexponents} are intimately related to the \emph{triangle condition}, which is a well-known  signifier of mean-field behaviour.
Let $G=(V,E)$ be a connected, locally finite, quasi-transitive graph, and for each $u,v\in V$ and $p\in [0,1]$ let $\tau_p(u,v)$ be the probability that $u$ and $v$ are connected in $G[p]$. This is known as the \textbf{two-point function}. For each $p\in[0,1]$ and $v\in V$, the \textbf{triangle diagram} is defined to be
\[
\nabla_p(v) =\sum_{x,y\in V}\tau_{p}(v,x)\tau_{p}(x,y)\tau_{p}(y,v). 
\]
We say that $G$ satisfies the \textbf{triangle condition at $p$} if $\nabla_p(v)<\infty$ for every $v\in V$.

The triangle condition was introduced by Aizenman and Newman \cite{MR762034}, who showed that if $G=\Z^d$ and the triangle condition holds at $p_c$ then the upper bounds of \eqref{exponent:susceptibility} and \eqref{exponent:gap} hold. Subsequently, and in the same setting, 
Barsky and Aizenman  established the upper bounds of \eqref{exponent:theta} and \eqref{exponent:volume} (see also \cite{MR912497}), and Nguyen \cite{MR923855}  established the lower bound of \eqref{exponent:gap} by combining the results of \cite{MR762034} with a differential inequality of Durrett and Nguyen \cite{durrett1985thermodynamic}. More recently, Kozma and Nachmias \cite{MR2551766,MR2748397} have established both the upper and lower bounds of \eqref{exponent:intradius} for $\Z^d$ when $d$ is sufficiently large, as well as the appropriate analogue of \eqref{exponent:radius} in the same setting. 
% showed that several other critical exponents also take their mean-field values
 % on $\Z^d$ when the triangle condition is satisfied. 

The triangle condition was established for critical percolation on $\Z^d$ for $d\geq 19$ (as well as spread-out models for $d>6$) in the landmark work of Hara and Slade \cite{MR1043524,MR1283177} using a technique known as the \emph{lace expansion}; their techniques have recently been refined by Fitzner and van der Hofstadt \cite{fitzner2015nearest} to prove that the triangle condition holds for critical percolation on $\Z^d$ for any $d\geq 11$. 
% (These papers also provide the only known proofs that critical percolation on $\Z^d$ has no infini)
It is believed that the triangle condition should hold for critical percolation on $\Z^d$ for every $d > 6$. 

It is also conjectured that the triangle condition holds at $p_c$ for every nonamenable quasi-transitive graph, and it is plausible that it holds for every quasi-transitive graph that has strictly larger than sextic volume growth.
% Kozma \cite{kozma2011percolation} has also shown that the triangle condition holds for the product of two $3$-regular trees, 
Similarly to \cref{conj:pcpu}, most previous results in this direction have been under perturbative hypotheses: for small spectral radius by Schonmann \cite{MR1833805} and for nonamenable graphs of large girth by Nachmias and Peres \cite{MR3005730}.
The only non-trivial example for which the triangle condition has previously been established via a non-perturbative method is due to Kozma \cite{kozma2011percolation}, who proved that it holds for critical percolation on the product of two $3$-regular trees. (In \cite{1712.04911} we show that this example admits a very short analysis using the methods of this paper.)
Schonmann \cite{MR1888869} has also shown, without using the triangle condition, that several mean-field exponents hold 
% the percolation probability, the susceptibility, the cluster volume, and the moment gap 
on every transitive nonamenable planar graph and every infinitely-ended, unimodular, transitive  graph. 

Our next theorem verifies this conjecture in the nonunimodular setting. In particular, it applies to the product of a $k$-regular tree with $\Z^d$ (or any other quasi-transitive graph), for which the result was only previously known for large $k$.
 % (in which case the result follows from the work of Schonmann \cite{MR1833805}).

\begin{thm}
\label{thm:triangle}
Let $G$ be a connected, locally finite graph, and suppose that $\Aut(G)$ has a quasi-transitive nonunimodular subgroup. Then $G$ satisfies the triangle condition at $p_c$. 
% there exists $\eps>0$ such that the triangle condition is satisfied for all $p\leq p_c+\eps$.
% there exists $\eps>0$ such that 
% \[\nabla_{p_c+\eps}(v)
% :=\sum_{x,y\in V}\tau_{p_c}(v,x)\tau_{p_c}(x,y)\tau_{p_c}(y,v) 
% <\infty
% \] 
% for every $v\in V$.
\end{thm}

% Unfortunately some of these implications have only previously been known to hold in the unimodular case. 

As observed by Schonmann \cite{MR1833805}, most aspects of the proofs of \cite{MR762034,MR1127713,MR923855,MR2551766} can easily be generalised to quasi-transitive graphs that satisfy the triangle condition at $p_c$. 
% In the interest of space we do not review these parts of the proofs, and encourage the reader to consult the original papers.
 There are, however, several exceptions to this requiring more serious attention that must be addressed in order to deduce \cref{thm:criticalexponents} from \cref{thm:triangle}. In fact, while it is certainly possible to adapt the original proofs, we are instead able to use the technology developed in the rest of the paper to give alternative, simpler proofs of several of the estimates of \cref{thm:criticalexponents} that are specific to the nonunimodular setting.
  % and are  that we develop in the rest of the paper.

 We refer the reader to  \cite{grimmett2010percolation,MR2239599,heydenreich2015progress} for further background on the triangle condition and its applications, as well as for related work on other models.
% , 

% 

% Our next main theorem is the following.

% In fact, our proof also implies that the triangle condition holds at $p_c+\eps$ for some $\eps>0$.

% \subsection{The triangle condition}

\subsection{The tilted susceptibility and the tiltability transition}
\label{subsec:intro_tilted}
% A key new idea in this 
A central contribution of this
% 
% One of the central ideas of this 
paper is 
% as follows. We 
the
introduction of \emph{tilted} versions of several classical thermodynamic quantities associated to percolation, such as the susceptibility and magnetization. 
 % which 
 These quantities have an additional parameter, which we call $\lambda$, and differ from their classical analogues (which correspond to $\lambda=0$) in that they are weighted in some sense by the modular function to the power $\lambda$. 
 We find that these tilted quantities 
% have natural  analogues
 can often be analysed by similar methods to their classical counterparts but, crucially, have different critical values associated to them.
 This methodology is also central to our analysis of self-avoiding walk in \cite{1709.10515}, and we expect that it will be useful for the analysis of other models in future. 
% We find that these quantities behave similarly to their classical analogues , 

The most important such quantity we introduce is the \emph{tilted susceptibility}. Given $v\in V$ and $\lambda \in \R$, we define the 
\textbf{tilted volume} of a set $W \subseteq V$ to be
\[|W|_{v,\lambda} = \sum_{u\in W} \Delta^\lambda(v,u) \]
and define the
\textbf{tilted susceptibility} $\chi_{p,\lambda}(v)$ to be the expected tilted volume of the cluster at $v$, that is,
\[\chi_{p,\lambda}(v) = \bE_p \left[ |K_v|_{v,\lambda} \right] = \sum_{u\in V} \tau_p(v,u) \Delta^\lambda(v,u).\]
For each $\lambda \in \R$ we also define the associated critical value
\[
\pcl{\lambda} = p_c(G,\Gamma,\lambda) = \sup\big\{ p \in [0,1] : \chi_{p,\lambda}(v) < \infty \big\}.
\]
We also define the \textbf{tiltability threshold}
 \[
p_t =p_t(G,\Gamma) = \sup \left\{ p \in [0,1] : \chi_{p,\lambda}(v) < \infty \text{ for some $\lambda \in \R$}\right\},
\]
and call the set $\{ p \in [0,1] : \chi_{p,\lambda}(v) < \infty \text{ for some $\lambda \in \R$}\}$ the \textbf{tiltable phase}. 
It is easily seen that these definitions do not depend on the choice of $v$. In \cref{sec:meanfieldsusceptibility} we observe that $\pcl{\lambda}=\pcl{1-\lambda}$ for every $\lambda\in \R$ and that $p_t=\pcl{1/2}$. Both statements are easy consequences of the tilted mass-transport principle.

In  \cref{sec:Fekete} we show that the tiltable phase can be analyzed rather straightforwardly using techniques that are traditionally used to analyze subcritical percolation, including in particular variants of the \emph{tree-graph inequality} method of Aizenman and Newman \cite{MR762034}. This allows us to develop a detailed picture of percolation in the tiltable phase, in particular the first and second moment estimates of \cref{prop:tamesecondmoment}, which is then used in the proofs of the main theorems.
 % It is an easy consequence of the tilted mass-transport principle that $\pcl{\lambda}=\pcl{1-\lambda}$ for every $\lambda \in \R$ and that $p_t=p_{1/2}$. 

It follows from the sharpness of the phase transition that $\pcl{0}=\pcl{1}=p_c$ and hence that $p_c \leq p_t$.
Moreover, it follows from Tim\'ar's characterisation of heaviness \eqref{eq:Timarph} that $p_t\leq p_h$.  Thus, \cref{thm:pcph,thm:pcpu} are immediate consequences of the following stronger result. In \cref{subsec:examplestree} we show that both equality and strict inequality between $p_t$ and $p_h$ are possible.

\begin{thm}
\label{thm:pcpt}
Let $G$ be a connected, locally finite graph and suppose that $\Gamma \subseteq \operatorname{Aut}(G)$ is transitive and nonunimodular. Then
$p_c(G) < p_c(G,\Gamma,\lambda) \leq p_t(G,\Gamma)$ for every $\lambda \in (0,1)$. 
\end{thm}

As well as implying \cref{thm:pcph,thm:pcpu}, \cref{thm:pcpt} also implies \cref{thm:triangle}.
% , which stated that the triangle diagram is finite at $p_c$. 
This deduction follows by a very short argument (\cref{lemma:diamondimpliestriangle}), which in the transitive case yields that $\nabla_p \leq \chi_{p,\lambda}^3$ for every $p\in[0,1]$ and $\lambda \in \R$ and hence that $\nabla_p<\infty$ for all $p<p_t$.
Thus, \cref{thm:pcpt} can be viewed as a `master theorem' that easily implies our other main theorems once proven. Further consequences of \cref{thm:pcpt} are explored in \cite[Theorem 2.9]{Hutchcroft2019Hyperbolic} and \cite{1901.10363}.

\subsection{Anisotropic percolation}
\label{subsec:inhomog}

Let $G=(V,E)$ be a connected, locally finite graph, let $\Gamma$ be a quasi-transitive subgroup of $\Aut(G)$, and let $O_1,\ldots,O_k$ be the orbits of the action of $\Gamma$ on the edge set of $G$.  We define \textbf{anisotropic bond percolation} on $G$ by taking a vector of probabilities $\mathbf{p}=(p_1,\ldots,p_k)$, and then letting every edge $e$ in $O_i$ be open with probability $p_i$, independently of all other edges, to obtain a random subgraph $G[\mathbf{p}]$. 
The following theorem extends \cref{thm:pcpu,thm:pcph,thm:triangle,thm:pcpt} to the anisotropic context. 

\begin{thm}
\label{thm:anisotropic}
Let $G$ be a connected, locally finite graph, let $\Gamma \subseteq \Aut(G)$ be quasi-transitive and nonunimodular with edge-orbits $O_1,\ldots,O_k$, and consider anisotropic bond percolation on $G$. Suppose that $\mathbf{p}(t):[0,1]\to [0,1]^k$ is continuous and increasing with $\mathbf{p}(0)=(0,0,\ldots,0)$, $\mathbf{p}(1)=(1,1,\ldots,1)$, and with $\mathbf{p}(t)\in(0,1)^k$ for every $0<t<1$. Then there must exist a positive-length interval $I \subseteq (0,1)$ such that
\begin{enumerate}
	\item
$G[\mathbf{p}(t)]$ has an infinite cluster almost surely for each $t\in I$, 
and 
\item $\chi_{\mathbf{p}(t),1/2}<\infty$ for every $t\in I$.
\end{enumerate}
In particular, $G[\mathbf{p}(t)]$ has infinitely many light infinite clusters almost surely for each $t\in I$.
\end{thm}

\cref{thm:anisotropic} establishes the basic topological features of the phase diagram of anisotropic percolation on $T_k \times \Z^d$, see \cref{fig:phasediagram}. The features of this phase diagram were first suggested by Grimmett and Newman \cite{MR1064560}, who proved that they hold for $T_k\times \Z$ when $k$ is large. They had not previously been established for $T_k\times \Z^d$ for any pair $(k,d)$ with $d\geq 2$. We remark that it is also possible to extend \cref{thm:criticalexponents} to the anisotropic case in a straightforward way. (Note however that the constants that appear will not in general be uniformly bounded along the critical curve.)

\begin{figure}[t!]
\centering
\includegraphics{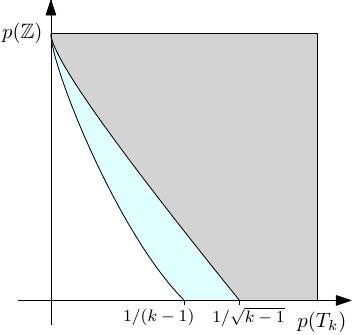}
\hspace{1.5cm}
\includegraphics{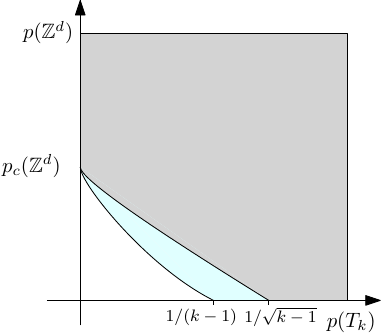}
\caption{
% \small{
The phase diagrams of anisotropic bond percolation on $T_k \times \Z$ and $T_k\times \Z^d$, in which edges of $T_k$ and edges of $\Z^d$ are given different retention probabilities.   In the white region there is no infinite cluster, in the blue region there are infinitely many infinite clusters, and in the grey region there is a unique infinite cluster. The diagrams should be interpreted at a topological level only; we make no claims concerning the shape of the critical curves. It is also known that there is no infinite cluster on the interior of the boundary between the white and blue regions \cite{BLPS99b}, and that there are infinitely many infinite clusters on the interior of the boundary between the blue and grey regions \cite{MR1770624}. The fact that the boundaries separating the phases are curves follows by uniqueness monotonicity \cite{MR1676835}. The fact that the boundary between the non-uniqueness phase and the uniqueness phase meets the horizontal axis at $1/\sqrt{k-1}$ was proven by Grimmett and Newman~\cite{MR1064560}. (Note that $1/\sqrt{k-1}$ is exactly $p_t(T_k,\Gamma)$ when $\Gamma$ is the group of automorphisms fixing an end of $T_k$.)
% }
}
\label{fig:phasediagram}
\end{figure}

% The condition that $0<\mathbf{p}(t)<1$ for all $0<t<1$ is used many of the times that we apply the FKG inequality; the theorem is clearly false without it as can be seen by travelling along the left and the top of the box in either of the phase diagrams in \cref{fig:phasediagram}.

% We remark that for $T_k \times \Z^d$, it follows by the work of Tim\'ar \cite{timar2006percolation} that there does not exist any $\mathbf{p}=(p(\Z^d),p(T_k))$ in the interior of $[0,1]^2$ for which anisotropic percolation has infinitely many heavy clusters almost surely.

% \subsection{}
The changes to the proof of \cref{thm:pcpt} required to prove \cref{thm:anisotropic} are merely notational, and in order to keep the paper readable, we do not include a proof. If desired, the diligent reader will have little trouble obtaining such a proof by replacing all the probabilities $p$ appearing in the remainder of the text with vectors of probabilities $\mathbf{p}$.

The ability to handle the entire phase diagram of anisotropic percolation is a major advantage of our non-perturbative approach. Indeed, the perturbative methods of \cite{bperc96,MR1756965,MR1833805} all rely on the graphs in question being highly nonamenable in some sense. These methods can all be extended to anisotropic percolation, but the relevant criteria now require that the associated anisotropic random walks on the graph are highly nonamenable in the same sense. If we consider, say, anisotropic percolation on $T_k \times \Z^d$ in which $T_k$ edges and $\Z^d$ edges have retention probabilities $p_1$ and $p_2$ respectively, then the associated anisotropic random walks get less and less nonamenable as $p_1/p_2 \to 0$ (e.g.\ in the sense that their spectral radii tend to $1$). Thus, any particular method similar to those of \cite{bperc96,MR1756965,MR1833805} cannot hope to apply to the entire phase diagram, and in particular will fail to show that for every $p_1\in (0,1/(k-1))$ there exists $p_2\in (0,1)$ such that the associated anisotropic percolation model has infinitely many infinite clusters. Similar obstructions apply to e.g.\ the methods of \cite{MR3005730}, in which the girth of the graph is required to be larger than some constant depending on the spectral radius.

\subsection{Organisation and overview}
\begin{itemize}
	\item
In \cref{sec:background} we review the basic background and tools that are used in the remainder of the paper. 

\item
In \cref{sec:meanfieldsusceptibility} we prove the mean-field lower bound on the tilted susceptibility, \cref{prop:tiltedmeanfieldlowerbound}. This result is applied in the derivation of the mean-field lower bound for the tilted magnetization in the following section.
\item
In \cref{sec:AizBar}, we introduce tilted versions of the \emph{ghost field} and of the \emph{magnetization}. 
  We then adapt the methods of Aizenman and Barsky \cite{aizenman1987sharpness}, applying these concepts to show that the tilted volume of a cluster $|K_v|_{v,\lambda}$ cannot have a finite $1/2+\eps$ moment at $\pcl{\lambda}$ whenever $0\leq \lambda <1/2$. This result is later used in the proof of the main theorems in \cref{sec:pc}.
\item
 In \cref{sec:Fekete}, we develop several estimates concerning probabilities of connecting to slabs and expected intersections with slabs, some of which hold for all $p$ and some for all $p<p_t$. In particular, we obtain very precise control of the subcritical regime $p<p_c$. 
\item 
  In \cref{sec:pc}, we apply the analysis of subcritical percolation from \cref{sec:Fekete} to prove that $|K_v|_{v,\lambda}$ has a $1-\eps$ moment at $p_c$ for every $\lambda \in (0,1/2]$ and $\eps>0$ via a bootstrapping procedure.
Together with the result of \cref{sec:AizBar}, this implies that $p_c<\pcl{\lambda}\leq p_t$ for every $\lambda \in (0,1/2]$, completing the proofs of \cref{thm:pcpt,thm:pcph,thm:pcpu}. This part of the paper is both the most technical and the least reliant on classical techniques.
   An important input to this bootstrapping procedure is an \emph{a priori} estimate on connection probabilities that is obtained via Fekete's Lemma, similar to the method used in \cite{Hutchcroft2016944}.  
\item
   In \cref{sec:exponents} we prove our results concerning critical exponents and the triangle condition, namely \cref{thm:triangle,thm:criticalexponents}. 

   \item We conclude with examples, remarks, and open problems in \cref{sec:closing}. 

   \item A glossary of recurring notation is given at the end of the paper. Our use of asymptotic notation is described in detail at the end of \cref{sec:background}.
\end{itemize}

% Let us remark that our analysis of self-avoiding walk in [ref] uses some of the same ideas as this paper but is substantially simpler. The reader may find it helpful to read that paper before this one.
  % The reader familiar with the theory of Euclidean percolation as presented in \cite{grimmett2010percolation} 

% \section{Background}

\paragraph{About constants.}
Let us remark that the proofs of our main theorems are \emph{ineffective}, meaning that they cannot be used, even in principle, to obtain explicit bounds on the constants that appear in e.g.\ \cref{thm:criticalexponents} or to lower bound $|p_u-p_c|$. 
% (Of course, this does not rule out the existence of a different proof that is effective.)
 The principal (but not exclusive) source of this ineffectivity is the proof of \cref{lem:peaksurvival}, which does not give any quantitative estimate on the $o(k)$ error term appearing there.  This contrasts our analysis of self-avoiding walk in \cite{1709.10515}, which is effective.

\section{Background, definitions, and basic tools}
\label{sec:background}

\subsection{The tilted mass-transport principle}
\label{subsec:MTP}

In this section, we define the modular function and prove the tilted mass-transport principle for general quasi-transitive graphs. This does not seem to have previously appeared in the literature, or at least not in the modern form involving a random root. The unimodular quasi-transitive and nonunimodular transitive cases can both be found in \cite{LP:book}, and related material can be found in \cite{BC2011}.

	 Let $G$ be a connected, locally finite graph, and let $\Gamma \subseteq \Aut(G)$ be quasi-transitive. We will always assume that $G$ is infinite. 
	For each vertex $v$, we write $[v]$ for the orbit of $v$ under $\Gamma$. Similarly, for each ordered pair of vertices $(u,v)$, we write $[u,v]$ for the orbit of $(u,v)$ under the diagonal action of $\Gamma$ on $V^2$. 
	% We denote the set of orbits of $\Gamma$ by $\cO$, and we identify this set with an arbitrary set of orbit
	% Let $\cO$ be the set of orbits of $\Gamma$. 
	Let $\cO \subseteq V$ be an arbitrary set of orbit representatives of the action of $\Gamma$ on $V$, meaning that for each $v\in V$ there is a unique $o\in \cO$ such that $[o]=[v]$. We identify $\cO$ with the set of orbits of $\Gamma$.
	Observe that if $\langle X_n \rangle_{n\geq0}$ is the lazy random walk on $G$, started at some vertex $v$, then the process $\langle [X_n] \rangle_{n\geq0}$ is a Markov chain taking values in the finite state space $\cO$, which we call the \textbf{lazy orbit chain}. The lazy orbit chain has  transition probabilities
	\[
	P\left([u],[v]\right) = \frac{1}{2\deg(u)}\left|\left\{e \in E_{u}^\rightarrow : [e^+]=[v] \right\}\right| + \frac{1}{2}\mathbbm{1}\left([u]=[v]\right), 
	\]
	where $E^\rightarrow_{u}$ is the set of oriented edges of $G$ emanating from $u$. (Although our graphs are undirected, it is useful to think of each unoriented edge as corresponding to a pair of oriented edges.) Note that this expression does not depend on the choice of representatives $[u]$ and $[v]$.  
	We also remark that if $\Gamma$ is unimodular then the lazy orbit chain is necessarily reversible, while if $\Gamma$ is nonunimodular then the lazy orbit chain can either be reversible or nonreversible \cite[Exercise 8.33]{LP:book}. 

	Since $G$ is connected, the orbit chain is irreducible. 
	Let $\tilde \mu=\tilde \mu_{G,\Gamma}$ be the unique stationary measure for the lazy orbit chain, and let
	$\mu=\mu_{G,\Gamma}$ be the $\deg([v])^{-1}$-biased measure
	\[
	\mu([v]) = \frac{\tilde \mu([v])\deg([v])^{-1}}{\sum_{o\in \cO} \tilde \mu([o])\deg([o])^{-1}}.
	\]
	% Let $\rho\in \cO$ be sampled from $\tilde \mu$, let $X=\langle X_n \rangle_{n\geq0}$ be a random walk started at $\rho$, and let $\tilde \P$ denote the law of $\rho$ and $X$. Let $\mu$
	 We define the \textbf{modular function} $\Delta:V^2\to (0,\infty)$ by setting
	 \[
\Delta(u,v) = \frac{\mu([v])|\stab_vu|}{\mu([u])|\stab_uv|} = \frac{\tilde \mu([v])\deg(u)|\stab_vu|}{\tilde \mu([u]) \deg(v)|\stab_uv|}.
	 \]
	 Note that this definition clearly agrees with that given in the introduction when $\Gamma$ is transitive, and that $\Gamma$ is unimodular if and only if $\Delta\equiv 1$.
	% \[
	% \Delta(u,v) = \frac{\tilde \P\left([X_0,X_{d(u,v)}] = [u,v]\right)}{\tilde \P\left([X_0,X_{d(u,v)}] = [v,u]\right)}.
	% \]

	% Let $$
	\begin{lem}
	\label{lem:modularwalk}
	Let $G$ be a connected, locally finite graph with at least one edge, and let $\Gamma \subseteq \Aut(G)$ be quasi-transitive. Let $\rho \in \cO$ be sampled from $\tilde \mu$, let $X=\langle X_n \rangle_{n\geq0}$ be a lazy random walk on $G$ with $X_0=\rho$, and let $\tilde \P$ denote the law of $\rho$ and $X$. Then we have that
	\[
	\Delta(u,v)
	=
	\frac{\tilde \P\left([X_0,X_n] = [v,u]\right)}{\tilde \P\left([X_0,X_n] = [u,v]\right)}
	\]
	for every $u,v \in V$ and every $n\geq d(u,v)$.
	\end{lem}

\begin{proof}
For each pair of vertices $v,u$, there are $|\stab_v u|$ vertices $w$ such that $[v,u]=[v,w]$. This leads to the expression
\[
\tilde \P([X_0,X_n]=[v,u]) = \tilde\mu([v])p_n(v,u)|\stab_v u|
\]
for every $n\geq0$, where $p_n(v,u)$ is the probability that the lazy random walk on $G$ started at $v$ is at $u$ after $n$ steps. If $n\geq d(u,v)$ then $p_n(u,v)>0$ and we obtain that
\[
\frac{\tilde \P\left([X_0,X_n] = [v,u]\right)}{\tilde \P\left([X_0,X_n] = [u,v]\right)} = \frac{\tilde\mu([v])p_n(v,u)|\stab_v u|}{\tilde\mu([u])p_n(u,v)|\stab_u v|}.
\]
Using the time-reversal identity $\deg(u) p_n(u,v) = \deg(v) p_n(v,u)$ yields the claimed identity.
% \[
% \frac{\tilde \P\left([X_0,X_n] = [v,u]\right)}{\tilde \P\left([X_0,X_n] = [u,v]\right)} = \frac{\tilde\mu([v])\deg(u)|\stab_v u|}{\tilde\mu([u])\deg(v)|\stab_u v|}
% \]
% as claimed.
\end{proof}

Let $\cO_2$ be a set of orbit representatives of the diagonal action of $\Gamma$ on $V^2$.
It follows from \cref{lem:modularwalk} that for any non-negative $\Gamma$-diagonally invariant function $F:V^2\to[0,\infty]$, we have that
\begin{align}\nonumber\tilde \E \left[ F(X_n,X_0) \right] &= \sum_{[u,v]\in \cO_2 } \tilde \P([X_n,X_0]=[u,v])F(u,v) \\&= \sum_{[u,v]\in \cO_2 } \tilde \P([X_0,X_n]=[u,v])F(u,v)\Delta(u,v) = \tilde \E \left[ F(X_0,X_n) \Delta(X_0,X_n) \right].
\label{eq:Radon-Nikodym}
\end{align}
In other words, $\Delta$ is the Radon-Nikodym derivative of the law of $[X_n,X_0]$ under $\tilde \P$ with respect to the law of $[X_0,X_n]$ under $\tilde \P$ for every $n\geq 0$. 
This can in fact be taken as the \emph{definition} of the modular function, and can then be extended in a natural way to stationary random rooted graphs, see \cite{BC2011,hutchcroft2018locality}. 

\begin{prop}[The tilted mass-transport principle]
Let $G$ be a connected, locally finite graph, let $\Gamma \subseteq \Aut(G)$ be quasi-transitive, and let $\rho\in \cO$ be sampled from $\mu$. Then for every $\Gamma$-diagonally invariant function $F:V^2 \to [0,\infty]$, we have that
\[
\E\left[ \sum_{v\in V} F(\rho,v) \right] = \E\left[ \sum_{v\in V} F(v,\rho)\Delta(\rho,v)\right].
\]
\end{prop}

\begin{proof} Let $X$ be a lazy random walk started at $\rho$, and let $\tilde \E$ denote the expectation with respect to $\rho$ and $X$ when $\rho$ is sampled from the degree-biased measure $\tilde \mu$.  It suffices to consider the case that $F(u,v)$ is supported on pairs $u,v$ with $d(u,v)=k$ for some $k\geq 0$, since every $\Gamma$-diagonally invariant $F$ can be written as a sum of $\Gamma$-diagonally invariant functions of this form. In this case, we can write
	\begin{align*}
	\E \left[\sum_{v\in V} F(\rho,v)\right] &= \E \left[ \sum_{v\in V} \mathbbm{1}\left[d(\rho,v) = k\right] \frac{p_k(\rho,v)}{p_k(\rho,v)} F(\rho,v)\right] = \E\left[ \frac{F(\rho,X_k)}{p_k(\rho,X_k)}\right]\\
	& = \left(\tilde \E \left[\deg(\rho)^{-1} \right]\right)^{-1}\tilde \E\left[ \frac{F(\rho,X_k)}{\deg(\rho)p_k(\rho,X_k)} \right],\end{align*}
	and by time-reversal we have that
	\begin{align*}
	\E \left[\sum_{v\in V} F(\rho,v)\right]
	&= \left(\tilde \E \left[\deg(\rho)^{-1} \right]\right)^{-1}\tilde \E\left[ \frac{F(\rho,X_k)}{\deg(X_k)p_k(X_k,\rho)} \right].
	\end{align*}
	Applying \eqref{eq:Radon-Nikodym} yields that
	\begin{align*}
	\E \left[\sum_{v\in V} F(\rho,v)\right] &= \left(\tilde \E \left[\deg(\rho)^{-1} \right]\right)^{-1}\tilde \E\left[ \frac{F(X_k,\rho)}{\deg(\rho)p_k(\rho,X_k)} \Delta(\rho,X_k) \right],
	\end{align*}
	and applying the manipulations above in reverse yields that
	\begin{align*}
	\E \left[\sum_{v\in V} F(\rho,v)\right] &=  \E\left[ \frac{F(X_k,\rho)}{p_k(\rho,X_k)} \Delta(\rho,X_k) \right] = \E\left[ \sum_{v\in V} F(v,\rho)\Delta(\rho,v)\right]. \qedhere
	\end{align*}
\end{proof}

Finally, we establish the basic symmetries of the modular function.

\begin{lemma}[Symmetries of the modular function]
\label{lem:modularsymmetries}
Let $G$ be a connected, locally finite graph, and let $\Gamma \subseteq \Aut(G)$ be quasi-transitive. Then the modular function $\Delta=\Delta_\Gamma:V^2\to(0,\infty)$ has the following properties.
\begin{enumerate}
\item $\Delta$ is $\Gamma$-diagonally invariant. 
% That is, $\Delta(u,v) = \Delta(w,x)$ for every $u,v,w,x\in V$ with $[u,v]=[w,x]$. 
\item $\Delta$ satisfies the \textbf{cocycle identity}
\[\Delta(u,v)\Delta(v,w) = \Delta(u,w)\]
for every $u,v,w \in V$. In particular, $\Delta(u,u)=1$ and $\Delta(u,v)=\Delta(v,u)^{-1}$ for every $u,v\in V$.
\item $\Delta(x,y)$ is a harmonic function of $y$ when $x$ is fixed. That is,
\[
\Delta(x,y) = \frac{1}{\deg(y)}\sum_{z \sim y} \Delta(x,z)
\]
for every $x,y \in V$, where the sum on the right hand side is taken with multiplicity if there are multiple edges between $y$ and $z$.
\end{enumerate}

\end{lemma}

 % An algebraic proof of the cocycle identity is given in \cite{LP:book}. 

\begin{proof}
Item $1$ is immediate from the definition. Item $2$ follows from \cite[Theorem 8.10]{LP:book}. (The reader may find it an illuminating exercise to prove the cocycle identity probabilistically using \cref{lem:modularwalk}.)
For item 3, observe that for every $v\in V$,
\begin{multline*}
\tilde \P([X_0]=[v])=\tilde \P([X_1]=[v])=\tilde \E \left[\mathbbm{1}([X_0]=[v]) \Delta(X_0,X_1) \right]\\= \left(\frac{1}{2} + \frac{1}{2\deg(v)} \sum_{u \sim v} \Delta(v,u)\right) \tilde \P([X_0]=[v]),
\end{multline*}
where \eqref{eq:Radon-Nikodym} is used in the second equality, and hence that
\[\frac{1}{\deg(v)}\sum_{u\sim v} \Delta(v,u) = 1\]
for every $v\in V$. The claimed harmonicity then follows from the cocycle identity (item 2). 
\end{proof}

Throughout the paper, we will use $\P_p$ and $\E_p$ to denote probabilities and expectations taken with respect to the joint law of the Bernoulli-$p$ bond percolation configuration $G[p]$ and the random root $\rho$ and, later on, the uniform separating layers decomposition.  (On the other hand, we will continue to use $\bP_p$ and $\bE_p$ for  probabilities and expectations taken with respect to the law of $G[p]$ only. This distinction is not very important.)
% , but can be useful to tell at a glance whether a statement concerns the random root vertex or applies to all vertices.)

\subsection{The Harris-FKG and BK inequalities}
\label{subsec:background_BK}
We now briefly recall the main correlation inequalities for Bernoulli percolation, referring the reader to \cite{grimmett2010percolation} for further background.
Let $G$ be a graph, and let $\bP_p$ be the law of Bernoulli bond percolation on $G$.
Given $\omega,\omega' \in \{0,1\}^E$. we write $\omega \leq \omega'$ if $\omega(e) \leq \omega'(e)$ for every $e \in E$. 
A function $f: \{0,1\}^E\to\R$ is \textbf{increasing} (resp.\ decreasing) if $f(\omega')\geq f(\omega)$ for every $\omega',\omega \in \{0,1\}^E$ with $\omega' \geq \omega$ (resp.\ $\omega' \leq \omega$). 
 We say that an \emph{event} $A \subseteq \{0,1\}^E$ is increasing (resp.\ decreasing) if its indicator function is increasing (resp.\ decreasing). The \textbf{Harris-FKG} inequality \cite{harris1960lower} states that
\[
\bP_p(A \cap B) \geq \bP_p(A)\cdot \bP_p(B)
\]
for every $p\in [0,1]$ and every two events $A$ and $B$ such that either $A$ and $B$ are both increasing or $A$ and $B$ are both decreasing. 
% \[
% A \circ B := \{\omega \in \{0,1\}^E :  \}
% \]

Suppose that $A \subseteq \{0,1\}^E$ is an event and $\omega\in A$. We say that a \emph{finite} set $W \subseteq E$ is a \textbf{witness} for the occurrence of $A$ on $\omega$ if an independent percolation configuration $\omega'$ lies in $A$ almost surely given that $\omega'(e)=\omega(e)$ for every $e\in W$.
%  More generally, we say that a (possibly infinite) set $W \subseteq E$ is a witness for for the occurrence of $A$ on $\omega$ if an independent percolation configuration $\omega'$ satisfies
% \[
% \sup\left\{\bP_p( \omega' \in A \mid \omega(e) = \omega(e) \text{ for every $e \in F$}) : F \subseteq W \text{ finite}\right\} = 1.
% \]
The \textbf{disjoint occurrence} $A \circ B$ of $A$ and $B$ is defined to be the set of $\omega \in A \cap B$ 
for which there exist finite witnesses $W_A$ and $W_B$ for the occurrence of $A$ on $\omega$ and $B$ on $\omega$ respectively such that $W_A$ and $W_B$ are disjoint. 
The \textbf{van den Berg and Kesten inequality} (or \textbf{BK inequality}) states that
\[
\bP_p(A \circ B) \leq \bP_p(A) \cdot \bP_p(B)
\]
for every $p\in [0,1]$ and every two increasing events $A$ and $B$. \textbf{Reimer's inequality} \cite{MR1751301} states that the same inequality holds for \emph{arbitrary} events $A$ and $B$.
 Both inequalities are usually stated for events depending on at most finitely many edges, but the finite statement was shown to imply the infinite statement in \cite[Theorem 9]{arratia2015van}. (It is also possible to relax the condition that the witnesses must be finite in various ways, but we shall not need this.)
We shall use Reimer's inequality only in the special case that $A$ and $B$ can each be written as the intersection of an increasing event and a decreasing event, which has a simpler and earlier proof due to van den Berg and Fiebig \cite{MR877608}.  
% We can usually use straightforward limiting arguments to apply the BK inequality and Reimer's inequality in situations where $A$ and $B$ depend on infinitely many edges; see the preprint \cite{arratia2015van} for a more systematic treatment of this issue. 

% \medskip

\noindent\hrulefill\\
\noindent
\emph{\textbf{Notation:} For the duration of \cref{sec:meanfieldsusceptibility,sec:AizBar,sec:Fekete,sec:pc}, we fix a connected, locally finite graph $G$ and a quasi-transitive nonunimodular subgroup $\Gamma \subseteq \Aut(G)$. The symbols $\preceq,\succeq$ and $\asymp$ denote inequalities or equalities that hold to within positive multiplicative constants depending only on $G$ and $\Gamma$. For example, \emph{``$f(n) \asymp g(n)$ for every $n\geq 1$"} means that there exist positive constants $c$ and $C$ such that $cg(n) \leq f(n) \leq Cg(n)$ for every $n\geq 1$. Inequalities or equalities that hold to within positive multiplicative constants depending also on some additional data (such as the choice of $p$) will be denoted using subscripts, e.g. $\preceq_p$. 
Note in particular, that, by quasi-transitivity, the statements  \emph{``$f(v,n) \preceq g(n)$ for every $v\in V$ and $n\geq 1$"} and \emph{``$\sup_v f(v,n) \preceq g(n)$ for every $n\geq 1$"} are equivalent whenever $f$ is $\Gamma$-invariant in the sense that $f(\gamma v,n)=f(v,n)$ for every $v\in V$ and $\gamma \in \Gamma$; We will however use both formulations for the sake of emphasis.
Similar conventions apply to our use of Landau's asymptotic notation, so that $f(n) = O(g(n))$ if and only if $|f(n)| \preceq |g(n)|$. For example, we write $f(n)=o_p(n)$ to mean that $|f(n)| \leq h_p(n)$ for some $h_p(n)$ satisfying $h_p(n)/n\to 0$ as $n\to\infty$, where the function $h_p$ can depend on $G$, $\Gamma$, and $p$ but not any other parameters. Given a random variable $X$ and an event $\mathscr{A}$, we write $\E[X ; \sA]$ for the restricted expectation $\E[X \mathbbm{1}(\sA)]$.
}
% \noindent
\hrulefill

\section{The mean-field lower bound for the tilted susceptibility}
\label{sec:meanfieldsusceptibility}

Recall from the introduction that we define the \textbf{$\Delta$-tilted susceptibility} to be 
\[
\chi_{p,\lambda}(v)=\sum_{x\in V}\tau_p(v,x)\Delta(v,x)^\lambda = \bE_p\left[\sum_{x \in K(v)} \Delta(v,x)^\lambda \right].
\]
We also define
\[\chi_{p,\lambda}=\E\left[ \chi_{p,\lambda}(\rho)\right]=\sum_{v\in \cO} \mu([v]) \chi_{p,\lambda}(v) \qquad \text{ and } \qquad \chi_{p,\lambda}^* = \sup_{v\in V} \chi_{p,\lambda}(v) \]
where $\mu$ is as in \cref{subsec:MTP}. 
Since $\tau_p(u,v)=\tau_p(v,u)$ and $\Delta(u,v)=\Delta(v,u)^{-1}$ for every $u,v \in V$, the tilted mass-transport principle and the cocycle identity imply that
\begin{equation}
\label{eq:chisymmetry}
\chi_{p,\lambda} = \E\left[\sum_{v\in V} \tau_p(\rho,v)\Delta^{\lambda}(\rho,v)\right] =   \E\left[\sum_{v\in V} \tau_p(\rho,v)\Delta^{1-\lambda}(\rho,v)\right]=\chi_{p,1-\lambda}
\end{equation}
for every $p\in [0,1]$ and $\lambda \in \R$, and hence that $\pcl{\lambda}=\pcl{1-\lambda}$ for every $\lambda \in \R$. Furthermore, being a sum of exponentials, $\chi_{p,\lambda}$ is a convex function of $\lambda$ for fixed $p$. Together with the $\lambda \mapsto 1-\lambda$ symmetry \eqref{eq:chisymmetry}, this implies that $\chi_{p,\lambda}$ is a decreasing function of $\lambda$ on $(-\infty,1/2]$ and an increasing function of $\lambda$ on $[1/2,\infty)$. This in turn implies that $\pcl{\lambda}$ is increasing  on $(-\infty,1/2]$ and decreasing on $[1/2,\infty)$, and in particular that $p_t=\pcl{1/2}$. 

% Let us also note, although we will not use this fact, that $\chi_{p,\lambda}$ is a differentiable function of $\lambda$ on $\{p\in [0,1], \lambda \in \R, p<\pcl{\lambda}\}$ with
% \[
% \frac{\partial}{\partial \lambda} \chi_{p,\lambda} = \begin{cases}
% -\frac{1}{2}\sum_{\rho \in \cO}\mu(\rho) \sum_{x\in V} \left|\log \Delta(\rho,x)\left[\Delta^\lambda(\rho,x)-\Delta^{1-\lambda}(\rho,x)\right]\right| \tau_p(\rho,x) & \lambda <1/2\\
% \phantom{-} 0 & \lambda = 1/2\\
% \phantom{-}\frac{1}{2}\sum_{\rho \in \cO}\mu(\rho) \sum_{x\in V} \left|\log \Delta(\rho,x)\left[\Delta^\lambda(\rho,x)-\Delta^{1-\lambda}(\rho,x)\right]\right| \tau_p(\rho,x) & \lambda >1/2,
% \end{cases}
% \]
% Moreover, it follows from 
% so that $\chi_{p,\lambda}$ is strictly increasing in $\lambda$ for $\lambda <1/2$ and strictly decreasing in $\lambda$ for $\lambda>1/2$.

% We now prove the mean-field lower bound for the tilted susceptibility. 
% \begin{prop}
% \label{prop:gammacont}
% For each $\lambda \in \R$, the set $\{p \in [0,1] : \chi_{p,\lambda}<\infty \}$ is open in $[0,1]$. In particular, the tame phase is an open subset of $[0,1]$. 
% \end{prop}

\medskip

The main purpose of this section is to show that 
  Aizenman and Newman's \cite{MR762034} proof of the mean-field lower bound for the susceptibility 
% \chi_{p_c-\eps} \geq C \eps^{-1}
% \]
 goes through \emph{mutatis mutandis} for the tilted susceptibility. This yields the following proposition, which is both an essential part of the proof of our main theorems and an interesting result in its own right. A complementary upper bound is proven in \cref{thm:tiltedsusceptibilityexponent} under the assumption that $p_c(\lambda)<p_t$.
% of the analogous result for the classical susceptibility \cite{MR762034} (see also \cite[Proposition 10.28]{grimmett2010percolation}). 

\begin{prop}[Mean-field lower bound for the tilted susceptibility]
\label{prop:tiltedmeanfieldlowerbound}
% the tilted susceptibility $\chi_{p,\$
% \[\chi_{\pcl{\lambda},\lambda}(v)=\infty\] for every $v\in V$ and $\lambda\in \R$.
 % Moreover, 
For each $\lambda \in \R$ there exists a positive constant $c_\lambda$ such that
\[
\chi_{\pcl{\lambda}-\eps,\lambda}(v) \geq c_\lambda \eps^{-1}
\]
for every $v\in V$ and $0 <\eps < \pcl{\lambda}$. In particular, $\chi_{\pcl{\lambda},\lambda}(v)=\infty$ for every $\lambda \in \R$ and $v\in V$.
\end{prop}

\begin{proof}[Proof of \cref{prop:tiltedmeanfieldlowerbound}]
The FKG inequality and the cocycle identity imply that 
\[\chi_{p,\lambda}(u) \geq \tau_p(u,v)\Delta^\lambda(u,v) \chi_{p,\lambda}(v)\]
for every $u,v\in V$, $p\in [0,1]$ and $\lambda\in \R$, 
and since $\Gamma$ is quasi-transitive it therefore suffices to prove that for each $\lambda\in \R$ there exists a constant $c_\lambda'$ such that
\[
\chi^*_{\pcl{\lambda}-\eps,\lambda} \geq c_\lambda' \eps^{-1}
\]
for every $0<\eps<\pcl{\lambda} $.

Let $0\leq p<1$ and let $0< \eps< 1-p$. 
Suppose that each edge of $G$ is \textbf{open} with probability $p$ and, independently, \textbf{blue} with probability $\eps/(1-p)$. The subgraph spanned by the open-or-blue edges has the same distribution as $G[p+\eps]$.
 % exactly $p+\eps$ percolation.
  Let $\tilde \tau_i(v,u)$ be the probability of the event $\sT_i(u,v)$ that $u$ and $v$ are connected by a simple open-or-blue path containing exactly $i$ blue edges,
  and let 
\[\tilde \chi_i(v) =  \sum_{u\in V} \tilde \tau_i(v,u) \Delta^{\lambda}(v,u)\]
so that 
\[\tau_{p+\eps}(v,u) \leq \sum_{i \geq 0} \tilde \tau_i(v,u) \quad \text{ and hence } \quad 
 \chi_{p+\eps,\lambda}(v) \leq \sum_{i \geq 0} \tilde \chi_i(v).\]
% It follows from the BK inequality that
Let $E^\rightarrow_w$ be the set of oriented edges of $G$ emanating from the vertex $w$. Considering the possible locations for the $(i+1)$th blue edge and applying the BK inequality (which holds for any product measure)  yields that
\begin{align*} \tilde \tau_{i+1}(v,u) &\leq 
\sum_{w \in V} \sum_{e \in E^\rightarrow_w} \bP_p\left(\sT_i(v,w) \circ \{\text{$e$ blue}\}\circ \sT_0(e^+,u)\right)\leq \frac{\eps}{1-p} \sum_{w\in V} \tilde \tau_i(v,w) \sum_{e \in E^\rightarrow_w} \tau_p(e^+,u). \end{align*}
Applying the cocycle identity (\cref{lem:modularsymmetries}), it follows that
\begin{align*}\tilde \chi_{i+1}(v) &\leq \frac{\eps}{1-p} \sum_{w \in V} \tilde \tau_i(v,w) \sum_{e \in E^\rightarrow_w} \sum_{u \in V} \tau_p(e^+,u) \Delta^{\lambda}(v,u)\\
&=  \frac{\eps}{1-p} \sum_{w\in V} \tilde \tau_i(v,w) \Delta^{\lambda}(v,w) \sum_{e \in E^\rightarrow_w}\Delta^\lambda(w,e^+) \chi_{p,\lambda}(e^+),
% & = \frac{\eps}{1-p}  \tilde \chi_i \chi_{p,\lambda} \sum_{z\sim 0} \Delta(0,z)^{\lambda}
\end{align*}
and hence by induction that 
\[\tilde \chi_i(v) \leq \left(\frac{C_\lambda \eps}{1-p}\right)^i  \left(\chi^*_{p,\lambda}\right)^{i+1}\]
for every $i \geq 0$, where $C_\lambda = \max_{v\in V}\sum_{e \in E^\rightarrow_v} \Delta^\lambda(v,e^+)$.
Summing over $i$, we obtain that if $\chi^*_{p,\lambda}<\infty$ then
\begin{align}
\label{eq:meanfieldproof}
\chi_{p+\eps,\lambda}(v) \leq \frac{(1-p)\chi^*_{p,\lambda}}{1-p - \eps C_\lambda\chi^*_{p,\lambda}} < \infty
 % \intertext{for all } 
 \quad \text{ for all } \quad
\eps^{-1} > (1-p)^{-1} C_\lambda \chi^*_{p,\lambda}.
\end{align}
This immediately implies that $\chi^*_{\pcl{\lambda},\lambda}=\infty$. By rearranging the inequality for $\eps$ on the right hand side of \eqref{eq:meanfieldproof} we obtain that
\begin{equation}
\label{eq:meanfieldproof2}
\chi^*_{\pcl{\lambda}-\eps,\lambda} \geq \frac{(1-\pcl{\lambda}+\eps)}{C_\lambda}\eps^{-1}
\end{equation}
for every $\lambda\in \R$ and $0<\eps\leq \pcl{\lambda}$. This clearly implies the claim. \qedhere
\end{proof}

% \begin{remark}
% Since $\chi_{0,\lambda}(v)=1$ for every $v\in V$ and $\lambda \in \R$, taking $\eps=\pcl{\lambda}$ in \eqref{eq:meanfieldproof2} implies that
% $\pcl{\lambda} \geq \big( \max_{v\in V}\sum_{e \in E^\rightarrow_w} \Delta(v,e^+)^\lambda\big)^{-1}$ for every $\lambda \in \R$. 
% \end{remark}

\section{The tilted magnetization and tilted ghost field}

% \subsection{Tilted ghost fields, tilted magnetization}
\label{sec:AizBar}

The goal of this section is to prove the following proposition. In \cref{sec:pc}, this proposition will be used to prove via contradiction that $p_c<p_t$.
 % on the magnetization.
% the mean-field lower bound.

\begin{prop}
\label{prop:tiledAizBar}
% Let $G$ be a connected, locally finite graph, and let $\Gamma \subseteq \Aut(G)$ be quasi-transitive and nonunimodular. Then
Let $\lambda \in [0,1/2)$. Then 
\[
\bE_{\pcl{\lambda}}\left[ |K_v|_{v,\lambda}^{(1+\eps)/2} \right] = \infty
\]
for every $\eps>0$ and $v\in V$. 
\end{prop}

\begin{remark}
This proposition is a tilted version of Aizenman and Barsky's \cite{aizenman1987sharpness} mean-field lower bound
\begin{equation}
\label{eq:classicalAizBar}
\bP_{p_c}(|K_v|\geq n) \succeq n^{-1/2}\end{equation}
(i.e., the lower bound of \eqref{exponent:volume}), which holds on any infinite, connected, locally finite, quasi-transitive graph.
 However, compared with the mean-field lower bound on the susceptibility (\cref{prop:tiltedmeanfieldlowerbound}), a rather more substantial modification to the classical proof is required to prove \cref{prop:tiledAizBar}. Moreover, the result we obtain is weaker than that available in the case $\lambda=0$. Finally, while \eqref{eq:classicalAizBar} is sharp, we do not expect \cref{prop:tiledAizBar} to be sharp in any case other than $\lambda=0$. Indeed, if $\lambda \in (0,1/2)$ and $\pcl{\lambda} <p_t$ (which we conjecture always holds), then it follows from the proof of \cref{cor:clustergrowth} that $\bP_{\pcl{\lambda}}(|K_v|_{v,\lambda}=\infty)>0$.
 \end{remark}
 % , where it is proven that

% This is probably a deficiency of our proof rather than reflecting the reality of the situation.

% \subsection{The tilted ghost field}
\medskip

We now begin working towards \cref{prop:tiledAizBar}. 
 We begin by introducing a family of random sets $\{g_v : v \in V\}$ that can be thought of as a tilted version of the \emph{ghost field} from \cite{aizenman1987sharpness}. 
 % Let $p\in [0,1]$ and consider Bernoulli site percolation $\omega_p$ on $G$.
  Let 
\[\left\{\left(N_v(t)\right)_{t\geq 0}: v \in V\right\}\] be a collection of independent, intensity $1$ Poisson processes indexed by $V$ and independent of $G[p]$. 
For each $\lambda \in \R$, $h>0$, and each vertex $v\in V$, let $\mathcal{G}_v=\mathcal{G}_{v,\lambda,h}: V \to \N$ be the random function
\[
\mathcal{G}_v(u) = N_u\left(h\Delta^\lambda(v,u)\right).
\]
Thus,  for each $v\in V$, $\mathcal{G}_v:V\to \N$ is a Poisson point process on $V$ with intensity
$h \Delta^\lambda(v,u)$ 
at each vertex $u\in V$.
We say that $u$ is $v$\textbf{-green} if $\cG_v(u) \geq 1$, and write $g_v=g_{v,\lambda,h}$ for the set of $v$-green vertices. 
% We think of $\cG_v(u)$ as describing the number of \textbf{$v$-green particles} at the site $u$. 
We write $\bP_{p,\lambda,h}$ for the joint law of $G[p]$ and $\{\mathcal{G}_v=\cG_{v,\lambda,h}:v\in V\}$, and write $\P_{p,\lambda,h}$ for the joint law of $G[p]$, $\{\cG_v=\cG_{v,\lambda,h}:v\in V\},$ and the random root vertex $\rho$.
% Given $u,v\in V$, we say that $u$ is $v$\textbf{-green} if $T_v \leq \Delta^\lambda(v,u)$, and we write $g_v$ for the set of $v$-green vertices. 
% reader familiar with the proof of \cite{aizenman1987sharpness} may wish to view 

  % exponential random va Poisson point process on $V$ with intensity measure $\lambda_{v,\lambda,h}(u)= h \Delta^\lambda(v,u)$, independently of percolation. 
  % We call $g_v$ the \textbf{tilted ghost field}, and call vertices of $V$ containing a point of $g_v$ \textbf{green}.

   For each $v\in V$, $p\in [0,1]$, $\lambda \in \R$ and $h>0$ we define the \textbf{tilted magnetization} to be
\[
M_{p,\lambda,h}(v) = \bP_{p,\lambda,h}(v \leftrightarrow g_v) = \bE_p\left[1- 
\exp\left( -h |K_v|_{v,\lambda} \right)
\right].
\]
Note that, as a function of $h$, $1-M_{p,\lambda,h}(v)$ is just the Laplace transform of the law of $|K_{v}|_{v,\lambda}$ under $\bP_{p}$. 
We also define
\[
\chi_{p,\lambda,h}(v) = \bE_{p,\lambda,h} \left[ |K_v|_{v,\lambda} \,;\, v \nleftrightarrow g_v \right] =
\bE_p\left[ 
|K_v|_{v,\lambda} \exp\left( -h |K_v|_{v,\lambda} \right)
\right],
\]
and write  $M_{p,\lambda,h}$ and $\chi_{p,\lambda,h}$ for the averaged quantities 
\[ M_{p,\lambda,h}=\E\left[M_{p,\lambda,h}(\rho)\right] \qquad \text{ and } \qquad  \chi_{p,\lambda,h}=\E\left[\chi_{p,\lambda,h}(\rho) \right].\]
Note also that the trivial inequalities
\begin{equation}
\label{eq:trivial}
M_{p,\lambda,h}(v) \geq 1-e^{-h} \qquad \text{ and } \qquad \chi_{p,\lambda,h}(v) \leq \max_{x \geq 0} xe^{-hx} = \frac{1}{eh}
\end{equation}
hold for every $v\in V$, $p\in [0,1]$, $\lambda \in \R$, and $h >0$.

We also define the \textbf{truncated tilted susceptibility}
\begin{align}
\chi^f_{p,\lambda}(v) =  \bE_p\left[ |K_v|_{v,\lambda} \,;\, |K_v|_{v,\lambda} < \infty \right] \qquad &\text{ and } \qquad  \chi^f_{p,\lambda} =  \E\left[ \chi^f_{p,\lambda}(\rho) \right]\end{align}
% \end{align}
% \intertext{
and observe that, by monotone convergence,
\begin{align}
\lim_{h\downarrow 0} \chi_{p,\lambda,h}(v) = \chi^f_{p,\lambda}(v) \qquad &\text{ and } \qquad \lim_{h\downarrow 0} \chi_{p,\lambda,h} = \chi^f_{p,\lambda}
 % \\
% \intertext{and}
% \lim_{h\downarrow 0} M_{p,\lambda,h}(v) &= \P_p(|K_v|_{v,\lambda} =\infty)
\end{align}
for every $v\in V$, $p\in [0,1]$ and $\lambda \in \R$.
% , for $p,\lambda$ fixed, $M_{p,\lambda,h}$ is an increasing, differentiable function of $h\geq0$, and that

% \begin{lemma}
% Let $p\in [0,1]$, $\lambda \in \R$, and suppose that
% \[
% \E_{p}\left[ |K_v|_{v,\lambda}^{\lambda} \right] <\infty
% \]
% for some  $0<\lambda <1$. Then 
% \[
% M(p,\lambda,h) \leq C h^{\lambda-\delta}
% \]
% \end{lemma}

 It follows by dominated convergence that, for fixed values of $p\in[0,1]$ and $\lambda\in \R$, the magnetization $M_h =M_{p,\lambda,h}$ is a differentiable function of $h$ for $h>0$ and that
\begin{equation}
\label{eq:Mdiff}
\frac{\partial}{\partial h} M_{p,\lambda,h} = \chi_{p,\lambda,h}
\end{equation}
for every $h>0$. This allows us to interpret the following lemma as a differential inequality.

\begin{lemma}
\label{lem:diffineq}
Let $0 \leq \lambda<1/2$, and suppose that $\bE_{\pcl{\lambda}}\left[ |K_v|_{v,\lambda}^{(1+\eps)/2} \right] <\infty$
for some  $\eps>0$ and some (and hence every) $v\in V$. Then for every $0<\delta \leq \min\{\eps,(1-2\lambda)/\lambda\}$ there exists 
% We claim that in this case there exists 
a constant $C_\delta$ such that
\begin{equation}
\label{eq:differentialineq}
M_{\pcl{\lambda},\lambda,h} \leq C_\delta M^2_{\pcl{\lambda},\lambda,h} + h \chi_{\pcl{\lambda},\lambda,h} + C_\delta h^{1+\delta} \chi_{\pcl{\lambda},\lambda,h}
\end{equation}
for every $h>0$.
\end{lemma}

Before proving \cref{lem:diffineq}, let us use it to prove \cref{prop:tiledAizBar}.
Before we start, let us recall that if $X$ is a non-negative random variable and $0<a\leq 1$ is such that $\E X^a <\infty$, then, since $1-e^{-x} \leq \min\{x,1\}$ for every $x\geq 0$, we have that
\begin{equation}
\label{eq:LaplaceBound}
\E \left[1- e^{-hX}\right] \leq \E\left[ \min\{ h X, 1\} \right]  \leq 
\E\left[ \min\{ h X, 1\}^a \right] \leq h^a \E X^a
% \leq h \E \left[ X \mathbbm{1}(X \leq h^{-1}) \right] + \P(X \geq h^{-1}) \\
% \leq h \E \left[ X^a h^{-1+a} \right] + h^a \E X^a
\end{equation}
for every $h>0$.

\begin{proof}[Proof of \cref{prop:tiledAizBar} given \cref{lem:diffineq}]
Let $0<\lambda<1/2$, suppose for contradiction that 
\[ \E_{\pcl{\lambda}}\left[ |K_\rho|_{\rho,\lambda}^{(1+\eps)/2}\right]<\infty\]
for some $\eps>0$, and let $n$ be such that $2^{1/n}< 1+\min\{\eps,(1-2\lambda)/\lambda\}$. 
In particular, this implies that 
\begin{equation}
\label{eq:AizBarprooftheta0}\lim_{h \downarrow 0} M_{\pcl{\lambda},\lambda,h}(v)=1-\bP_{\pcl{\lambda}}(|K_v|_{v,\lambda}<\infty)=0\end{equation}
for every $v\in V$, and hence that 
\begin{equation}
\label{eq:chifisinfinite}
\chi_{\pcl{\lambda},\lambda}^f(v)=\chi_{\pcl{\lambda},\lambda}(v)=\infty
\end{equation}
 for every $v\in V$ by \cref{prop:tiltedmeanfieldlowerbound}. 
Write $\phi(h)=M_{\pcl{\lambda},\lambda,h}$ and let $\psi(t)$ be the inverse of $\phi$. 
 The trivial magnetization lower bound \eqref{eq:trivial} yields that
\begin{equation}
\label{eq:psiupper}
\psi(t) \leq -\log(1-t) \leq (2\log 2) t \qquad \text{ for every $0<t \leq 1/2$}.
\end{equation}
% for $0<t\leq 1/2$. 
More is true, however: 
Since $\phi$ is differentiable and $\phi(0)=0$, we have by the mean value theorem that for every $0<h<1$ there exists $0<s<h$ such that
\begin{equation}
\frac{1}{h}\phi(h) = \phi'(s) = \chi_{\pcl{\lambda},\lambda,s},
\end{equation}
and we deduce that
\begin{equation}
\label{eq:psiovert}
\limsup_{t\downarrow 0} \frac{1}{t} \psi(t) = \limsup_{h\downarrow 0} \frac{h}{\phi(h)}  \leq \limsup_{h\downarrow 0} \frac{1}{\chi_{\pcl{\lambda},\lambda,h}}=\frac{1}{\chi^f_{\pcl{\lambda},\lambda}}=0, \end{equation}
where we have applied \eqref{eq:chifisinfinite} in the final equality.
On the other hand, the inequality \eqref{eq:LaplaceBound} implies that 
\begin{equation}\phi(h) \leq h^{(1+\eps)/2}\E_{\pcl{\lambda}}\left[ |K_\rho|_{\rho,\lambda}^{(1+\eps)/2}\right]\end{equation}
for every $h>0$ and hence that
\begin{equation}
\label{eq:psilower}
\psi(t) \geq \E_{\pcl{\lambda}}\left[ |K_\rho|_{\rho,\lambda}^{(1+\eps)/2}\right]^{-2/(1+\eps)} t^{2/(1+\eps)} =: c_1 t^{2/(1+\eps)} \geq  c_1 t^{2^{1-1/n}} \qquad \text{for every $t\in(0,1)$}.
\end{equation}

 By \cref{lem:diffineq}, there exists $C_2<\infty$ such that
\begin{equation}
\label{eq:differentialineq}
\phi(h) \leq C_2 \phi^2(h) + h \phi'(h) + C_2 h^{2^{1/n}} \phi'(h).
\end{equation}
for every $h>0$, which is equivalent to the inequality
\begin{align}
t \leq C_2 t^2 + \frac{\psi(t)}{\psi'(t)} + \frac{C_2 \psi^{2^{1/n}}(t)}{\psi'(t)}.
\end{align}
Multiplying both sides by $\psi'(t)/t^2$ and rearranging, we obtain that
% $0 \leq t\leq 1$.
% 
% 
% 
% which rearranges to give
\begin{equation}
% \label{eq:diffineqrearranged}
\left(\frac{1}{t}\psi(t)\right)' = \frac{1}{t}\psi'(t)-\frac{1}{t^2}\psi(t) \leq C_2\psi'(t) +\frac{C_2}{t^2}\psi^{2^{1/n}}(t).
\end{equation}
for every $t\in(0,1)$. Since $\psi'(t)=1/\chi_{\pcl{\lambda},\lambda,\psi(t)}$ is an increasing function of $t$, it follows that there exists a constant $C_3<\infty$ such that
\begin{equation}
\label{eq:diffineqrearranged}
\left(\frac{1}{t}\psi(t)\right)' \leq C_3\left(1 +\frac{1}{t^2}\psi^{2^{1/n}}(t)\right) \qquad \text{for every $0<t\leq 1/2$.}
\end{equation}

To proceed, we will show that the four statements
 \eqref{eq:psiupper}, \eqref{eq:psiovert}, \eqref{eq:psilower}, and \eqref{eq:diffineqrearranged} cannot all hold.  
To do this, we will apply  \eqref{eq:psiupper}, \eqref{eq:psiovert}, and \eqref{eq:diffineqrearranged} to prove by induction on $k$ that for each $0 \leq k \leq n$ there exists a constant $C'_k$ such that
\begin{equation}
\label{eq:psiinduction}
\psi(t) \leq C'_k t^{2^{k/n}}
\end{equation}
for every $0<t\leq 1/2$: The case $k=n$ will then contradict the lower bound of \eqref{eq:psilower}.
% while follows from \eqref{eq:LaplaceBound} and the assumption that $2^{1/n}< 1+\min\{\eps,(1-2\lambda)/\lambda\}$

The base case $k=0$ follows from \eqref{eq:psiupper}. 
Suppose that $0\leq k< n$  and that \eqref{eq:psiinduction} holds for $k$. 
Applying \eqref{eq:psiovert}, substituting the induction hypothesis into the right hand side of the differential inequality \eqref{eq:diffineqrearranged} and integrating 
 % the differential inequality \eqref{eq:diffineqrearranged}
  yields that
\begin{multline}
\frac{1}{t} \psi(t) = \frac{1}{t} \psi(t) - \lim_{s\downarrow 0}\frac{1}{s} \psi(s)  
% &\leq C' t + C'\int_0^t \frac{1}{s^2} \psi^{2^{1/n}}(s)\dif s\\
% &\leq C' t + C'\cdot C_{k-1} \int_0^t \frac{1}{s^2}\left(s^{2^{k/n}}\right)^{2^{1/n}}\dif s\\
= \int_0^t \left(\frac{1}{s}\psi(s)\right)' \dif s
\leq C_3 t +  C_3 \int_0^t s^{-2} (C'_k s^{2^{k/n}})^{2^{1/n}} \dif s
\\
\leq C_3 t + \frac{C_3(C'_k)^{2^{1/n}} }{2^{(k+1)/n}-1} t^{2^{(k+1)/n}-1}
% \frac{C_3 \cdot C'_{k}}{2^{k/n}-1} t^{2^{(k+1)/n}-1}, 
% \leq C_k t^{2^{k/n}},
\end{multline}
for every $0 < t \leq 1/2$, 
and hence that
\begin{equation}\psi(t)\leq C'_{k+1} t^{2^{(k+1)/n}} \quad \text{ where }  \quad C'_{k+1} = C_3\left(1+\frac{(C'_{k})^{2^{1/n}}}{2^{(k+1)/n}-1}\right)\end{equation}
for every $0<t\leq 1/2$.
% $
% C_k = 1+C \cdot C_{k-1}/(2^{k/n}-1)$.
This completes the induction, which in turn yields the desired contradiction. 
% , and shows that
% \[\psi(t) \leq C_n t^2. \]
 % $\delta>0$ such that
\end{proof}

It remains only to prove \cref{lem:diffineq}. 
% We follow the basic outline of the proof of as presented in \cite{grimmett2010percolation}. 
We begin with a preliminary lemma concerning disjoint occurrences with respect to both the percolation configuration and the ghost field.

Given $u,v \in V$, we write $\{u \leftrightarrow g_v \} \circ \{u \leftrightarrow g_v \}$ for the event that either
\begin{enumerate}
\item There exist two distinct $v$-green vertices $w,z \in g_v$ and two edge-disjoint open paths connecting $u$ to $w$ and $u$ to $z$ (including the case that one of $w,z$ is equal to $u$ and the open path from $u$ to this vertex is the empty path), or 
\item There exists a vertex $w$ with $\mathcal{G}_v(w) \geq 2$, and two edge-disjoint open paths connecting $u$ to $w$ (including the case that $w=u$ and both paths are the empty path).
	\end{enumerate}
	That is, we require there to be two (necessarily finite) witnesses for $\{u \leftrightarrow g_v \}$ that are disjoint with respect to both the percolation configuration \emph{and} the ghost field. 

	\begin{lemma}
	\label{lem:GreenBK}
	 The estimate
	\[\bP_{p,\lambda,h}\left(\{u \leftrightarrow g_v \} \circ \{u \leftrightarrow g_v \}\right) \leq \bP_{p,\lambda,h}(u \leftrightarrow g_v)^2\]
	holds for all $u,v\in V$, $p\in [0,1]$, $\lambda \in \R$ and $h>0$.
	\end{lemma}
\begin{proof}
This follows from the standard BK inequality by approximating the Poisson random variables $\mathcal{G}_v(u)$ by Binomial random variables. 
\end{proof}

\begin{proof}[Proof of \cref{lem:diffineq}] Let $\mathcal{G}_v(K_v)=\sum_{u\in K_v} \mathcal{G}_v(u)$. Write
\begin{equation}
M_{p,\lambda,h}(v) = \bP_{p,\lambda,h}\left(\mathcal{G}_v(K_v)=1\right) + \bP_{p,\lambda,h}\left(\mathcal{G}_v(K_v) \geq 2\right)
\end{equation}
and
\begin{multline}
\bP_{p,\lambda,h}\left(\mathcal{G}_v(K_v) \geq 2\right) =\\ \bP_{p,\lambda,h}\left(\{v \leftrightarrow g_v\}\circ \{v\leftrightarrow g_v\}\right) + 
\bP_{p,\lambda,h}\left( \left\{\mathcal{G}_v(K_v) \geq 2 \right\} \setminus\left[\{v \leftrightarrow g_v\}\circ \{v\leftrightarrow g_v\}\right]\right).
\end{multline}

Conditional on $K_v$, the random variable $\mathcal{G}_v(K_v)$ has a Poisson distribution with parameter $h |K_v|_{v,\lambda}$. It follows that $\bP_{p,\lambda,h}\left( \mathcal{G}_v(K_v) =1\right) = h \chi_{p,\lambda,h}(v)$ for every $v\in V$ and hence that
\begin{equation}
\P_{p,\lambda,h}\left( \mathcal{G}_\rho(K_\rho) =1\right) = h \chi_{p,\lambda,h}.
\end{equation}
 Meanwhile, \cref{lem:GreenBK} implies that
\begin{equation}
\bP_{p,\lambda,h}\left(\{v\leftrightarrow g_v\}\circ\{v\leftrightarrow g_v\}\right) \leq M_{p,\lambda,h}^2(v)
\end{equation}
and hence that
\begin{equation}
\P_{p,\lambda,h}\left(\{\rho\leftrightarrow g_\rho\}\circ\{v\leftrightarrow g_\rho\}\right) \leq  \left[\inf_{v\in V} \P([\rho]=[v])\right]^{-1} M_{p,\lambda,h}^2.
\end{equation}
Thus, to prove \cref{lem:diffineq} it remains to show only that there exists a constant $C$ such that
\begin{equation}
\P_{p,\lambda,h}\left( \left\{\mathcal{G}_\rho(K_\rho) \geq 2 \right\} \setminus\left[\{\rho \leftrightarrow g_\rho\}\circ \{\rho\leftrightarrow g_\rho\}\right]\right) \leq C h^{1+\delta} \chi_{p,\lambda,h}.
\label{eq:desiredghost}
\end{equation}

If the event $\{\mathcal{G}_v(K_v) \geq 2\}$ occurs but $\{v \leftrightarrow g_v\}\circ \{v\leftrightarrow g_v\}$ does not, then 
it follows from Menger's Theorem that 
there exist vertices $u$ and $w$ of $G$ and an edge $e$ of $G$ with endpoints $u$ and $w$ such that the following hold:
\begin{enumerate}
	\item $e$ is open,
	\item if $e$ is made to be closed then $v$ remains connected to $u$ but is no longer connected to $g_v$, and
	\item the event $\{w \leftrightarrow g_v\} \circ \{w \leftrightarrow g_v\}$ occurs. 
\end{enumerate}
Let $\sA(v,u,w,e)$ be the event that these three conditions hold. Fix $v,u,w,e$, and let $G[p]_e$ be obtained from $G[p]$ by making the edge $e$ closed. (In particular if $e$ is closed in $G[p]$ then $G[p]_e=G[p]$.)
% \[G[p]_e(f)
% =\begin{cases}
% \omega(f) : f \neq e\\
% 0 : f = e,
% \end{cases}
% \]
Let $K'_v$ be the connected component of $v$ in $G[p]_e$. Then we have that
\begin{multline}
\bP_{p,\lambda,h}\bigl(\sA(v,u,w,e) \mid K'_v\bigr) =\\ p\, \mathbbm{1}\bigl(u \in K'_v,\, w \notin K'_v\bigr) \cdot \bP_{p,\lambda,h}\bigl(K'_v\cap g_v = \emptyset \mid K'_v\bigr) \cdot \bP_{p,\lambda,h}\left(\{ w \leftrightarrow g_v \} \circ \{w \leftrightarrow g_v\} \text{ off } K'_v \mid K'_v\right)
\end{multline}
and hence that
\begin{multline}
\bP_{p,\lambda,h}(\sA(v,u,w,e) \mid K'_v) \leq \\ p\, \mathbbm{1}\bigl(u \in K'_v\bigr) \cdot \bP_{p,\lambda,h}\bigl(K'_v\cap g_v = \emptyset \mid K'_v\bigr) \cdot \bP_{p,\lambda,h}\left(\{ w \leftrightarrow g_v \} \circ \{w \leftrightarrow g_v\} \right).
\end{multline}
Thus, taking expectations over $K'_v$ we obtain that
\begin{equation}
\bP_{p,\lambda,h}(\sA(v,u,w,e))\leq p \, \bP_{p,\lambda,h} (u \in K'_v \text{ and } K'_v \cap g_v =\emptyset) \cdot \bP_{p,\lambda,h}\left(\{ w \leftrightarrow g_v \} \circ \{w \leftrightarrow g_v\} \right).
\end{equation}
On the other hand,
\begin{equation}
\bP_{p,\lambda,h}(v \leftrightarrow u, v\nleftrightarrow g_v) \geq (1-p)\,\bP_{p,\lambda,h} (u \in K'_v \text{ and } K'_v \cap g_v =\emptyset)
\end{equation}
and so we obtain that
\begin{align}
\bP_{p,\lambda,h}(\sA(v,u,w,e))&\leq \frac{p}{1-p} \bP_{p,\lambda,h} (v \leftrightarrow u, v \nleftrightarrow g_v)  \cdot \bP_{p,\lambda,h}\left(\{ w \leftrightarrow g_v \} \circ \{w \leftrightarrow g_v\} \right)
\nonumber
\\
& \leq \frac{p}{1-p} \bP_{p,\lambda,h} (v \leftrightarrow u, v \nleftrightarrow g_v) \cdot M^2_{p,\lambda,\Delta^\lambda(v,w) h}(w).
\end{align}
Applying \eqref{eq:LaplaceBound} to control the magnetization appearing here, we obtain that there exists a constant $C$ such that
\begin{align}
\bP_{p,\lambda,h}(\sA(v,u,w,e))
% &\leq \frac{p}{1-p} \P_{p,\lambda,h} (v \leftrightarrow u, v \nleftrightarrow g_v)  \cdot \P_{p,\lambda,h}\left(\{ w \leftrightarrow g_v \} \circ \{w \leftrightarrow g_v\} \right)\\
\leq \frac{Cp}{1-p} h^{1+\delta}\, \bP_{p,\lambda,h} (v \leftrightarrow u, v \nleftrightarrow g_v) \Delta^{(1+\delta)\lambda}(v,u).
\end{align}
Taking $v=\rho$, summing over the possible choices of $u,w,$ and $e$, and taking the expectation over $\rho$ yields that
\begin{multline}
\P_{p,\lambda,h}\left( \left\{\mathcal{G}_\rho(K_\rho) \geq 2 \right\} \setminus\left[\{\rho \leftrightarrow g_\rho\}\circ \{\rho\leftrightarrow g_\rho\}\right]\right) \\
% \preceq_p h^{1+\delta} \sum_{u \in V}\P_{p,\lambda,h} (\rho \leftrightarrow u, \rho \nleftrightarrow g_\rho) \Delta^{(1+\delta)\lambda}(\rho,u).
\leq \frac{Cp}{1-p} h^{1+\delta} \, \E_{p,\lambda,h} \left[ \sum_{u\in V} \mathbbm{1}\left(\rho \leftrightarrow u, \rho \nleftrightarrow g_\rho\right) \Delta^{(1+\delta)\lambda}(\rho,u) \right]
\label{eq:ghost1}
\end{multline}
We break the sum on the right hand side of \eqref{eq:ghost1} into two pieces according to whether $\Delta(\rho,u)\leq 1$ or $\Delta(\rho,u)>1$, and claim that the expectation of each such piece is bounded by $\chi_{p,\lambda,h}$. The first is easily handled by observing that, trivially,
\begin{multline}
 \E_{p,\lambda,h} \left[ \sum_{u\in V, \Delta(\rho,u) \leq 1} \mathbbm{1}\left(\rho \leftrightarrow u, \rho \nleftrightarrow g_\rho\right) \Delta^{(1+\delta)\lambda}(\rho,u) \right]\\ \leq \E_{p,\lambda,h} \left[ \sum_{u\in V, \Delta(\rho,u) \leq 1} \mathbbm{1}\left(\rho \leftrightarrow u, \rho \nleftrightarrow g_\rho\right) \Delta^{\lambda}(\rho,u) \right]  \leq \chi_{p,\lambda,h}.
\label{eq:ghost2}
\end{multline}
For the second, we apply the tilted mass-transport principle to obtain that
\begin{multline}
\E_{p,\lambda,h} \left[ \sum_{u\in V, \Delta(\rho,u) > 1} \mathbbm{1}\left(\rho \leftrightarrow u, \rho \nleftrightarrow g_\rho\right) \Delta^{(1+\delta)\lambda}(\rho,u) \right]\\
=
\E_{p,\lambda,h} \left[ \sum_{u\in V, \Delta(\rho,u) < 1} \mathbbm{1}\left(\rho \leftrightarrow u, \rho \nleftrightarrow g_u\right) \Delta^{1-(1+\delta)\lambda}(\rho,u) \right].
% =
% \sum_{u\in V, \Delta(\rho,v) < 1} \P_{p,\lambda,h} (\rho \leftrightarrow u, \rho \nleftrightarrow g_u) \Delta^{1-(1+\delta)\lambda}(\rho,u).
\end{multline}
If $\Delta(\rho,u)<1$ then $g_u$ stochastically dominates $g_\rho$, so that $\P_{p,\lambda,h}(\rho \leftrightarrow u , \rho \nleftrightarrow g_u ) \leq \P_{p,\lambda,h}(\rho \leftrightarrow u , \rho \nleftrightarrow g_\rho )$. Meanwhile, our choice of $\delta$ ensures that $1-(1+\delta)\lambda \geq \lambda$, and so we have that
\begin{multline}
\E_{p,\lambda,h} \left[ \sum_{u\in V, \Delta(\rho,u) > 1} \mathbbm{1}\left(\rho \leftrightarrow u, \rho \nleftrightarrow g_\rho\right) \Delta^{(1+\delta)\lambda}(\rho,u) \right]\\
\leq 
\E_{p,\lambda,h} \left[ \sum_{u\in V, \Delta(\rho,u) < 1} \mathbbm{1}\left(\rho \leftrightarrow u, \rho \nleftrightarrow g_\rho \right) \Delta^{\lambda}(\rho,u) \right]
\leq \chi_{p,\lambda,h}.
\label{eq:ghost3}
\end{multline}
Combining \eqref{eq:ghost1}, \eqref{eq:ghost2}, and \eqref{eq:ghost3} yields the desired inequality \eqref{eq:desiredghost}. \qedhere
% 
% there must either exist an open edge of $G$ or a green point 

\end{proof}

\section{Analysis of the tiltable and subcritical phases}
\label{sec:Fekete}

In this section we study percolation in the tiltable ($0<p<p_t$) and subcritical ($0<p<p_c$) phases. We begin by introducing Tim\'ar's \emph{uniform separating layer decomposition} in \cref{subsec:layersdef}. We then give an overview of the results of the section in \cref{subsec:tiltable_overview}. These results are then stated in detail and proven in the following subsections.

\subsection{The uniform separating layer decomposition}
\label{subsec:layersdef}
 % Let $G$ be a connected, locally finite graph, and let $\Gamma \subseteq \Aut(G)$ be nonunimodular and quasi-transitive.
 Recall that we have fixed a connected, locally finite graph $G$ and a nonunimodular, quasi-transitive subgroup $\Gamma \subseteq \Aut(G)$. 
For each $-\infty \leq s \leq t \leq \infty$ and $v \in V$ we define the \textbf{slab} 
\[
S_{s,t}(v)=\{u \in V : s \leq \log \Delta(v,u) \leq t\}.
\]
We also define
\[t_0 = \sup \left\{\log \Delta(u,v) : u,v \in V, u \sim v\right\},\]
so that for every $t\in \R$, every path in $G$ that starts at a vertex in the slab $S_{-\infty,t}(v)$ and ends at a vertex in the slab $S_{t,\infty}(v)$ must pass through the slab $S_{t,t+t_0}(v)$. It will be convenient to write
\[
\nexp(x) = \exp( t_0 x)  \qquad  \text{ and } \qquad  \nlog x = \frac{1}{t_0}\log x
\]
for the appropriately normalized exponential and logarithm.
% It is a consequence of the cocycle identity that there exists a collection $\lambda_1,\ldots,\lambda_k \in (1,\infty)$ such that
% \[
% \{\Delta(u,v) : u,v \in V \} = \{\lambda_1^{n_1} \lambda_2^{n_2} \cdots \lambda_k^{n_k} : n_1,\ldots,n_k \in \Z\}.
% \]
% We say that $(G,\Gamma)$ has \textbf{discrete layers} if there exists a (necessarily unique) $\lambda_0 \in (1,\infty)$ such that
% \[
% \{\Delta(u,v) : u,v \in V \} = \{\lambda^n : n \in \Z\}.
% \]
% If $(G,\Gamma)$ has discrete layers. Note in particular that if 

% Note if $n \leq m$ and $v\in V$ then any path in $G$ that starts in $L_n(v)$ and ends in $L_m(v)$ must visit $L_\ell(v)$ for every $n \leq \ell \leq m$. Moreover, the tilted mass-transport principle implies that
% \[
% \E\left[\sum_{v\in V} F(\rho,v)\right] = e^{-t_0 k} \E\left[\sum_{v\in V} F(v,\rho)\right]
% \]
% whenever $F:V^2 \to [0,\infty]$ is invariant under the diagonal action of $\Gamma$ and supported on pairs $u,v$ with $v\in L_k(u)$. These features make the layer decomposition $V = \bigcup L_n(v)$ very useful for studying percolation when $(G,\Gamma)$ has discrete layers. 

% Following Tim\'ar \cite{timar2006percolation}, we now generalize the separating layers decomposition to the case that $(G,\Gamma)$ does not necessarily have discrete layers. 
% This construction will be convenient when applying the tilted mass-transport principle in calculations.

The following construction, due to Tim\'ar \cite{timar2006percolation}, will be very useful.
Let $G[p]$ be a Bernoulli bond percolation on $G$, and let $\rho \in \cO$ be a random variable with law $\mu$, defined in the previous subsection, independent of $G[p]$. 
Let $v_0$ be an arbitrary vertex of $G$, let $U_{v_0}$ be a uniform $[0,1]$ random variable independent of $\rho$ and $G[p]$, and let 
\[
U_v = U_{v_0} - \nlog \Delta(v_0,v) \mod 1
\]
for every other $v\in V$.  The law of the collection of random variables $U=\{ U_v : v\in V\}$ does not depend on the choice of $v_0$. From now on, we write $\P_p$ and $\E_p$ for probabilities and expectations taken with respect to the joint law of $G[p]$, $U$, and $\rho$. 
% (We continue to write $\bP_p$ and $\bE_p$ for the law of $G[p]$ and the associated expectation operator.)
Given $U$, we define the \textbf{separating layers}
\[
L_n(v) =  
\Bigl\{ x \in V :  (n+U_v-1) \leq \nlog \Delta(v,x) \leq (n+U_v)\Bigr\}
= S_{(n+U_v-1)t_0,(n+U_v)t_0}(v).
\]
for each $n \in \Z$ and $v\in V$, and 
% \textbf{half-spaces}
% \begin{align*}
% L_{n,\infty}(v) &= \bigcup_{m \geq n} L_m(v) = \big\{ x \in V :  \log \Delta(v,x) \geq (n+U_v-1)t_0\big\}
% \end{align*}
% \begin{align*}
% H^-_n(v) &= \bigcup_{m \leq n} L_m(v) = \big\{ x \in V :  \log \Delta(v,x) \leq (n+U_v)t_0\big\}.
% \end{align*}
% We also 
define \[L_{m,n}(v)=\bigcup_{k=m}^n L_k(v)\] for every $v\in V$ and $-\infty \leq m\leq n \leq \infty$. Note that $L_n(v)$ and $L_m(v)$ are almost surely disjoint if $n \neq m$ and that if $u \in L_k(v)$ then $L_m(u)=L_{m+k}(v)$. We will refer to sets of the form $L_{n,\infty}(v)$ and $L_{-\infty,n}(v)$ as upper and lower \textbf{half-spaces} respectively\footnote{Note however that, in hyperbolic space, $L_{n,\infty}(v)$ is analogous to a horoball rather than a half-space.}. Intuitively, we think of $\nlog \Delta$ as being a sort of `normalized height function', and think of $L_n(v)$ and $L_{n,m}(v)$ as unit layers and slabs of integer normalized height, each with a random offset.

\medskip

The group $\Gamma$ acts on $\{0,1\}^E$ and $[0,1]^V$ by $\gamma \omega(e)=\omega(\gamma^{-1} e)$ and $\gamma \omega(v) = \omega(\gamma^{-1}v)$ respectively. We define the diagonal action of $\Gamma$ on
% We define the diagonal action of $\Gamma$ on 
$V^2\times\{0,1\}^E \times [0,1]^V$ by setting 
\[\gamma(u,v,\omega_1,\omega_2) = (\gamma u, \gamma v , \gamma \omega_1, \gamma \omega_2) \]
for each $\gamma \in \Gamma$ and $(u,v,\omega_1,\omega_2)\in V^2\times\{0,1\}^E\times[0,1]^V$. 
If 
$f:V^2\times\{0,1\}^E\times [0,1]^V \to [0,\infty]$ is invariant under the diagonal action of $\Gamma$, then applying the tilted mass-transport principle to the function
\[
F(u,v)=\E_p\left[f\!\left(u,v,G[p],U\right)\right]
\]
yields that
 % and is such that $F(u,v,\omega,U)$ is supported on $L_n(u)$ for every $u \in V$,
 % then the tilted mass-transport principle implies that
\begin{equation}
\label{eq:fullMTP}
\E_p\left[ \sum_{v\in V} f\!\left(\rho,v,G[p],U\right) \right] =\E_p\left[ \sum_{v\in V} f\!\left(v,\rho,G[p],U\right)\Delta(\rho,v)\right].
\end{equation}
We refer to this equality simply as the tilted mass-transport principle also.
In particular, if $k\in \Z$ and $f:V^2\times\{0,1\}^E\times[0,1]^V\to[0,\infty]$ is supported on pairs $u,v$ with $v \in L_k(u)$, then we have the approximate equality
\begin{equation}
\label{eq:roughMTP}
\E_p\left[ \sum_{v\in V} f\!\left(\rho,v,G[p],U\right) \right] \asymp\nexp(-k)\E_p\left[ \sum_{v\in V} f\!\left(v,\rho,G[p],U\right)\right].
\end{equation}
Indeed, the equality is exact up to a factor of $e^{\pm t_0}$.

\medskip

% \begin{remark}
Let us now draw attention to a special case in which our proofs can often be substantially simplified. 
 It is a consequence of quasi-transitivity and the cocycle identity that there exists $k\geq 1$ and a collection $\lambda_1,\ldots,\lambda_k \in (1,\infty)$ such that
\[
\{\Delta(u,v) : u,v \in V \} = \{\lambda_1^{n_1} \lambda_2^{n_2} \cdots \lambda_k^{n_k} : n_1,\ldots,n_k \in \Z\}.
\]
We say that $(G,\Gamma)$ has \textbf{discrete layers} if there exists a (necessarily unique) $\lambda_0 \in (1,\infty)$ such that
\[
\{\Delta(u,v) : u,v \in V \} = \{\lambda_0^n : n \in \Z\}.
\]
We say furthermore that $(G,\Gamma)$ has \textbf{simple layers} if it has discrete layers with constant $\lambda_0=e^{t_0}$. This assumption holds in particular 
when $G=T_k \times H$ for $T_k$ a $k$-regular tree for $k\geq 3$, $H$ is transitive and unimodular, and $\Gamma=\Gamma_\xi\times \Aut(H)$ is the product of the group of automorphisms of $T_k$ fixing an end with the full automorphism group of $H$. Observe that if $(G,\Gamma)$ has simple layers then we almost surely have that
$L_n(v) = \bigl\{ u \in V : \nlog \Delta(v,u) = n \bigr\}$  for every $n \in \Z$ and $v \in V$.
% \end{remark}
% \end{equation}

\subsection{Overview of results}
\label{subsec:tiltable_overview}

 Consider the triple of random variables $(G[p],\rho,U)$ as in \cref{subsec:layersdef}. 
For each $v\in V$, $-\infty\leq m \leq n \leq \infty$ and $m \leq k \leq n$ we define
\[
X_{k}^{m,n}(v) = \big|\big\{ x \in L_k(v) : v \xleftrightarrow{L_{m,n}(v)} x \big\}\big|
\]
to be the number of points in $L_k(v)$ that are connected to $v$ by an open path in the subgraph of $G[p]$ induced by $L_{m,n}(v)$. We also define $X^{m,n}_k=X^{m,n}_k(\rho)$. 
The goal of the remainder of this section is to study the distribution of these random variables, primarily in the tiltable phase $p<p_t$.

The results obtained in this section can be summarised as follows. (The precise results we prove will in some cases be a little stronger and more technical.)
\begin{enumerate}
	\item (\cref{lem:alpha}) For every $p\in (0,1)$, there exists $\alpha_p \geq 0$ such that
	\begin{equation}
	\label{eq:alphaoverview}
\P_p\left( X^{0,n}_n > 0 \right) = \P_p\left( X^{0,\infty}_n > 0 \right) = \nexp\bigl[-\alpha_p  n + o_p(n)\bigr] \preceq_p \nexp\bigl[-\alpha_p n\bigr]
	\end{equation}
	as $n\to +\infty$. This is proven using Fekete's Lemma.
	\item (\cref{lem:beta}) For every $0<p<p_t$, there exists $\beta_p \geq 0$ such that
	\begin{equation}
	\label{eq:betaoverview} \E_p\left[ X^{-\infty,n}_n \right] = \nexp\bigl[-\beta_p n + o_p(n)\bigr] \succeq_p \nexp\bigl[-\beta_p n\bigr]
	\end{equation}
	as $n\to +\infty$. This is also proven using Fekete's Lemma.
	\item (\cref{lem:alphacontinuity}) $\alpha_p$ is left-continuous in $p$ and satisfies $\alpha_{
\pcl{\lambda}}\geq \max\{\lambda,1-\lambda\}$ for every $\lambda\in \R$. In particular, $\alpha_{p_c}\geq 1$. 
	% (In fact, once \cref{thm:pcpt} is proven, it will follow that $\alpha_{p_c}=\beta_{p_c}=1$.)
	\item (\cref{lemma:gammabeta,prop:tamedecay,prop:tamesecondmoment})  $p<p_t$ if and only if $\alpha_p \geq \beta_p >1/2$, and in this case  $\alpha_p=\beta_p$ and the $o_p(n)$ corrections in \eqref{eq:alphaoverview} and \eqref{eq:betaoverview} are in fact $O_p(1)$. In particular, $\alpha_{\pcl{\lambda}} = \beta_{\pcl{\lambda}}=\max\{\lambda,1-\lambda\}$ for every $\lambda \in \R$ with $\pcl{\lambda} < p_t$. Moreover, for $\beta_p>2/3$ the expectations of the products $\E_p\left[ X_n^{-\infty,\infty} X_m^{-\infty,\infty} \right]$ also admit similar descriptions up to constant factors.
	\item (\cref{lem:subcritpeaksurvival}) If $p<p_c$, then 
	\begin{equation}\P_p\left( X^{-n,0}_{-n} > 0 \right) \asymp_p \nexp\bigl[-(\alpha_p-1)n\bigr]. \end{equation}
	Moreover, the same estimate holds conditional on the event that $\rho$ is the unique highest point of its cluster.
\end{enumerate}

Roughly speaking, for $0<p<p_t$, the above results show that $(X^{-\infty,\infty}_n)_{n\geq 0}$ behaves similarly to a subcritical branching process, whereas $(X^{-\infty,\infty}_{-n})_{n\geq 0}$ behaves similarly to a branching process that is either subcritical (if $p<p_c$), critical (if $p=p_c$), or supercritical (if $p_c<p<p_t$).

% \medskip

In terms of their application to the proofs of the  main theorems, the most important estimates obtained from these considerations are
\begin{equation}
\P_{p_c}\left( X_n^{0,n} >0 \right) \preceq \nexp\left(-\alpha_{p_c}  n\right) \preceq \nexp(-n),
\end{equation}
which follows from 1, 2, and 3 above, and
\begin{equation}
\P_{p_c}\left( X_{-n}^{-n,0} >0 \mid \text{ $\rho$ is the unique highest point of its cluster}\right) \geq \nexp\bigl[-(\alpha_{p_c}-1) n + o(n)\bigr],
\end{equation}
which is proven in \cref{lem:peaksurvival} of \cref{sec:pc} using the estimates from items 4 and 5 above.
 Intuitively, these estimates imply that, at criticality, crossing a large slab from bottom to top is much more difficult than crossing from top to bottom.

\begin{remark}
The proofs in this and the following section can be simplified substantially if one assumes that $\Gamma$ is transitive and that 
$(G,\Gamma)$ has simple layers. 
% Note that these hypotheses are satisfied with $q=k-1$ when $G=T_k \times H$ for $T_k$ a $k$-regular tree for $k\geq 3$ and $H$ is transitive and unimodular, and $\Gamma=\Gamma_\xi\times \Aut(H)$ is the product of the group of automorphisms of $T_k$ fixing an end with the automorphism group of $H$. 
For these graphs many tedious technicalities resulting from the inhomogeneity of the uniform separating layers decomposition do not arise; For example, \cref{lem:probinfsup,lem:bestworst} are used specifically to deal with this inhomogeneity and are not needed in the above special case. The reader may find it an illuminating exercise to simplify the proofs in this case. 
\end{remark}

\begin{remark}
It is a consequence of \cite[Theorem 2.38]{grimmett2010percolation} that $\alpha_p$ and $\beta_p$ are both strictly decreasing when they are positive. A further straightforward fact is that $\beta_p$ is right-continuous on $(0,p_t)$. Together with item 3 above this implies that $\alpha_p=\beta_p$ is continuous on $(0,p_t)$. Since these facts will not be used in the proofs of the main theorems, their proofs are omitted. 
\end{remark}

\subsection{Probability decay}
\label{subsec:probdecay}

We begin by studying the probability of connecting from the bottom to the top of a thick slab.

\begin{lemma}[Probability decay]
\label{lem:alpha}
The limit
% \begin{enumerate}
	% \item The limits
\begin{equation}
\label{eq:probsupmult}
\alpha_p := -\lim_{n\to\infty}\frac{1}{n}\nlog \P_p\Big(v \xleftrightarrow{L_{0,\infty}(v)} L_n(v) \mid U_v =x \Big) \in [0,\infty)
\end{equation}
% and
% \[
% \beta_p := -\lim_{t\to\infty}\frac{1}{t}\log E_p(t) \in \{-\infty\} \cup [0,\infty)
% \]
exists for every $p\in (0,1]$, and does not depend on $v\in V$ or $x\in [0,1]$. Furthermore,
we have that
% there exists a constant $C=C(G,\Gamma)$ such that the estimate
\begin{equation}
\label{eq:alphapupperbound}
\P_p\Big(v \xleftrightarrow{L_{0,\infty}(v)} L_n(v) \mid U_v=x \Big) \preceq_p   \nexp(-\alpha_p n)
\end{equation}
for every $p\in (0,1]$, $v\in V$, $x\in [0,1]$, and $n\geq0$.
\end{lemma}

We stress that \eqref{eq:probsupmult} \emph{defines} the quantity $\alpha_p$ for each $p\in(0,1]$.
% 
% \medskip
% 
% \[
% \alpha = \inf_{n\geq 1} \frac{a(n)+c}{n+n_0}.
% \]
% 
% 
% 
The proof will use the following lemma, which allows us to compare infimal and supremal choices of $v$ and $U_v$. The additional parameter $r$ will not be used in the proof of \cref{lem:alpha}, but is included for later use in \cref{sec:pc}.

\begin{lemma}
\label{lem:probinfsup}
There exist positive constants $r_0$ and $C$ such that
\[
\inf_{v\in V, x\in [0,1]} \P_p\Big(v \xleftrightarrow{L_{0,\infty}(v)} L_{n}(v) \mid U_v=x \Big) \geq p^{Cr+Cr_0} \sup_{v\in V, x\in [0,1]} \P_p\Big(v \xleftrightarrow{L_{-r,\infty}(v)} L_{n}(v) \mid U_v=x \Big)
\]
for every $n\geq 1$ and $r\geq 0$.
\end{lemma}

\begin{proof}
Quasi-transitivity of $\Gamma$ and the maximum principle applied to $\Delta$ implies that there exists $r_0$ such that for any two vertices $u,v\in V$, there exists a path  $u=u_0,u_1,\ldots,u_{k}$ in $G$ such that $k\leq r_0(1+r)$, $u_{k}$ is in the same orbit as $v$, $(r+1) \leq \nlog \Delta(u,u_{k})$, and $\nlog \Delta(u,u_i)\geq 0$ for every $1\leq i \leq k$.
Given such a path, it follows by the Harris-FKG inequality that
\begin{align}
\P_p\Bigl(u \xleftrightarrow{L_{0,\infty}(u)} L_n(u) \mid U_u=1\Bigr) &\geq p^{r_0(1+r)} \P_p\Bigl(u_k \xleftrightarrow{L_{-r-1,\infty}(u_k)} L_{n-1}(u_k) \mid U_u=1\Bigr).
\end{align}
Considering the definitions of $L_{-r-1,\infty}(u_k)$ and $L_{n-1}(u_k)$ yields that
\begin{align}
\P_p\Bigl(u \xleftrightarrow{L_{0,\infty}(u)} L_n(u) \mid U_u=1\Bigr)&\geq \ p^{r_0(1+r)} \P_p\Bigl(v \xleftrightarrow{L_{-r,\infty}(v)} L_{n}(v) \mid U_v=0\Bigr),
\end{align}
and the estimate \eqref{eq:alphapupperbound} follows by  observing that $\P_p\Bigl(v \xleftrightarrow{L_{-r,\infty}(v)} L_n(v) \mid U_v=x\Bigr)$ and $\P_p\Bigl(u \xleftrightarrow{L_{0,\infty}(u)} L_n(u) \mid U_u=x\Bigr)$ are both decreasing functions of $x\in [0,1]$.
% \begin{align}
% \sup_{x\in [0,1]} \P_p\Bigl(v \xleftrightarrow{L_{0,\infty}(v)} L_n(v) \mid U_v=x\Bigr) &= \P_p\Bigl(v \xleftrightarrow{L_{0,\infty}(v)} L_n(v) \mid U_v=0\Bigr)\\
% \intertext{and}
% \inf_{x\in [0,1]} \P_p\Bigl(v \xleftrightarrow{L_{0,\infty}(v)} L_n(v) \mid U_v=x\Bigr) &= \P_p\Bigl(v \xleftrightarrow{L_{0,\infty}(v)} L_n(v) \mid U_v=1\Bigr)
% \end{align}
% for every $v\in V$.
\end{proof}

The proof of \cref{lem:alpha} will also apply \textbf{Fekete's Lemma} \cite[Appendix II]{grimmett2010percolation}, one form of which is as follows: Suppose that $a(n)$ is a sequence of real numbers satisfying the 
% generalized
 subadditive estimate
$a(n+m) \leq a(n)+a(m)$
for 
% some constant $c \in \R$
  every $n,m\geq 0$. 
Then we have that
\begin{equation}
\label{eq:Fekete}
 \lim_{n\to\infty} \frac{a(n)}{n} = \inf_{n\geq 1} \frac{a(n)}{n} \in [-\infty,\infty).
\end{equation}
In particular, the limit on the left hand side exists.
% exists and satisfies

\begin{proof}[Proof of \cref{lem:alpha}]
 For each $t\geq 0$ let
$P_p(t) = 
\inf_{v\in V} \bP_p\bigl(v \xleftrightarrow{S_{0,\infty}(v)} S_{t,\infty}(v) \bigr)$.
% = \inf_{v\in V} \P_p\Big(v \xleftrightarrow{S_{0,\infty}(v)} S_{t_0 n,\infty}(v) \Big).
% By quasi-transitivity of $\Gamma$ and harmonicity of $\Delta$ (specifically, the maximum principle applied to $\Delta$), there exists $k_0\geq 1$ such that for every vertex $v\in V$, there exists $k\leq k_0$ and a path $v=v_0,v_1,\ldots,v_k$ in $G$ such that $\nlog \Delta(v,v_k)\geq 1$ and the sequence $\nlog \Delta(v,v_i)$ is (weakly) increasing. 
% Let $k_0$ be minimal with this property. 
We claim that the supermultiplicative estimate
\begin{equation}
% \inf_{v\in V, x\in [0,1]} \P_p\Big(v \xleftrightarrow{L_{0,\infty}(v)} L_{n+m}(v) \mid U_v=x \Big) \\\geq p^{\lceil t_0/s_0 \rceil} \inf_{v\in V, x\in [0,1]} \P_p\Big(v \xleftrightarrow{L_{0,\infty}(v)} L_n(v) \mid U_v=x \Big)\inf_{v\in V, x\in [0,1]} \P_p\Big(v \xleftrightarrow{L_{0,\infty}(v)} L_m(v) \mid U_v=x \Big)
P_p(s+t) \geq  P_p(s)P_p(t)
\label{eq:probsupermult}
% \intertext{and}
% \E_p\left[ X^d \right] &\leq E_p(t)E_p(s)
\end{equation}
holds for every $p\in [0,1]$ and $s,t \geq 0$. Indeed, fix $v \in V$ and $s,t\geq 0$, and let $Y_t(v)$ be the set of vertices in $S_{t,\infty}(v)$ that are connected to $v$ by an open path in $S_{0,t}(v)$ none of whose edges have both endpoints in $S_{t,\infty}(v)$.
Thus, the event $v \xleftrightarrow{S_{0,\infty}(v)} S_{t,\infty}(v)$ occurs if and only if $Y_t(v)\neq \emptyset$, and $Y_t(v)$ is independent of the status of every edge that has both endpoints in $S_{t,\infty}(v)$.
Condition on $Y_t(v)$ and the event that $Y_t(v) \neq \emptyset$, and let $u$ be chosen arbitrarily from $Y_t(v)$. Then the aforementioned independence property implies that the conditional probability that $u$ is connected to $S_{s+t,\infty}(v) \supseteq S_{s,\infty}(u)$ by an open path in $S_{0,\infty}(u)$ is at least $P_p(s)$. Thus, we have that
\begin{equation}
\bP_p\Big(v \xleftrightarrow{S_{0,\infty}(v)} S_{s+t,\infty}(v) \Big) \geq \bP_p\Big(v \xleftrightarrow{S_{0,\infty}(v)} S_{t,\infty}(v) \Big) P_p(s),
\end{equation}
 and the claim follows by taking an infimum over $v$ on both sides.

Now observe that
\begin{equation}
\inf_{x\in [0,1]}\P_p\Big(v \xleftrightarrow{L_{0,\infty}(v)} L_n(v) \mid U_v =x \Big)
=
\bP_p\Big(v \xleftrightarrow{S_{0,\infty}(v)} S_{t_0n,\infty}(v) \Big)
\end{equation}
for every $n \geq 1$ and $v\in V$.
Thus, applying Fekete's lemma to the sequence $\langle -\nlog P_p(t_0n) \rangle_{n\geq1}$, we obtain from \eqref{eq:probsupermult} that the limit
\begin{equation}
\label{eq:alphadefinproof}
\alpha_p=-\lim_{n\to\infty}\frac{1}{n} \nlog P_p(t_0n) = \inf_{n\geq 1} \frac{1}{n}\Bigl(- \nlog P_p(t_0 n)\Bigr)
\end{equation}
exists for every $p\in [0,1]$. 
% Rearranging, we have that 
% \begin{equation}
% \label{eq:Ppn}
% P_p(n) \leq p^{-k_0}\nexp(-\alpha_p n) \preceq_p \nexp(-\alpha_p n)
% \end{equation}
% for every $n\geq 1$. 
The uniform upper bound \eqref{eq:alphapupperbound}, and the fact that the limit in \eqref{eq:probsupmult} exists and is equal to $\alpha_p$ for all $v\in V$ and $x\in [0,1]$, follows from 
 \eqref{eq:alphadefinproof} together with \cref{lem:probinfsup}.
\end{proof}

\begin{lemma}
\label{lem:alphacontinuity}
 $\alpha_p$ is left-continuous on $(0,1]$.
\end{lemma}

\begin{proof} Recall that left-continuity is equivalent to lower semi-continuity for increasing functions and is equivalent to upper semi-continuity for decreasing functions. Moreover, lower semi-continuity is preserved by taking minima over finite collections of functions and by taking suprema over arbitrary collections of functions. 
For each $v\in V$, $t\geq 0$, and $x\in [0,1]$, the probability
\[  \bP_p\Big(v \xleftrightarrow{S_{0,\infty}(v)} S_{t,\infty}(v) \Big)\]
can be written as the supremum of the continuous increasing functions
\[\bP_p\Big(v \xleftrightarrow{S_{0,\infty}(v)} S_{t,\infty}(v) \text{ by an open path of length at most $r$} \Big),\]
and is therefore lower semi-continuous. Since $\Gamma$ is quasi-transitive, $P_p(t_0n)$ can be written as a minimum of finitely many such functions, and thus is lower semi-continuous itself. Using the expression \eqref{eq:alphadefinproof}, we see that $-\alpha_p$ can be written as a supremum of lower semi-continuous functions and is therefore lower semi-continuous itself. Since $-\alpha_p$ is increasing in $p$ the result follows.
% This follows from a standard semi-continuity argument, very similar to that used in \cite{Hutchcroft2016944}. 
\end{proof}

\subsection{Path-decomposition inequalities}
\label{subsec:decomps}

We now gather and prove several related inequalities that will be used in the following subsection, each of which follows by a standard application of the BK inequality.
% The following inequalities follow by standard applications of the BK inequality.
We define
\begin{equation}
\label{eq:Edef}
 E_p^{m,n}(k) = \sup_{v\in V, x\in [0,1]} \E_p\left[ X^{m,n}_k(v) \mid U_v =x \right] 
\end{equation}
for every $-\infty \leq m \leq k \leq n \leq \infty$.

\begin{lemmae}[Path-decomposition inequalities]
\label{lem:pathdecomposition}
The following inequalities hold for every $p\in [0,1]$.
\begin{enumerate}[leftmargin=*]
\item (First visit decomposition.) For every $-\infty \leq m \leq \ell \leq k \leq n \leq \infty$  with $\ell \geq 0$ we have that
\begin{align}
\label{eq:firstvisitdecompup}
	\E_p\left[ X_k^{m,n}(v) \mid U_v =x \right] &\leq \E_p\left[ X_\ell^{m,\ell}(v) \mid U_v =x \right]E_p^{m-\ell,n-\ell}(k-\ell) 
\intertext{for every $v\in V$ and $x\in [0,1]$.
Similarly, for every $-\infty \leq m \leq k \leq \ell \leq n \leq \infty$  with $\ell \leq 0$ we have that
}
	\E_p\left[ X_k^{m,n}(v) \mid U_v =x \right] &\leq \E_p\left[ X_\ell^{\ell,n}(v) \mid U_v =x \right]E_p^{m-\ell,n-\ell}(k-\ell)
\label{eq:firstvisitdecompdown}
\end{align}
for every $v\in V$ and $x\in [0,1]$.
\item (Last visit decomposition.)
For every $-\infty \leq m \leq \ell \leq k \leq n \leq \infty$  with $\ell \geq 0$ we have that
\begin{align}
\label{eq:lastvisitdecompup}
\E_p\left[ X_k^{m,n}(v) \mid U_v =x \right] &\leq \E_p\left[ X_\ell^{m,n}(v) \mid U_v =x \right]E_p^{0,n-\ell}(k-\ell) \intertext{for every $v\in V$ and $x\in [0,1]$.
Similarly, for every $-\infty \leq m \leq k \leq \ell \leq n \leq \infty$  with $\ell \leq 0$ we have that
}
\E_p\left[ X_k^{m,n}(v) \mid U_v =x \right] &\leq \E_p\left[ X_\ell^{m,n}(v) \mid U_v =x \right]E_p^{m-\ell,0}(k-\ell)
\label{eq:lastvisitdecompdown}
\end{align}
for every $v\in V$ and $x\in [0,1]$.
\item (Extreme point decompositions.) For every $-\infty \leq m \leq k \leq n \leq \infty$  we have that
\begin{align}
\label{eq:extremedecompup}
\E_p\left[ X_k^{m,n}(v) \mid U_v =x \right] &\leq \sum_{\ell = k \vee 0}^n \E_p\left[ X_\ell^{m,\ell}(v) \mid U_v =x \right]E_p^{m-\ell,0}(k-\ell) \intertext{and similarly that
% for every $v\in V$ and $x\in [0,1]$.
% Similarly, for every $0<p<1$ and $-\infty \leq m \leq \ell \leq k \leq n \leq \infty$  with $\ell \leq 0$ we have that
}
\E_p\left[ X_k^{m,n}(v) \mid U_v =x \right] &\leq \sum_{\ell = m}^{k \wedge 0} \E_p\left[ X_\ell^{\ell,n}(v) \mid U_v =x \right]E_p^{0,n-\ell}(k-\ell)
\label{eq:extremedecompdown}
\end{align}
for every $v\in V$ and $x\in [0,1]$.
\end{enumerate}
\end{lemmae}

\begin{figure}[t]
\vspace{-1em}
\centering
\includegraphics[width=0.73\textwidth]{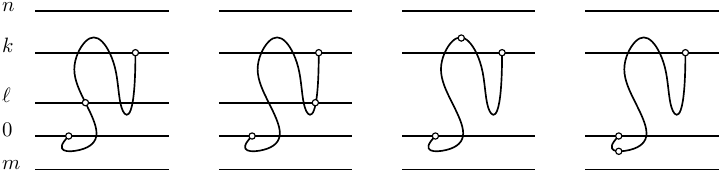}
\caption{Any simple path starting in level zero and ending in level $k\geq 0$ can be decomposed into two disjoint paths either by considering the first time it visits some intermediate level $0 \leq \ell \leq k$ (far left), the last time it visits such an intermediate level (centre left), the first time it attains its maximum height (centre right), or the first time it attains its minimum height (far right). Similar decompositions exist for paths ending in a level of negative height. Applying the BK inequality yields the estimates of \cref{lem:pathdecomposition}.}
\end{figure}

\vspace{-0.5em}

\begin{proof}
Let $v\in V$, $x\in [0,1]$ and  condition on $U_v=x$.
In order to give a representative sample of the proofs, we prove \eqref{eq:firstvisitdecompup} and \eqref{eq:extremedecompup}. The proofs of the remaining inequalities are similar. We begin with \eqref{eq:firstvisitdecompup}. Observe that, since $0\leq l \leq k$, if $u\in L_k(v)$ is connected to $v$ by an open simple path in $L_{m,n}(v)$ then this path must visit $L_\ell(v)$ for some first time, at some vertex $w \in L_\ell(v)$, and the part of this path up until this first visit to $L_\ell(v)$ is contained in $L_{m,\ell}(v)$.
Thus, we have the containment of events
\begin{equation}
\{u \xleftrightarrow{L_{m,n}(v)} v\} \subseteq \bigcup_{w\in L_\ell(v)} \left(\{v \xleftrightarrow{L_{m,\ell}(v)} w\} \circ \{ w \xleftrightarrow{L_{m,n}(v)} u\}\right).
\end{equation}
Applying the union bound and the BK inequality and summing over $w \in L_\ell(v)$ and $u\in L_k(v)=L_{k-\ell}(w)$, we obtain that
\begin{align}
\E_p\left[ X^{m,n}_k \mid U_v = x\right] &\leq \sum_{w \in L_\ell(v)} \P_p( v \xleftrightarrow{L_{m,\ell}(v)} w \mid U_v = x) \sum_{u\in L_{k}(v)} \P_p( w \xleftrightarrow{L_{m,n}(v)} u \mid U_v = x)
\nonumber
\\
&\leq 
\E_p\left[ X_\ell^{m,\ell}(v) \mid U_v =x \right]E_p^{m-\ell,n-\ell}(k-\ell) 
\end{align}
as claimed.

We now turn to \eqref{eq:extremedecompup}. Suppose that $u \in L_k(v)$ is connected to $v$ by an open path in $L_{m,\ell}(v)$ but not in $L_{m,\ell-1}(v)$. Then any open simple path from $v$ to $u$ in $L_{m,\ell}(v)$ must visit some vertex $ w \in L_{\ell}(v)$.
 % Thus, we have the containment of events
 A similar argument to above yields the containment of events
\begin{equation}
\{u \xleftrightarrow{L_{m,\ell}(v)} v\} \setminus \{u \xleftrightarrow{L_{m,\ell-1}(v)} v\}  \subseteq \bigcup_{w\in L_\ell(v)} \left(\{v \xleftrightarrow{L_{m,\ell}(v)} w\} \circ \{ w \xleftrightarrow{L_{m,\ell}(v)} u\}\right).
\end{equation}
Thus, applying the union bound and the BK inequality as above we obtain that
\begin{align}
\E_p\left[ X^{m,\ell}_k-X^{m,\ell-1}_k \mid U_v = x\right] 
&\leq 
\E_p\left[ X_\ell^{m,\ell}(v) \mid U_v =x \right]E_p^{m-\ell,0}(k-\ell).
\end{align}
Summing over $\ell$ completes the proof.
% there must 
 % since every open simple path from $v$ to $L_{k}(v)$
\end{proof}

\subsection{Expectation decay}
\label{subsec:expdecay}

In this section we apply similar arguments to those of \cref{subsec:probdecay} to study the exponential rate of growth/decay of the \emph{expected} number of points that are connected to in a slab. We define the interval
\[
I_h := \Bigl\{ p \in (0,1] : E_p^{-\infty,n}(k) < \infty \text{ for every $-\infty < k \leq n < \infty$}\Bigr\},
\]
% \[
% \tilde p_h := \sup \Bigl\{ p \in [0,1] : E_p^{-\infty,n}(k) < \infty \text{ for every $-\infty < k \leq n < \infty$}\Bigr\} 
% \]
and define $\tilde p_h = \sup I_h$.
Note that $\tilde p_h \geq p_t \geq p_c$ and that $(0,\tilde p_h) \subseteq I_h \subseteq (0,\tilde p_h]$.

\begin{lemma}[Expectation decay]
\label{lem:beta}
Let $p\in I_h$. 
Then the limits
\begin{align}
\label{eq:betadef}
\beta_p &:= - \lim_{k\to\infty} \frac{1}{k} \nlog \E_p\Big[ X_k^{-\infty,k}(v) \mid U_v =x \Big]
= 1 - \lim_{k\to\infty} \frac{1}{k} \nlog \E_p\Big[ X_{-k}^{-\infty,0}(v) \mid U_v =x \Big]
\end{align}
exist, are equal, and do not depend on $v\in V$ or $x\in [0,1]$. Furthermore, the estimates
% there exists a constant $C=C(G,\Gamma)$ such that the estimate
\begin{align}
\label{eq:betalowerbound}
\E_p\Big[ X_k^{-\infty,k}(v) \mid U_v =x \Big] &\succeq_p \nexp\bigl[ -\beta_p k\bigr]
% \\
\intertext{and}
\E_p\Big[ X_{-k}^{-\infty,0}(v) \mid U_v =x \Big] &\succeq_p \nexp\bigl[ -(\beta_p-1) k \bigr]
\end{align}
hold for every $p\in (0,1]$, $k\geq0$, $v\in V$ and $x\in [0,1]$.
\end{lemma}

Again, we stress that \eqref{eq:betadef} \emph{defines} the quantity $\beta_p$ for each $p\in I_h$. 

\begin{remark}
We believe that it is possible to prove that $\tilde p_h =p_h$ and that $\alpha_p \geq \beta_p >0$ for every $p\in (0,p_h)$. This could be thought of as a sharpness result for the heaviness transition. Since this result is not needed for the proofs of our main theorems we do not pursue it here.
\end{remark}

  As with \cref{lem:alpha}, the proof of \cref{lem:beta} will use Fekete's Lemma. The following Lemma, which plays a role analogous to \cref{lem:probinfsup}, allows us to compare supremal and infimal choices of $v\in V$ and $x\in [0,1]$. (This lemma is not needed in the case that the graph is transitive and has simple layers.)

\begin{lemma}
\label{lem:bestworst}
Let $p \in I_h$. 
Then 
\begin{align}
\inf_{v\in V, x\in [0,1]} \E_p\left[ X_k^{-\infty,k}(v) \mid U_v = x \right]  
% \geq p^r C_p^{-1}
&\succeq_p
 \sup_{v\in V, x\in [0,1]}\E_p\left[ X_{k+1}^{-\infty,k+1}(v) \mid U_v = x \right]
 \label{eq:bestworstup}
% \]
\intertext{for every $k\geq 0$ and}
% \[
\inf_{v\in V, x\in [0,1]} \E_p\left[ X_{-k}^{-\infty,0}(v) \mid U_v = x \right]  
% \geq p^r C_p^{-1}
&\succeq_p
 \sup_{v\in V, x\in [0,1]}\E_p\left[ X_{-k-1}^{-\infty,0}(v) \mid U_v = x \right]
 \label{eq:bestworstdown}
\end{align}
for every $k\geq 0$.
\end{lemma}

\begin{proof} 
Let $p\in I_h$. 
We prove \eqref{eq:bestworstup}, the proof of \eqref{eq:bestworstdown} being similar. We begin by proving that
\begin{equation}
\label{eq:bestworstonevert}\inf_{x\in [0,1]} \E_p\left[ X_{k}^{-\infty,k}(v) \mid U_v = x \right]
\succeq_p \sup_{x\in [0,1]} \E_p\left[ X_{k+1}^{-\infty,k+1}(v) \mid U_v = x \right]
\end{equation}
for every $v\in V$ and $k\geq 0$. 
Observe that, whatever the value of $U_v$, we have the inclusions
\begin{equation}
\Bigl\{u \in L_{k+1}(v) : u \xleftrightarrow{L_{-\infty,k+1}(v)} v \Bigr\} \subseteq \Bigl\{u \in S_{t_0k,t_0(k+2)} : u \xleftrightarrow{ S_{-\infty, t_0(k+2)}(v)} v\Bigr\}
\end{equation}
and
\begin{equation}
\Bigl\{u \in S_{t_0k,t_0(k+2)} : u \xleftrightarrow{ S_{-\infty, t_0(k+2)}(v)} v\Bigr\}
\subseteq \Bigl\{u \in L_{k,k+2}(v) : u \xleftrightarrow{L_{-\infty,k+2}(v)} v \Bigr\}.
\end{equation}
It follows that
\begin{multline}
\inf_{x\in [0,1]}\E_p\left[ X_{k}^{-\infty,k+2}(v) + X_{k+1}^{-\infty,k+2}(v) + X_{k+2}^{-\infty,k+2}(v)\mid U_v = x\right] \\\geq \sup_{x\in [0,1]}\E_p\left[ X_{k+1}^{-\infty,k+1}(v) \mid U_v = x\right]
\end{multline}
for each $v\in V$ and $k\geq 0$. 
We can now deduce \eqref{eq:bestworstonevert} from this together with the estimate \eqref{eq:firstvisitdecompup} of \cref{lem:pathdecomposition}, which implies that
\begin{multline}
 \inf_{x\in [0,1]}\E_p\left[ X_{k}^{-\infty,k+2}(v) + X_{k+1}^{-\infty,k+2}(v) + X_{k+2}^{-\infty,k+2}(v) \mid U_v = x\right] \\\leq  \left[ E_p^{-\infty,2}(0)+E_p^{-\infty,2}(1) + E_p^{-\infty,2}(2) \right]\inf_{x\in [0,1]}\E_p\left[ X_{k}^{-\infty,k}(v)\mid U_v = x\right]. 
\end{multline}

Now, by a similar argument to that used in the proof of \cref{lem:probinfsup}, there exists $r_0$ such that for each two vertices $u$ and $v$ of $G$, there exists a path $u_0,\ldots,u_r$ in $G$ such that $r\leq r_0$,  $[u_0]=[u]$, $[u_r]=[v]$, and $\nlog \Delta(u_0,u_i)\geq 0$ for every $1 \leq i \leq r$. Let $k\geq r_0$. Since $u_r \in L_{m}$ for some $0\leq m \leq r_0$, it follows by the Harris-FKG inequality and the estimate \eqref{eq:firstvisitdecompup} of \cref{lem:pathdecomposition} that 
\begin{multline}
\E\left[X_k^{-\infty,k}(u) \mid U_u =x\right] \geq p^{r_0} \min_{m=0,\ldots,r_0}\E\left[X_{k-m}^{-\infty,k-m}(u_r) \mid U_u =x\right]\\
\geq p^{r_0} \left[\max_{m=0,\ldots,r_0} E_p^{-\infty,m}(m)\right]^{-1} \E\left[X_{k}^{-\infty,k}(u_r) \mid U_u =x\right],
\end{multline}
so that
\begin{equation}
\label{eq:smallvalues}
\inf_{x\in [0,1]}\E\left[X_k^{-\infty,k}(u) \mid U_u =x\right] \succeq_p \inf_{x\in [0,1]} \E\left[X_k^{-\infty,k}(v) \mid U_v =x\right]
\end{equation}
for every $u,v \in V$ and $k\geq r_0$. Small values of $k$ can then be handled by decreasing the implicit constant, so that in fact \eqref{eq:smallvalues} holds for every $k\geq 0$. 
The claimed inequality \eqref{eq:bestworstup} now follows from this estimate together with \eqref{eq:bestworstonevert}.
\end{proof}

\begin{proof}[Proof of \cref{lem:beta}]
It follows from the estimate \eqref{eq:firstvisitdecompup} of \cref{lem:pathdecomposition} that
\begin{equation}E^{-\infty,m+n}_p(m+n) \leq E^{-\infty,m}_p(m) E^{-\infty,n}_p(n)\end{equation}
and from the estimate \eqref{eq:lastvisitdecompdown} of \cref{lem:pathdecomposition} that
\begin{equation}E^{-\infty,0}_p(-m-n) \leq E^{-\infty,0}_p(-m) E^{-\infty,0}_p(-n)\end{equation}
for every $m,n\geq 0$.
 % and in particular that $E_p^{-\infty,n}(n)<\infty$ and $E_p^{-\infty,0}(-n)<\infty$ for every $n\geq 0$.
  Applying Fekete's Lemma and using that $p\in I_h$ we deduce that the limits
\begin{equation}
\beta_p := -\lim_{n\to+\infty} \frac{1}{n}\nlog E_p^{-\infty,n}(n) = -\inf_{n\geq1} \frac{1}{n}\nlog E^{-\infty,n}_p(n)
\end{equation}
and
\begin{equation}
\beta_p' := 1 -\lim_{n\to +\infty} \frac{1}{n}\nlog E^{-\infty,0}_p(-n) = 1-\inf_{n\geq1} \frac{1}{n}\nlog E^{-\infty,0}_p(-n)
\end{equation}
both exist and are not equal to $-\infty$. 

Let $h_p$ and $h_p'$ be the error terms
\begin{align}
\label{eq:hpdef}
h_p(n) = \nlog E_p^{-\infty,n}(n) + \beta_p n \qquad \text{and} \qquad
h'_p(n) = \nlog E_p^{-\infty,0}(-n) + (\beta_p'-1) n.
% \label{eq:hpprimedef}
\end{align}
It follows from the above discussion that 
$h_p(n)$ and $h'_p(n)$ are both subadditive, are both non-negative, and are both $o_p(n)$ as $n\to+\infty$. Moreover, \cref{lem:bestworst} implies that 
\begin{equation}
 \nexp\left[-\beta_p n + h_p(n+1) \right] \preceq_p \E_p\left[ X_n^{-\infty,n}(v) \mid U_v =x \right] \preceq_p \nexp\left[-\beta_p n + h_p(n) \right]
\label{eq:hplower}
 \end{equation}
 % \asymp_p \nexp\left[-\beta_p n + h_p(n) \right]
% \intertext{
for every $v\in V$, $x\in [0,1]$, and $n\geq 0$, and similarly that
%\begin
\begin{equation}
\nexp\left[-(\beta'_p-1) n + h'_p(n+1) \right] \preceq_p \E_p\left[ X_{-n}^{-\infty,0}(v) \mid U_v =x \right] 
\preceq_p \nexp\left[-(\beta'_p-1) n + h'_p(n) \right]
\label{eq:hpprimelower}
\end{equation}
for every $v\in V$, $x\in [0,1]$, and $n\geq 0$. 
On the other hand, the tilted mass-transport principle implies that
\begin{equation}
\label{eq:betaprimeMTP}
\E_p\left[X_n^{-\infty,n}(\rho)\right] \asymp \nexp(-n) \E_p\left[X_{-n}^{-\infty,0}(\rho)\right].
\end{equation}
Since $h_p(n)$ and $h'_p(n)$ are both $o_p(n)$ as $n\to\infty$, comparing  \eqref{eq:hplower} and \eqref{eq:hpprimelower} in light of \eqref{eq:betaprimeMTP} yields that $\beta_p=\beta_p'$. The result then follows from \eqref{eq:hplower} and \eqref{eq:hpprimelower}.
\end{proof}

\subsection{Slab intersections in the tiltable phase}
\label{subsec:tiltableslab}

We now restrict attention to the tiltable phase $0<p<p_t$, in which a sharper analysis is possible. We begin by moving from half-space first moment estimates to full-space first moment estimates. 
% (In fact, this is only possible in the tilted phase, as we show in the sequel to this paper.)

\begin{lemma}
\label{lemma:gammabeta}
Let $p\in I_h$, and let $h_p$ and $h_p'$ be defined as in \eqref{eq:hpdef}. If  $\beta_p>1/2$, then 
% there exists a constant $r$ such that
\begin{align*}
\begin{rcases} k \geq 0\quad&\nexp\left[-\beta_p k \right] \\
k<0\quad& \nexp\left[(\beta_p-1) k \right]
\end{rcases} \preceq_p \E_p\left[ X_k^{-\infty,\infty}(v) \mid U_v = x\right] &\preceq_p \begin{cases} \nexp\left[-\beta_p k + h_p(k)\right] &\quad k \geq0\\
\nexp\left[(\beta_p-1) k + h_p'(-k)\right] &\quad k < 0
\end{cases}
% \label{eq:gammabeta}
% \end{align}
% \intertext{and}
% % \begin{align}
% \E_p\left[ X_k^{-\infty,\infty}(v) \mid U_v = x\right] &\succeq_p \begin{cases} \nexp\left[-\beta_p k \right] &k \geq0\\
% \nexp\left[(\beta_p-1) k \right] &k <0
% \end{cases}
% \label{eq:gammabeta2}
\end{align*}
for every $v\in V$, $x\in [0,1]$, and $k \in \Z$.
In particular, $\chi_{p,\lambda}<\infty$ if and only if $\beta_p> \max\{\lambda,1-\lambda\}$, $p<p_c$ if and only if $\beta_p>1$, and $p<p_t$ if and only if $\beta_p>1/2$.
\end{lemma}

\begin{proof}
The lower bound  follows from \cref{lem:beta}. To obtain the upper bound,
% of \eqref{eq:gammabeta},
 we apply \cref{lem:beta} and the estimate \eqref{eq:extremedecompup} of \cref{lem:pathdecomposition} to deduce that
\begin{align}
\E_p\left[ X_k^{-\infty,\infty}(v) \mid U_v = x\right] & \leq \sum_{\ell \geq k \vee 0}
E_p^{-\infty,\ell}(\ell)E_p^{-\infty,0}(k-\ell)
\nonumber
\\
&=
 \sum_{\ell \geq k \vee 0} \nexp\left[ -\beta_p \ell + (\beta_p-1)(k-\ell) + h_p(\ell) + h_p'(\ell-k) \right].
 % &\leq \sum_{\ell \geq k \vee 0} \nexp\left[ -\beta_p \ell + (\beta_p-1)(k-\ell) + h_p(\ell) + h_p'(\ell-k) \right]
 \end{align}
%  If $k<0$, then using the subadditive inequality $h_p'(\ell-k)\leq h_p'(\ell) +h_p'(-k)$ and rearranging yields that
%  \begin{multline}
%  \label{eq:gammabetaproofneg}
% \E_p\left[ X_k^{-\infty,\infty}(v) \mid U_v = x\right]
%  \\\leq \left(\sum_{\ell \geq 0} \nexp\left[ -(2\beta_p-1) \ell +  h_p(\ell)+h_p'(\ell) \right]\right) \cdot
% \nexp\left[ (\beta_p-1) k + h_p'(-k)\right].
%  \end{multline}
%  If $k\geq 0$, 
Changing variables to $r = \ell - (k\vee 0)$, using the subadditivity of $h_p$ and $h_p'$, and rearranging yields that
\begin{multline}
\E_p\left[ X_k^{-\infty,\infty}(v) \mid U_v = x\right]
 \\\leq \left(\sum_{r \geq 0} \nexp\left[ -(2\beta_p-1) r +  h_p(r)+h_p'(r) \right]\right) \cdot
\begin{cases} \nexp\left[ -\beta_p k + h_p(k)\right] & k\geq0\\
\nexp\left[ (\beta_p-1) k + h_p'(-k)\right] & k <0.
\end{cases}
\end{multline}
When $\beta_p>1/2$ the prefactor on the right is finite, and since it does not depend on $k$, we deduce the claimed upper bound. 
The other claims follow immediately from this estimate together with \cref{prop:tiltedmeanfieldlowerbound} by noting that, by definition of the involved quantities, 
\begin{equation}
\label{eq:XtoChi}
\chi_{p,\lambda}(v) \asymp_{\lambda} \sum_{k\in \Z} \E_p\left[ X_k^{-\infty,\infty}(v)\right]\nexp(\lambda k)
\end{equation}
for every $v\in V$, $\lambda \in \R$ and $p\in [0,1]$. 
\end{proof}

% We next show that 

We next show that when $p<p_t$ the error terms $h_p$ and $h_p'$, along with the implicit error terms from \cref{lem:alpha},  are bounded from above. (We do not generally expect this to be the case when $p\geq p_t$.)
 % terms appearing above can be replaced with constants.

\begin{prop}[First moments in the tiltable phase]
\label{prop:tamedecay}
Let $0<p<p_t$. Then we have that
\begin{align}
\label{eq:tamefirstmoment}
% \P_p( v \xleftrightarrow H^+_t(v)) &\asymp_p \exp\left[-\alpha_p t + o_p(t)\right]\\
% \intertext{and}
\E_p\left[ X_{k}^{-\infty,\infty}(v) \mid U_v =x \right] \asymp_p \begin{cases} \nexp\left[-\beta_p k  \right] &k \geq0\\
\nexp\left[(\beta_p-1)k \right] & k < 0
\end{cases}
\end{align}
for every $v\in V$, $x\in [0,1]$, and $k \in \Z$. Moreover, we have that
\begin{align}
\P_p\left( v \xleftrightarrow{L_{0,\infty}(v)} L_k(v) \mid U_v = x \right) \asymp_p \P_p\left( v \xleftrightarrow{} L_k(v) \mid U_v = x \right) \asymp_p \nexp\left[-\beta_p k\right]
\label{eq:tameprobability}
\end{align}
for every $v\in V$, $x\in [0,1]$, and $k \geq 0$.
In particular, $\alpha_p=\beta_p$ for every $0<p<p_t$.
 % Moreover, both estimates hold locally uniformly in the tame phase.
\end{prop}

This immediately implies the following very useful corollary.

\begin{corollary}
\label{cor:lambdatoalpha}
\hspace{1cm}
\begin{enumerate}
\item 
$\alpha_{\pcl{\lambda}}\geq \max\{\lambda,1-\lambda\}$ for every $\lambda \in \R$.
\item If $\lambda \in \R$ is such that $\pcl{\lambda}<p_t$, then 
$
\alpha_{\pcl{\lambda}}=\beta_{\pcl{\lambda}}=\max\{\lambda,1-\lambda\}$. 
\end{enumerate}
% More generally,
% \[\alpha_{\pcl{\lambda}}\geq \max\{\lambda,1-\lambda\} \]
% for every $\lambda \in \R$.
\end{corollary}

\begin{proof} 
Markov's inequality implies that $\alpha_p \geq \beta_p$ for every $p\in I_h$.
By \cref{lemma:gammabeta} and \cref{prop:tamedecay}, $\alpha_p \geq \beta_p > \max\{\lambda,1-\lambda\}$ for all $0<p<\pcl{\lambda}$. Thus, the bound $\alpha_{\pcl{\lambda}}\geq \max\{\lambda,1-\lambda\}$ follows by left-continuity of $\alpha$ (\cref{lem:alphacontinuity}). On the other hand, \cref{lemma:gammabeta} implies that $\beta_{\pcl{\lambda}}\leq \max\{\lambda,1-\lambda\}$, so that if $\pcl{\lambda}<p_t$ then $\alpha_{\pcl{\lambda}}=\beta_{\pcl{\lambda}}=\max\{\lambda,1-\lambda\}$ by \cref{prop:tamedecay}.
\end{proof}

We will require another simple inequality that follows by a standard application of the BK inequality. It is related to the tree-graph inequalities of Aizenman and Newman \cite{MR762034}.
% \begin{corollary}

% \end{corollary}

\begin{lemma}
\label{lem:BKsecondmoment}
The estimate
\begin{align*}
\E_p\left[X^{m,n}_k(v) X^{m,n}_\ell(v) \mid U_v=x \right] \leq \sum_{i=m}
^n E^{m,n}_{p}(i) E^{m-i,n-i}_{p}(k-i) E^{m-i,n-i}_{p}(\ell-i) 
% \sup_{v\in V, x \in [0,1]}\E_p\left[X^{m,n}_i(v) \mid U_v =x \right] \sup_{v\in V, x \in [0,1]} \E_p\left[ X^{m-i,n-i}_{k-i}(v) \mid U_v=x\right] \sup_{v\in V, x \in [0,1]} \E_p\left[ X^{m-i,n-i}_{\ell-i}(v) \mid U_v=x\right]. 
\end{align*}
holds for every $p\in [0,1]$, $v\in V$, $x\in [0,1]$, and $-\infty \leq m \leq k,\ell \leq n \leq \infty$.
\end{lemma}

\begin{proof}
If $u \in L_k(v)$ and $w\in L_\ell(v)$ are such that $v$ is connected to both $u$ and $w$ by open paths in $L_{m,n}(v)$, then there must exist a vertex $z \in L_{m,n}(v)$ (possibly equal to one of $u,v,$ or $w$) such that $\{v \xleftrightarrow{L_{m,n}(v)} z \}\circ \{z \xleftrightarrow{L_{m,n}(v)} u \}\circ \{z \xleftrightarrow{L_{m,n}(v)} w \}$ occurs. Applying the BK inequality and summing over the possible choices of $u,w,$ and $z$ yields the claimed inequality.
\end{proof}

We are now ready to prove \cref{prop:tamedecay}.

\begin{proof}[Proof of \cref{prop:tamedecay}]

We begin with \eqref{eq:tamefirstmoment}. By \cref{lem:beta,lemma:gammabeta} it suffices to show that $h_p(n)$ and $h'_p(n)$ are both $O_p(1)$. 
% The estimates \eqref{eq:hplower}, \eqref{eq:hpprimelower} and \eqref{eq:betaprimeMTP} imply that  $|h_p(n)-h_p'(n)|=O_p(1)$ as $n\to+\infty$, and so it suffices to prove that $h_p(n)=O_p(1)$.
% We first prove this for $h_p(n)$.
The estimates \eqref{eq:hplower}, \eqref{eq:hpprimelower}, and \eqref{eq:betaprimeMTP} imply that $h'_p(n) - h_p(n-1)$ is bounded from above by a $p$-dependent constant, and since both $h_p(n)$ and $h_p'(n)$ are non-negative it suffices to prove that $h_p(n)=O_p(1)$. 
Applying \cref{lem:BKsecondmoment} and \cref{lemma:gammabeta}, we obtain that
\begin{align}
\E_p\left[ \left(X_k^{0,k}(v)\right)^2 \mid U_v=x\right] &\leq \sum_{\ell=0}^k E_p^{-\infty,\infty}(k-\ell) \left(E_p^{-\infty,\infty}(\ell)\right)^2
\nonumber
\\
&\preceq_p \sum_{\ell=0}^k \nexp\left[-\beta_p (k-\ell) -2\beta_p \ell   + h_p(k-\ell)  + 2h_p(\ell) \right]
\nonumber
\\
&\leq \left(\sum_{\ell=0}^\infty \nexp\left[-\beta_p \ell + 2h_p(\ell) \right]\right) \cdot \nexp\left[ - \beta_p k + \max_{0\leq \ell \leq k} h_p(\ell)\right]
\end{align}
for all $v\in V$,  $x\in [0,1]$, and $k\geq 0$.
Since $h_p(\ell)=o_p(\ell)$, the sum in the prefactor on the last line is finite, and since it does not depend on $k$ we obtain that
\begin{equation}
\label{eq:hpsecondmoment}
\E_p\left[ \left(X_k^{0,k}(v)\right)^2 \mid U_v=x\right] \preceq_p \nexp\left[ - \beta_p k + \max_{0\leq \ell \leq k} h_p(\ell)\right]
\end{equation}
for all $v\in V$, $x\in [0,1]$, and $k\geq 0$. 
On the other hand, applying the last-visit path-decomposition inequality \eqref{eq:lastvisitdecompup} of \cref{lem:pathdecomposition} implies that
\begin{align}
\label{eq:hpfirstmoment}
E_p^{0,k}(k) \geq \frac{E_p^{-\infty,k}(k)}{E_p^{-\infty,\infty}(0)} \succeq_p \nexp\left[ -\beta_p k + h_p(k)\right].
\end{align}

Recall that the Cauchy-Schwarz inequality implies that $\P(Y>0) \geq \E[Y]^2/\E[Y^2]$ for any non-negative random variable $Y$. Applying this to $X_k^{0,k}(v)$ conditional on $U_v=x$ yields that
% The Cauchy-Schwarz inequality implies that
\begin{equation}
\P_p\big( v \xleftrightarrow{L_{0,\infty}(v)} L_k(v) \mid U_v =x \big)
\geq \E_p\left[ X_k^{0,k}(v) \mid U_v = x\right]^2
	\E_p\left[ \left(X_k^{0,k}(v)\right)^2 \mid U_v = x\right]^{-1}.
\end{equation}
Taking suprema over $v \in V$ and $x\in [0,1]$ and applying \eqref{eq:hpsecondmoment} and \eqref{eq:hpfirstmoment} we obtain that
\begin{equation}
\label{eq:tfmprob1}
\sup_{v\in V, x\in [0,1]} \P_p\big( v \xleftrightarrow{L_{0,\infty}(v)} L_k(v) \mid U_v =x \big) \succeq_p \nexp\left[-\beta_p k + 2h_p(k) - \max_{0\leq \ell \leq k} h_p(\ell) \right]
\end{equation}
for every $k\geq 0$.
On the other hand, since $\alpha_p \geq \beta_p$ by Markov's inequality, it follows from \cref{lem:alpha} that 
\begin{equation}
\label{eq:tfmprob2}
\P_p\big(v \xleftrightarrow{L_{0,\infty}(v)} L_k(v) \mid U_v = x\big) \preceq_p \nexp\left[ -\beta_p k\right]
\end{equation}
for every $v\in V$ and $x\in [0,1]$. 
Comparing \eqref{eq:tfmprob1} and \eqref{eq:tfmprob2} at those values of $k$ for which $h_p(k)=\max_{0\leq \ell \leq k} h_p(\ell)$ yields that $h_p(k)=O_p(1)$ as claimed. 
The estimates \eqref{eq:tameprobability} then follow from \eqref{eq:tfmprob1} (for the lower bounds) and Markov's inequality applied to \eqref{eq:tamefirstmoment} (for the upper bounds). 
\end{proof}

Applying \cref{prop:tamedecay} and \cref{lem:BKsecondmoment}, we immediately obtain the following. Similar estimates hold for $k,\ell \in \Z$, but we shall not require these.

\begin{prop}
% [Second moment analysis in the tiltable phase]
\label{prop:tamesecondmoment}
 If $0<p<p_t$ is such that $\beta_p > 2/3$ then
\begin{equation}
\label{eq:tamesecondmoment}
\E_p\left[ \left(X_{-k}^{-\infty,\infty}(v)\right)\left(X_{-\ell}^{-\infty,\infty}(v)\right) \mid U_v =x \right] \preceq_p
\begin{cases} \nexp\left[-(\beta_p-1)(k \vee \ell)\right] & \beta_p > 1\\ 
k \wedge \ell &\beta_p =1\\
\nexp\left[ (1-\beta_p)(k+\ell)\right] &\beta_p < 1
\end{cases}
\end{equation}
for every $v\in V$, $x\in [0,1]$, and $k,\ell \geq 0$.
Similarly, if $0<p<p_t$ is such that $1/2<\beta_p \leq 2/3$ then
\begin{equation}
\E_p\left[ \left(X_{-k}^{-\infty,0}(v)\right)\left(X_{-\ell}^{-\infty,0}(v)\right) \mid U_v=x \right] \preceq_p \nexp\left[ (1-\beta_p)(k+\ell)\right]
\label{eq:tamesecondmoment2/3}
\end{equation}
for every $v\in V$, $x\in [0,1]$, and $k,\ell \geq 0$.
\end{prop}

We remark that \cref{prop:tamesecondmoment} also easily yields the following interesting result on the quenched growth rate of infinite clusters in the supercritical tiltable regime.
This result is not required for the proofs of the main theorems.

\begin{corollary}[Growth of infinite clusters in the tiltable supercritical regime]
\label{cor:clustergrowth}
 If $0<p<p_t$ is such that $1/2<\beta_p<1$, then for every $v\in V$ we have that
\begin{equation}
\label{eq:quenchedgrowth}
\limsup_{k\to\infty} \frac{1}{k} \nlog X_{-k}^{-\infty,\infty}(v) = 1-\beta_p
\end{equation}
almost surely on the event that the cluster of $v$ is infinite. 
\end{corollary}
\begin{proof}
The Paley-Zygmund inequality \cite{paley1932note} implies that 
\begin{equation}\P\left(X_{-k}^{-\infty,0}(v) \geq \frac{1}{2}\E_p\left[X_{-k}^{-\infty,0}(v) \right]\right) 
\geq \frac{\E_p\left[X_{-k}^{-\infty,0}(v) \right]^2}{4\E_p\left[\left(X_{-k}^{-\infty,0}(v)\right)^2 \right]}.
\end{equation}
\cref{prop:tamedecay} and \cref{prop:tamesecondmoment} imply that the right hand side is bounded below by a positive constant depending on $p$. 
 % with positive probability for each $v\in V$.
 Thus, it follows by Fatou's Lemma that 
 \begin{equation}
\limsup_{k\to\infty} \frac{1}{k} \nlog X_{-k}^{-\infty,\infty}(v) \geq 1-\beta_p
 \end{equation}
 with positive probability. The indistinguishability theorem of
 Haggstr\"om, Peres, and Schonmann \cite{HPS99} implies that in fact this inequality must hold almost surely on the event that the cluster of $v$ is infinite. 
  On the other hand, the reverse inequality holds almost surely by \cref{lem:beta}, Markov's inequality and the Borel-Cantelli lemma. 
\end{proof}

% \begin{remark}
% By analogy with the Kesten-Stigum Theorem from the theory of supercritical branching processes \cite{KestenStigum66,LPP95,KurtzLyonsPemantlePeres94}, it is reasonable to expect that the following holds: If $p\in (0,1]$ is such that $1/2<\beta_p<1$, then on the event that the cluster containing $v$ is infinite, the normalized sequence
% \[
% \nexp\left[-(1-\beta_p)k\right] X_{-k}^{-\infty,\infty}(v)
% \]
% converges to a non-degenerate random variable that is almost surely positive. We do not pursue this further here. It would also be interesting to investigate the possible behaviours of $X_k^{-\infty,\infty}$ when $p_t \leq p \leq p_h$.
% \end{remark}

\subsection{The view from a peak in the subcritical phase}
\label{subsec:subcritical_peak}
% Let $G$ be a connected, locally finite graph, let $\Gamma \subseteq \Aut(G)$ be quasi-transitive and nonunimodular.
 We say that a vertex $v$ is the \textbf{peak} (a.k.a.\ unique highest point) of its cluster $K(v)$ if $\nlog \Delta(v,u) < 0$ for every $u \in K(v)\setminus \{v\}$. In particular, each cluster has at most one peak. We write $\peak(v)$ for the peak of $v$'s cluster (when it exists) and $\sP_v$ for the event that $v$ is the peak of its cluster, i.e., $v=\peak(v)$. 

In this subsection, we apply \cref{prop:tamesecondmoment} to study the probability that a vertex is a peak and its cluster survives for $k$ levels in the subcritical regime. 
We begin with the following slight strengthening of the estimate \eqref{eq:tameprobability}. Note that the statement `$\peak(v)\in L_k(v)$' implicitly includes the statement that the cluster of $v$ has a peak.  

\begin{lemma}
 If $0<p<p_t$  then
\begin{align}
\label{eq:highpeaktame}
\P_p\left(\Peak(v) \in L_k(v) \mid U_v=x\right) \asymp_p \nexp\left[-\beta_p k\right]
\end{align}
for every $v\in V$, $x\in [0,1]$, and $k\geq0$. 
\end{lemma}

\begin{proof}
It suffices to prove the lower bound, as the upper bound follows from \cref{prop:tamedecay}. 
As in the proof of \cref{lem:probinfsup}, there exists $n_0 \geq 1$ such that for every vertex $v\in V$, there exists $n \leq n_0$ and a path $v=v_0,v_1,\ldots,v_{n}$ in $G$ such that $v_n \in L_1(v)$, $\nlog \Delta(v_i,v_n) >0$ for every $0 \leq i < n$, and the sequence $\nlog\Delta(v,v_i)$ is (weakly) increasing. 

 Fix $v\in V$, $x\in [0,1]$ and $k\geq 0$. 
 % Let $Z_k$ be the set of vertices in $L_k(v)$ that are connected to $v$ by an open path whose edges all have at least one endpoint in $L_{-\infty,k-1}(v)$. 
 Let $H_k$ be the subgraph of $G$ spanned by those edges with at least one endpoint in $L_{-\infty,k-1}(v)$, and let $K_k$ be the cluster of $v$ in $H_k$. That is, $K_k$ is the connected component of $v$ in the subgraph of $H_k$ spanned by the open edges of $G[p]$. Let $Z_k=K_k \cap L_k$. 
 % subgraph of $G$ consisting of those vertices of $L_{-\infty,k}$ that can be reached from $v$ by an open path each edge of which has at least one endpoint in $L_{-\infty,k-1}$ and all the open edges between these vertices, and let 
 % intersection of the cluster of $v$ with the subgraph of $G$ spanned by those edges with at least one endpoint in $L_{-\infty,k-1}(v)$, and let $Z_k=K_k \cap L_k(v)$.
% vertices in $L_{-\infty,k}(v)$ that are connected to $v$ by an open path of edges each of which has at least one endpoint in $L_{-\infty,k-1}(v)$, and let $Z_k=K_k \cap L_k(v)$. 
  It follows from \cref{prop:tamedecay,prop:tamesecondmoment} that
\begin{equation}
\E_p\Bigl[ |Z_k| \mid U_v=x,\, v \leftrightarrow L_k(v) \Bigr] \leq \E_p\left[ X_k^{-\infty,\infty}(v) \mid U_v=x,\, v \leftrightarrow L_k(v) \right] \preceq_p 1
\end{equation}
% for every $v\in V$, $x\in [0,1]$, $k\geq 0$ and $0<p<p_t$. 
for each $0<p<p_t$. 
Condition on $K_k$, and suppose that $Z_k \geq 1$. Pick a vertex $u \in Z_k$ and a path $u=u_0,\ldots,u_{n}$ with $n \leq n_0$ as above. If at least one edge connecting $u_i$ to $u_{i+1}$ is open for every $0\leq i < n$, and every other edge incident to $Z_k$ or $\{u_1,\ldots,u_{n_0}\}$ is either closed or lies in $K_k$, then $\peak(v) \in L_{k+1}(v)$. Lower bounding the conditional probability of this event yields that
\begin{equation}
\P_p\Bigl(\peak(v) \in L_{k+1}(v) \mid K_k, U_v=x\Bigr) \geq \mathbbm{1}\bigl(Z_k \geq 1 \bigr) p^{n_0} (1-p)^{C(Z_k + n_0)}.
\end{equation}
where $C = \max_{v\in V} \deg(v)$. Taking expectations and applying Jensen's inequality yields that
\begin{equation}
\P_p\Bigl(\peak(v) \in L_{k+1}(v) \mid U_v=x,\, v \leftrightarrow L_k(v)\Bigr) \succeq_p 1,
\end{equation}
and the claim follows from \cref{prop:tamedecay}.
\end{proof}

Applying the tilted mass-transport principle to \eqref{eq:highpeaktame} yields that
\begin{equation}
\label{eq:MTPPeakBasic}
\E_p\left[ X_{-k}^{-\infty,\infty}(\rho) \,;\, \sP_\rho \right] \asymp \nexp(k) \P_p\Bigl( \Peak(\rho) \in L_k \Bigr) \asymp_p \nexp\left[ -(\beta_p-1) k\right]
\end{equation}
for every $0<p<p_t$ and $k\geq 0$, and it follows by a straightforward finite-energy argument that there exists $v_0 \in V$ such that
\begin{equation}
\label{eq:somepeak}
\E_p\left[ X_{-k}^{-\infty,\infty}(v_0) \mid  U_{v_0}=x, \sP_{v_0} \right] \asymp_p \nexp\left[ -(\beta_p-1) k\right]
\end{equation}
for every $x\in [0,1]$ and $k\geq 0$. (This estimate might not hold for every $v\in V$. Indeed, some vertices may have the same height as all their neighbours, in which case they cannot be the peak of a non-singleton cluster.)
 The next lemma gives a similar analysis for connection \emph{probabilities} in the subcritical phase.
% We next show that for subcritical $p$, there exist vertices for which the probability of connecting 
% the local structure of $\Delta$ may prevent this estimate holding for every $v\in V$.

\begin{lemma}[Subcritical peak survival]
\label{lem:subcritpeaksurvival}
% \begin{itemize}
	% \item (Subcritical) If $0<p<p_c$ then
	% \item (Critical) If 
	% \item
	If $p<p_c$ then there exists $v_0\in V$ such that
\begin{equation}
\label{eq:subcritpeaksurvivalprob}
\P_p\left(v_0 \leftrightarrow L_{-k}(v_0) \mid U_{v_0}=x \right) \asymp_p \P_p\left(\{v_0 \leftrightarrow L_{-k}(v_0)\} \cap \sP_{v_0} \mid U_{v_0}=x \right) \asymp_p 
\nexp\left[-(\beta_p-1)k \right]
\end{equation}
for every $x\in [0,1]$ and $k \geq 0$, and
% Moreover, 
\begin{equation}
\label{eq:subcritpeaksurvivalcorrelation}
\E_p\left[ X_{-\ell}^{-\infty,\infty}(v_0) \mid  U_{v_0}=x,\, v_0 \leftrightarrow L_{-k}(v_0),\, \sP_{v_0} \right] \preceq_p \begin{cases} 1 & \ell \leq k\\
\nexp\left[ -(\beta_p-1)(\ell-k) \right] & \ell > k
\end{cases}
\end{equation}
for every $x\in [0,1]$ and $\ell, k \geq 0$.
% for every $v\in V$ and $0 < p \leq p_c$.
% \end{itemize}
\end{lemma}

\begin{proof}
The upper bounds of \eqref{eq:subcritpeaksurvivalprob} follow from \cref{prop:tamedecay} and Markov's inequality. The lower bounds of \eqref{eq:subcritpeaksurvivalprob} follow from \cref{prop:tamesecondmoment}, the estimate \eqref{eq:somepeak}, and the Cauchy-Schwarz inequality. For \eqref{eq:subcritpeaksurvivalcorrelation}, we have by \cref{prop:tamesecondmoment} that
\begin{align*}
\E_p\left[ X_{-\ell}^{-\infty,\infty}(v_0) \mathbbm{1}\left(v_0 \leftrightarrow L_{-k}(v_0),\, \sP_{v_0}\right)  \mid U_{v_0}=x \right] &\leq \E_p\left[ X_{-\ell}^{-\infty,\infty}(v_0) X_{-k}^{-\infty,\infty}(v_0) \mid U_{v_0}=x \right]\\ &\preceq_p \nexp\left[-(\beta_p-1)(k\vee \ell) \right],
\end{align*}
so that the conditional expectation estimate \eqref{eq:subcritpeaksurvivalcorrelation} follows from \eqref{eq:subcritpeaksurvivalprob}.
\end{proof}

\cref{lem:subcritpeaksurvival} should be compared with analogous estimates for subcritical branching processes.

\section{The critical point is tiltable}
\label{sec:pc}

In this section we prove \cref{thm:pcpt}, which immediately implies \cref{thm:pcph,thm:pcpu}. We begin by highlighting the following special case of \cref{cor:lambdatoalpha}, which is similar to the observation powering the proof of \cite{Hutchcroft2016944}.

\begin{lemma}
\label{lem:alphapc} $\alpha_{p_c}\geq 1$.
\end{lemma}

\begin{proof}
This follows immediately from the sharpness of the phase transition, which implies that $\pcl{0}=\pcl{1}=p_c$, together with \cref{cor:lambdatoalpha}.
\end{proof}

\cref{thm:pcpt} follows easily from \cref{prop:tiledAizBar} together with the following proposition, which the remainder of this section is devoted to proving.
% the following two propositions. The second, \cref{prop:tiledAizBar}, is proven in \cref{subec:AizBar}, while the second, \cref{lem:bootstrap} is proven in \cref{subsec:peak_comparison,subsec:bootstrap}.

\begin{prop}
\label{lem:bootstrap}
The estimate
	\begin{equation}
	\label{eq:bsstep6}
	\E_{p_c}\left[ \left(X_{k}^{-\infty,\infty}(v) \right)^{1-\eps} \right] 
	\preceq_{\eps}
	\begin{cases} \nexp\left[-k+o_\eps(k)\right] &k \geq 0\\
	\nexp\left[o_\eps(k)\right] &k \leq 0\\
	\end{cases}
	\end{equation}
	holds for every $v\in V$, $0<\eps \leq 1$ and $k\in \Z$.
\end{prop}

\begin{proof}[Proof of \cref{thm:pcpt} given \cref{lem:bootstrap}]
Suppose for contradiction that $p_c=p_c(\lambda)$ for some $\lambda \in (0,1/2)$, and fix one such choice of $\lambda$. 
% Then we have that $p_c=\pcl{\lambda}$ for all $\lambda \in [0,1]$, and in particular that $p_c=\pcl{1/4}$. 
Then we have that
\begin{multline}
\E_{\pcl{\lambda}}\left[|K_v|_{v,\lambda}^{3/4}\right] \preceq 
\E_{\pcl{\lambda}}\left[\left(\sum_{k\in \Z} \nexp\left[\lambda k\right] X_k^{-\infty,\infty}(v)\right)^{3/4}\right] \hspace{-0.6em}\\
\leq
\sum_{k \in \Z} \nexp\left[\frac{3}{4}\lambda k\right] \E_{\pcl{\lambda}}\left[\left(X_k^{-\infty,\infty}(v)\right)^{3/4}\right],
\end{multline}
so that applying \cref{lem:bootstrap} and the assumption that $p_c=\pcl{\lambda}$ we deduce that
% &\preceq 
\begin{align}
\E_{\pcl{\lambda}}\left[|K_v|_{v,\lambda}^{3/4}\right] &\preceq \sum_{k \geq 0} \nexp\left[\frac{3}{4} \lambda k - k +o(k)\right]
+
\sum_{k < 0} \nexp\left[\frac{3}{4}\lambda k  +o(k)\right] <\infty.
\end{align}
This contradicts \cref{prop:tiledAizBar}. 
\end{proof}

\paragraph{Proof overview.} Let us now briefly outline the strategy by which we will prove \cref{lem:bootstrap}. The argument is much simpler in the case that $\Gamma$ is transitive and $(G,\Gamma)$ has simple layers, as defined in \cref{subsec:layersdef}, and we restrict to this case for the purposes of this overview. We recall that, in this setting, the root $\rho$ and the uniform separating layers $(L_n)_{n\in \Z}$ can be taken to be deterministic, so that the percolation configuration $G[p]$ is the only source of randomness: to emphasize this fact we will 
return to using the notation $\bP_p$ and $\bE_p$ when dealing with this special case. 
 Recall that $\sP_v$ denotes the event that $v$ is the peak (unique highest point) of its cluster, 
and  define $\sP_v(k)$  to be the event that $v$ is the peak of the set of vertices that are connected to $v$ by an open path in $L_{-k,\infty}(v)$.
Similarly to \eqref{eq:MTPPeakBasic}, the tilted mass-transport principle implies that
\begin{equation}
\label{eq:PeakMTP_overview}
\bE_{p_c}\left[ X_{-k}^{-k,\infty}(v) \,;\, \sP_v(k) \right] \preceq \nexp(k) \bP_{p_c}\Bigl( v \xleftrightarrow{L_{0,\infty}(v)} L_k(v) \Bigr) \preceq \nexp\left[ -(\alpha_{p_c}-1) k\right] \preceq 1
\end{equation}
for every $k\geq 0$, where we have used \cref{lem:alpha} in the second inequality and \cref{lem:alphapc} in the third. To proceed, we would ideally like to remove the restriction to the event $\sP_\rho(k)$ from the left hand side.
Unfortunately, we did not find any way to do this directly. Instead, we first use an exploration argument with Reimer's inequality to relate the expectation on the left hand side to the expectation of a similar quantity restricted to the (very likely) event that $v$ is not connected to a high layer. This is done in \cref{subsec:peak_comparison}, where we obtain that (under the simplifying assumptions above)
\begin{equation}
\label{eq:simple_peak_comparison_overview}
\bE_{p_c}\left[X_{-k}^{-k-r,0}(v) \,; v \nxleftrightarrow{L_{-k-r,\infty}(v)} L_{\ell}(v) \right]\\
\leq \myfrac[0.4em]{\bE_{p_c}\left[ X_{-k-\ell}^{-k-\ell-r,0}(v) \,\mid \sP_v(k+\ell+r) \right]}{\bP_{p_c}\left(v \leftrightarrow L_{-\ell}(v) \mid \sP_v(k+\ell+r) \right)}
\end{equation}
for every $k,r,\ell \geq 0$. The bound in the general case is more complicated and is given in \cref{lem:peakcomparison}. 

In order to apply this bound, we then prove a lower bound on the denominator on the right hand side when $p=p_c$. This is done in \cref{lem:peaksurvival}, which extends \cref{lem:subcritpeaksurvival} to the critical case. Combining this estimate with \eqref{eq:PeakMTP_overview}  and \eqref{eq:simple_peak_comparison_overview}  yields that
\begin{equation}
\label{eq:simple_peak_comparison_overview2}
\bE_{p_c}\left[X_{-k}^{-k-r,0}(v) \,; v \nxleftrightarrow{L_{-k-r,\infty}(v)} L_{\ell}(v) \right]\\
\preceq_p \nexp\bigl[-\alpha_{p_c}k +O(r) + o(\ell)\bigr]
\end{equation}
for every $k,r,\ell \geq 0$. On the other hand, \cref{lem:probinfsup} and \cref{lem:alphapc} imply that
\begin{equation}
\label{eq:simple_peak_comparison_overview3}
\bP_{p_c}\left(v \xleftrightarrow{L_{-k-r,\infty}(v)} L_\ell(v) \right) \preceq \nexp\left[-\alpha_{p_c}\ell +O(r+k)\right] \preceq \nexp\left[-\ell +O(r+k)\right].
\end{equation}
Putting \eqref{eq:simple_peak_comparison_overview2} and \eqref{eq:simple_peak_comparison_overview3} together, we obtain via elementary analysis (\cref{lem:startingpoint}) that
\begin{equation}
\label{eq:simple_peak_comparison_overview4}
\bE_{p_c}\left[ \left(X_{-k}^{-k-r,0}(v)\right)^{1-\eps}\right]
\preceq_\eps \nexp\left[ o_\eps(k) +O(r)\right]
\end{equation}
for every $\eps>0$ and $k,r \geq 0$.

To finish the proof, we bootstrap from the slab fractional moment estimate \eqref{eq:simple_peak_comparison_overview4} to the full-space fractional moment estimate claimed in \cref{lem:bootstrap}. At an intuitive level, the ideas used to do this are similar to those used in \cref{subsec:decomps,subsec:expdecay,subsec:tiltableslab}. However, substantial technicalities arise since the BK inequality and the tilted mass-transport principle are much less well-suited to dealing with fractional moments than with first moments. We develop the tools used to carry out this analysis in \cref{subsec:toolkit}, and perform the analysis itself in \cref{subsec:bootstrap}.

\subsection{The peak-comparison estimate}
\label{subsec:peak_comparison}

% The goal of this subsection is to prove the following estimate, which is like \cref{lem:bootstrap} except that it applies to the cluster restricted to a slab. In the following subsection we deduce \cref{lem:bootstrap} from \cref{lem:startingpoint} via a bootstrapping procedure.

% \cref{lem:startingpoint} will follow by combining the following two lemmas. We define $\sP_v(k)$ to be the event that $v$ is the peak of the set of vertices that are connected to $v$ by an open path in $H^+_{-k}(v)$.

The goal of this section is to implement the first step of the strategy outlined above, namely, to enlarge the event in the restricted expectation in \eqref{eq:simple_peak_comparison_overview}. Since the proof in the general case is rather technical, we begin by stating and proving the following special case. We define $\sP_v(k)$  to be the event that $v$ is the peak of the set of vertices that are connected to $v$ by an open path in $L_{-k,\infty}(v)$.

\begin{lemma}[Peak-comparison estimate, simplified]
\label{lem:peakcomparisonsimple}
Suppose that $\Gamma$ is transitive and that $(G,\Gamma)$ has simple layers. Then 
 % the estimate
\begin{equation}  
\bE_p\left[X_{-k}^{-k-r,0}(v) \,; v \nxleftrightarrow{L_{-k-r,\infty}(v)} L_{\ell}(v) \right]
\leq \myfrac[0.4em]{ \bE_p\left[ X_{-k-\ell}^{-k-\ell-r,0}(v) \,\mid \sP_v(k+\ell+r) \right]}{\bP_p\left(v \leftrightarrow L_{-\ell}(v) \mid \sP_v(k+\ell+r) \right)}
% \label{eq:startingpoint}
\end{equation}
% holds for every $r\geq 0$.
for every $v\in V$, $k,\ell,r\geq 0$, and $0<p \leq p_c$.
\end{lemma}

(The condition that $p \leq p_c$ is not really necessary, but slightly simplifies some details.)

\begin{figure}[t]
\centering
\includegraphics[height=0.3\textwidth]{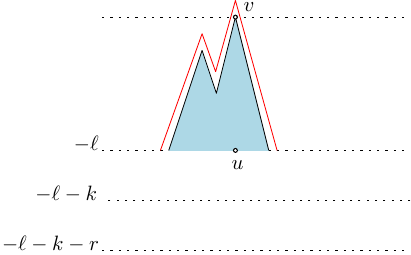} \hspace{1cm} \includegraphics[height=0.3\textwidth]{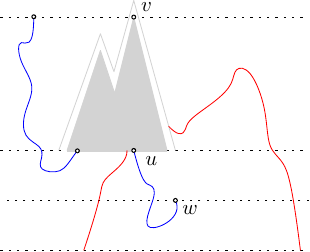}
\caption{Schematic illustration of the proof of the peak-comparison estimate in the simplified case. Left: We condition on the set of vertices that can be reached from $v$ by an open path contained in the complement of the set of edges with both endpoints in the lower half-space $L_{-\infty,-\ell}$, together with the open edges connecting them, and also condition on the event that $v$ is the peak of this cluster. The blue shaded region represents this cluster, while the red curve represents the edges in the boundary of the cluster that are known to be closed.
Right: Pick some vertex $u$ in the lower boundary of the revealed cluster, as well as some vertex $w$ in $L_{-\ell-k}$. If they all occur, then the events that $u$ is connected to $w$ in $L_{-\ell-k-r,-\ell}$, that $u$ is not connected to $L_0$ in $L_{-r-\ell-k,\infty}$ off the revealed set, and that $v$ is not the peak of its connected component in $L_{-\ell-k-r,\infty}$ must occur disjointly from each other (on the smaller probability space not including the status of edges we revealed in the first step, shown in grey). This puts us in a situation to apply Reimer's inequality, from which \cref{lem:peakcomparison} can be deduced.
}
\label{fig:Reimer}
\end{figure}

\begin{proof}[Proof of \cref{lem:peakcomparisonsimple}]
Fix $k,\ell,r\geq 0$, $0< p \leq p_c$ and $v\in V$. Since  $p \leq p_c$, every cluster is finite almost surely \cite{timar2006percolation,Hutchcroft2016944}.
 To lighten notation, we write $L_n=L_n(v)$ and $L_{m,n}=L_{m,n}(v)$. Let $A$ be a set of vertices in $L_{-\ell,\infty}$, and let $B$ be a set of edges none of which has both endpoints in $L_{-\infty,-\ell}$. Consider the event
\[
\sB(A,B) = \Bigl\{A \xleftrightarrow{L_{-r-k-\ell,\infty}} L_0 \text{ off $B$} \Bigr\},
\]
i.e., that $A$ is connected to $L_0$ by an open path in $L_{-r-k-\ell,\infty}$ that does not include any edges of the set $B$.
Let $u \in L_{-\ell}$, let $w \in L_{-\ell-k}=L_{-k}(u)$, and let $\sE(u,B)$ be the event 
\[\sE(u,B) = \big\{ u \xleftrightarrow{L_{-k-\ell-r,\infty}} L_0 \text{ off $B$}\big\}^c.\]
 We claim that 
 % and 
\begin{equation}
 \{ u \xleftrightarrow{L_{-k-\ell-r,-\ell}} w \} 
\cap 
\sE(u,B)
\cap
\sB(A,B) \\= \left( \{ u \xleftrightarrow{L_{-k-\ell-r,-\ell}} w \} 
\cap 
\sE(u,B) \right) \circ \sB(A,B).
\end{equation}
Indeed, suppose that the event on the left-hand side holds. Then there exists a set of open edges  $W_1$ in $L_{-k-\ell-r,-\ell}$ (which is necessarily disjoint from $B$) that form a path connecting $u$ to $w$, and a set of open edges $W_2$ in $L_{-k-\ell-r,\infty}$ that is disjoint from $B$ and forms a path from $A$ to $L_0$. $W_2$ must be disjoint from $W_1$, since otherwise  $u$ would be connected to $L_0$ in $L_{-k-\ell-r,\infty}$ off of $B$ and the event $\cE(u,B)$ would not occur.  
% Moreover, \cite[Lemma 5.2]{timar2006percolation} implies that the cluster of $u$ in $L_{-k-\ell-r} \setminus B$ is almost surely finite on the event that 
 Thus, the union of the set $W_1$ and the set of all the closed edges touching cluster of $u$ is a finite witness for $\{ u \xleftrightarrow{L_{-k-\ell-r,-\ell}} w\} 
\cap 
\sE(u,B)$, while  the set $W_2$ is a finite witness for $\sB(A,B)$ disjoint from this set. This yields the claimed equality of events.
Applying Reimer's inequality, we obtain that 
\begin{multline}
\bP_p\left(
\{ u \xleftrightarrow{L_{-k-\ell-r,-\ell}} w \} 
\cap 
\sE(u,B)
\cap
\sB(A,B) \right) \leq
\\ \bP_p\left( 
 \{ u \xleftrightarrow{L_{-k-\ell-r,-\ell}} w \} 
\cap 
\sE(u,B) \right)
\bP_p\left(
\sB(A,B)\right).
\label{eq:Reimer_exploration_simplified}
\end{multline}
% (Again, we refer the reader to \cite{arratia2015van} for a justification of  Reimer's inequality in infinite volume.)

Let $K'$ be the set of vertices that are connected to $v$ by a path consisting of open edges none of which have both endpoints in $L_{-\infty,-\ell}$, and let $\overline{K}'$ be the set of edges have at least one endpoint in $K'$ and do not have both endpoints in $L_{-\infty,-\ell}$. (Note that $\overline{K}'$ is determined by $K'$.) 
Let $\cK'$ be the set of pairs $(A,B)$ such that $\peak(A)=v$, $A \cap L_{-\ell} \neq \emptyset$, and the event $\{ K' = A, \overline{K}' =B \}$ has positive probability, and fix a pair $(A,B) \in \cK'$. Observe that we have the equality of events
\begin{equation}\{ K' = A, \overline{K}' =B \} \cap \sP^c_v(k+\ell+r) = \{K'=A, \overline{K}'=B\} \cap \sB(A,B).\end{equation}
Moreover, the event $\{K'=A, \overline{K}'=B\}$ is independent of the event $\left\{u \xleftrightarrow{L_{-\ell-k-r,0}} w\right\} \cap \sE(u,B) \cap \sB(A,B)$, since the former event depends only on edges in $B$ while the latter depends only on edges outside of $B$. We deduce that if $\bP_p(K'=A, \overline{K}'=B,\, \sP^c_v(k+\ell+r))>0$ then
 \begin{multline}
\bP_p\left(\left\{u \xleftrightarrow{L_{-\ell-k-r,0}} w\right\} \cap \sE(u,B)
 \mid  K'=A, \overline{K}'=B,\, \sP^c_v(k+\ell+r) \right)\\
=\bP_p\left(\left\{u \xleftrightarrow{L_{-\ell-k-r,0}} w\right\} \cap \sE(u,B)
 \mid  K'=A, \overline{K}'=B,\, \sB(A,B) \right)\\
 =\bP_p\left(\left\{u \xleftrightarrow{L_{-\ell-k-r,0}} w\right\} \cap \sE(u,B)
 \mid \sB(A,B) \right),
 \end{multline}
 and hence by \eqref{eq:Reimer_exploration_simplified} that
 \begin{multline}
   \bP_p\left(\left\{u \xleftrightarrow{L_{-\ell-k-r,0}} w\right\} \cap \sE(u,B)
 \mid  K'=A, \overline{K}'=B,\, \sP^c_v(k+\ell+r) \right) \\\leq 
 \bP_p\left(\left\{u \xleftrightarrow{L_{-\ell-k-r,0}} w\right\} \cap \sE(u,B)
  \right).
 \end{multline}
 % where we applied \eqref{eq:Reimer_exploration_simplified} in the last line.
 % 
Taking complements, it follows that if $\bP_p(K'=A, \overline{K}'=B,\, \sP^c_v(k+\ell+r))>0$ then
 \begin{multline}
 % \P_p\left( v \xleftrightarrow{L_{-\ell-k-r,0}} u,\, 
 % v \text{ not }\xleftrightarrow{L_{-k-r,\infty}} L_0(\rho) 
 % \mid \rho, U, K'=A, \overline{K}'=B,\, \sP_\rho(k+\ell+r) \right)\\
 % \geq 
\bP_p\left(\left\{u \xleftrightarrow{L_{-\ell-k-r,0}} w\right\} \cap \sE(u,B)
 \mid K'=A, \overline{K}'=B,\, \sP_v(k+\ell+r) \right)\\  
\geq 
 \bP_p\left(\left\{u \xleftrightarrow{L_{-\ell-k-r,0}} w\right\} \cap \sE(u,B)
 \right).
 % \\
 % \geq 
 % \P_p\left(C \text{ open, } u \xleftrightarrow{L_{-k-r,0}} w,\,
 % u \text{ not }\xleftrightarrow{L_{-k-r,\infty}} L_0(\rho)
 % \mid \rho, U, K'=A, \overline{K}'=B \right)\\
 % = 
 % \P_p\left(C \text{ open, } u \xleftrightarrow{L_{-k-r,0}} w,\,
 % u \text{ not }\xleftrightarrow{L_{-k-r,\infty}} L_0(\rho)
 % \mid \rho, U \right)
 % \geq 
 % \P_p\left( v \xleftrightarrow{L_{-k-r,0}} u,\,
 % v \text{ not }\xleftrightarrow{L_{-k-r,\infty}} L_0(\rho) 
 % \mid \rho, U, K'=A, \overline{K}'=B \right).
 \label{eq:simplified_first_Reimer}
 \end{multline}
 On the other hand, the same inequality holds trivially if $\bP_p(K'=A, \overline{K}'=B,\, \sP^c_v(k+\ell+r))=0$ by the aforementioned independence of $\{u \xleftrightarrow{L_{-\ell-k-r,0}} w\} \cap \sE(u,B)$ and $\{K'=A, \overline{K}'=B\}$.

Let $u \in L_{-\ell} \cap A$. Summing over $w \in L_{-k-\ell}=L_{-k}(u)$ in the inequality \eqref{eq:simplified_first_Reimer}, we obtain that
\begin{align}
\bE_p\Big[  X_{-k-\ell}^{-k-\ell-r,0}(v)\mid  &K'=A, \overline{K}'=B,\, \sP_v(k+\ell+r) \Big]\nonumber\\
&\geq \bE_p\Big[  X_{-k}^{-k-r,0}(u)\mid  K'=A, \overline{K}'=B,\, \sP_v(k+\ell+r) \Big]\nonumber
\\
&\geq \bE_p\Big[  X_{-k}^{-k-r,0}(u) \mathbbm{1}\left[\sE(u,B)\right]\mid  K'=A, \overline{K}'=B,\, \sP_v(k+\ell+r) \Big]\nonumber
\\
&\geq   \bE_p\left[ X^{-k-r,0}_{-k}(u) \mathbbm{1}\left[\sE(u,B)\right]  \right]
\geq \bE_p\left[ X^{-k-r,0}_{-k}(u) \mathbbm{1}\left(u \nxleftrightarrow{L_{-k-r-\ell,\infty}} L_0\right)  \right]
% &\geq 
 % \mathbbm{1}\big(\rho \leftrightarrow L_{-\ell}(\rho),\, \sP_v(\ell)\big) \inf_{x\in [0,1]} \E_p\left[ X^{-k-r+r_0,1}_{-k}(u) \,;\, C \text{ open, } v \text{ not }\xleftrightarrow{L_{-k-r,\infty}(u)} L_\ell  \mid  U_v=x \right]\\
 % &\geq \mathbbm{1}\big(\rho \leftrightarrow L_{-\ell}(\rho),\, \sP_v(\ell)\big) \sup_{x\in [0,1]} \E_p\left[ X^{-k-r+r_0-1,0}_{-k}(u) \,;\, C \text{ open, } v \text{ not }\xleftrightarrow{L_{-k-r,\infty}(u)} L_\ell  \mid  U_v=x \right],
 \label{eq:simplified1}
 \end{align}
where all inequalities other than the third are trivial. 
 It follows by transitivity that 
 % we deduce that
\begin{equation}
\bE_p\Big[  X_{-k-\ell}^{-k-\ell-r,0}(v)\mid  K'=A, \overline{K}'=B,\, \sP_v(k+\ell+r) \Big] \geq \bE_p\left[ X^{-k-r,0}_{-k}(v) \mathbbm{1}\left(v \nxleftrightarrow{L_{-k-r,\infty}} L_\ell\right)  \right]
\end{equation}
for every $(A,B) \in \cK'$. Summing over the possible values of $A$ and $B$, we obtain that
\begin{multline}
\bE_p\Big[  X_{-k-\ell}^{-k-\ell-r,0}(v)\mid \sP_v(k+\ell+r) \Big]\\ \geq \bE_p\left[ X^{-k-r,0}_{-k}(v) \mathbbm{1}\left(v \nxleftrightarrow{L_{-k-r,\infty}} L_\ell\right)  \right] \bP_p\bigl(v \leftrightarrow L_{-\ell}\mid \sP_v(k+\ell+r)\bigr),
\end{multline}
which is equivalent to the claim.
\end{proof}

We now generalize \cref{lem:peakcomparisonsimple} to the general case.  As in \cref{sec:Fekete}, this will involve implementing various finite-energy and index-shifting arguments to deal with the quasi-transitivity and the inhomogeneity of the uniform separating layers decomposition. The additional details required are  not very interesting, and the reader may wish to skip this proof on a first reading of the paper.

\begin{lemma}[Peak-comparison estimate]
\label{lem:peakcomparison}
There exists a constant $r_0$ such that
 % the estimate
\begin{multline}  
% \sup_{v \in V, x\in [0,1]} \E_p\left[ X^{-k-r-1,0}_{-k}(v_0) \mathbbm{1}\left(v_0 \nxleftrightarrow{L_{-k-r-1,\infty}(v_0)} L_{\ell-1}(v_0)\right) \mid U_{v_0}=x \right] 
% \\\preceq_p \inf_{v \in V, x \in [0,1]}
% \myfrac{\sum_{i=0}^{n_0}\E_p\left[  X^{-k-r-\ell,0}_{-k-\ell-i+1}(v) \mid U_v=x,\,  \sP_v(k+\ell+r) \right]}{\P_p(v \leftrightarrow L_{-\ell} \mid U_v=x,\,\sP_v(k+\ell+r))}
\sup_{v \in V, x\in [0,1]} \E_p\left[ X^{-k-r,0}_{-k}(v) \mathbbm{1}\left(v \nxleftrightarrow{L_{-k-r-2,\infty}(v)} L_{\ell-1}(v)\right) \mid U_{v}=x \right] 
\\\preceq_p \inf_{v \in V, x \in [0,1]}
\myfrac{\sum_{i=0}^{r_0}\E_p\left[  X^{-k-r-r_0-\ell,0}_{-k-\ell-i}(v) \mid U_v=x,\,  \sP_v(k+\ell+r+r_0) \right]}{\P_p(v \leftrightarrow L_{-\ell} \mid U_v=x,\,\sP_v(k+\ell+r+r_0))}
% \label{eq:unsimplified2}
\label{eq:startingpoint}
\end{multline}
% holds for every $r\geq 0$.
for every $k,\ell,r\geq 0$ and $0<p \leq p_c$.
\end{lemma}

% Let us first see how \cref{lem:peakcomparison,lem:peaksurvival} imply \cref{lem:startingpoint}.

\begin{proof}[Proof of \cref{lem:peakcomparison}]
Fix $k,\ell,r\geq 0$, $0< p \leq p_c$ and a vertex $v_0\in V$.  As before, since $p \leq p_c$, every cluster is finite almost surely. Condition on the random root $\rho$ and the random variable $U$ used to define the uniform separating layers decomposition. To lighten notation, we write $\bP'_p$ for the associated conditional probabilities, and write $L_{m,n}=L_{m,n}(\rho)$. Let $A$ be a set of vertices and let $B$ be a set of edges none of which has both endpoints in $L_{-\infty,-\ell}(\rho)$. Similarly to the simplified setting, we consider the event
\[
\sB(A,B) = \Bigl\{A \xleftrightarrow{L_{-r-k-\ell,\infty}} S_{0,\infty}(\rho) \text{ off $B$} \Bigr\},
\]
i.e., that $B$ is connected to $S_{0,\infty}(\rho)=\{v\in V : \nlog \Delta(\rho,v)\geq 0\}$ by an open path in $L_{-r-k-\ell,\infty}$ that does not include any edges of the set $A$. Similarly to above, we let $K'$ be the set of vertices that are connected to $\rho$ by a path consisting of open edges none of which have both endpoints in $L_{-\infty,-\ell}$, let $\overline{K}'$ be the set of edges have at least one endpoint in $K'$ and do not have both endpoints in $L_{-\infty,-\ell}$, and let $\cK'$  be the set of pairs $(A,B)$ such that $\peak(A)=v$, $A \cap L_{-\ell} \neq \emptyset$, and the event $\{ K' = A, \overline{K}' =B \}$ has positive $\bP'_p$-probability.

Fix a pair $(A,B) \in \cK'$. Choose a vertex $u \in A \cap L_{-\ell}$. By a similar argument to that of \cref{lem:probinfsup}, there exists a constant $n_0$ and a path $u=u_0,u_1,\ldots,u_{n}$ such that $n \leq n_0$, $\nlog \Delta(u,u_i)$ is decreasing, $\nlog \Delta(u,u_n) \leq -1$, and $[u_n]=[v_0]$. In particular, $v:=u_n \in L_{-\ell-n_0,-\ell-1}$. Let $\gamma$ be the edge set of such a path. Let $w\in L_{-\ell-k-n_0,-\ell-k-1}$. We claim that if $r \geq n_0$ then we have the equality of events
 % and 
\begin{multline}
\{ \gamma \text{ open}\} \cap \{ u \xleftrightarrow{L_{-k-\ell-r,-\ell}} w \} 
\cap 
\sE(u,B)
\cap
\sB(A,B) \\= \left( \{\gamma \text{ open}\} \cap \{ u \xleftrightarrow{L_{-k-\ell-r,-\ell}} w \} 
\cap 
\sE(u,B) \right) \circ \sB(A,B),
\end{multline}
where, similarly to before, we write
\[
\sE(u,B) = \Bigl\{ u \xleftrightarrow{L_{-k-\ell-r,\infty}} S_{0,\infty}(\rho) \text{ off $B$}\Bigr\}^c.
\]
Indeed, suppose that the event on the left-hand side holds. Then there exists a set of open edges  $W_1$ in $L_{-k-\ell-r,-\ell}$ (which is necessarily disjoint from $B$) that form a path connecting $u$ to $w$, and a set of open edges $W_2$ in $L_{-k-\ell-r,\infty}$ that is disjoint from $B$ and forms a path from $A$ to $S_{0,\infty}(\rho)$. The set $W_2$ must be disjoint from $W_1$ and $\gamma$, since otherwise  $u$ would be connected to $S_{0,\infty}(\rho)$ by an open path in $L_{-k-\ell-r,\infty}$ that does not use any edges of $B$ (this is where we use that $r \geq n_0$). Thus, the union of the set $W_1$ together with $\gamma$ and the set of all the closed touching the cluster of $u$ is a finite witness for $\{\gamma \text{ open}\}\cap \{ v \xleftrightarrow{L_{-k-\ell-r,-\ell}(\rho)} u\} 
\cap 
\sE(u,B)$, and the set $W_2$ is a finite witness for $\sB(A,B)$ disjoint from this set. This yields the claimed equality of events.
Applying Reimer's inequality, we deduce that 
\begin{multline}
\bP'_p\left(
\{\gamma \text{ open} \}\cap 
\{ u \xleftrightarrow{L_{-k-\ell-r,-\ell}} w \} 
\cap 
\sE(u,B)
\cap
\sB(A,B)\right) \leq
\\ \bP'_p\left(\{\gamma \text{ open} \}\cap 
 \{ u \xleftrightarrow{L_{-k-\ell-r,-\ell}} w \} 
\cap 
\sE(u,B) \right)
\bP'_p\left(
\sB(A,B) \right)
\label{eq:Reimer1974}
\end{multline}
if $r \geq n_0$.

Now, as before we have that
$\{ K' = B, \overline{K}' =A \} \cap \sP^c_\rho(k+\ell+r) = \{K'=A, \overline{K}'=B\} \cap \sB(A,B)$ and that the events $\{ K' = B, \overline{K}' =A \} $ and $\{\gamma$ open$\} \cap \{ u \xleftrightarrow{L_{-k-r-\ell,-\ell}} w \}  \cap \cE(u,B) \cap \sB(A,B)$ are independent. Thus, applying \eqref{eq:Reimer1974} and arguing as in the simplified case yields that if $r \geq n_0$ then
%  \begin{multline}
% \bP'_p\left(\{\gamma \text{ open}\} \cap \left\{ u \xleftrightarrow{L_{-k-r-\ell,-\ell}} w \right\} \cap \sE(u,B)
%  \mid K'=A, \overline{K}'=B,\, \sP^c_\rho(k+\ell+r) \right)\\
% =\bP'_p\left(\{\gamma \text{ open}\} \cap \left\{ u \xleftrightarrow{L_{-k-r-\ell,-\ell}} w \right\} \cap \sE(u,B)
%  \mid  K'=A, \overline{K}'=B,\, \sB(A,B) \right)\\
%    \leq 
%  \bP'_p\left(\{\gamma \text{ open}\} \cap \left\{ u \xleftrightarrow{L_{-k-r-\ell,-\ell}} w \right\} \cap \sE(u,B)
%  \mid K'=A, \overline{K}'=B \right)
%  \label{eq:Reimer1975}
%  \end{multline}
%  and hence, taking complements, that
 \begin{multline}
 \label{eq:Reimer1975}
 % \P_p\left( v \xleftrightarrow{L_{-k-r-\ell,-\ell}} u,\, 
 % v \text{ not }\xleftrightarrow{L_{-k-r,\infty}(v)} L_0 
 % \mid \rho, U, K'=A, \overline{K}'=B,\, \sP_v(k+\ell+r) \right)\\
 % \geq 
\bP'_p\left(\{\gamma \text{ open}\} \cap \left\{ u \xleftrightarrow{L_{-k-r-\ell,-\ell}} w \right\} \cap \sE(u,B)
 \mid  K'=A, \overline{K}'=B,\, \sP_\rho(k+\ell+r) \right)\\  
\geq 
 \bP'_p\left(\{\gamma \text{ open}\} \cap \left\{ u \xleftrightarrow{L_{-k-r-\ell,-\ell}} w \right\} \cap \sE(u,B)
 \right).
 \end{multline}
 % if $r \geq n_0$. 
Next, we observe that if $r \geq n_0$ then
\begin{multline}
\label{eq:Reimer_Finite_Energy}
\bP'_p\left(\{\gamma \text{ open}\} \cap \left\{ u \xleftrightarrow{L_{-k-r-\ell,-\ell}} w \right\} \cap \sE(u,B)
  \right)
= \bP'_p\left(\{\gamma \text{ open}\} \cap \left\{ v \xleftrightarrow{L_{-k-r-\ell,-\ell}} w \right\} \cap \sE(v,B)
  \right)
 \\ \succeq_p 
\bP'_p\left( \left\{ v \xleftrightarrow{L_{-k-r-\ell,-\ell}} w \right\} \cap \sE(v,B)
  \right).
 \end{multline}
 The first equality is trivial, while the inequality on the third line
 can be proved via a finite-energy argument, outlined as follows: Since $\gamma$ has bounded length, at the cost of a $p$-dependent constant, we can force the path $\gamma$ to be open without affecting whether or not the event $\{v \xleftrightarrow{L_{-k-r-\ell,-\ell}} w\}\cap \sE(v,B)$ occurs. Indeed, simply open every edge in $\gamma$, and close every edge that is incident to but not contained in $\gamma$, is not in $B$, and does not have that both endpoints were already in the off-$B$ cluster of $v$ before we made this modification.

Consider the random variable
$Z = \#\bigl\{a \in L_{-k-1,-k+1}(v) : v \xleftrightarrow{L_{-k-r-\ell,-\ell}} a \bigr\}$. 
Combining the estimates \eqref{eq:Reimer1975} and \eqref{eq:Reimer_Finite_Energy} and summing over all choices of $w$ in the set $L_{-k-1,-k+1}(v),$ we deduce that, since $L_{-k-1,-k+1}(v) \subseteq L_{-\ell-k-n_0-1,-\ell-k}$,
\begin{multline}
\sum_{i=0}^{1+n_0} \bE'_p\left[  X^{-k-r-\ell,0}_{-k-\ell-i}(\rho) \mid  K'=A, \overline{K}'=B,\, \sP_\rho(k+\ell+r) \right]\\
\succeq_p   \bE'_p\left[ Z \mathbbm{1}(\sE(v,B))  \right]
% &\geq 
 % \mathbbm{1}\big(\rho \leftrightarrow L_{-\ell},\, \sP_\rho(\ell)\big) \inf_{x\in [0,1]} \E_p\left[ X^{-k-r+r_0,1}_{-k}(u) \,;\, C \text{ open, } v \text{ not }\xleftrightarrow{L_{-k-r,\infty}(u)} L_\ell(v)  \mid  U_v=x \right]\\
 % &\geq \mathbbm{1}\big(\rho \leftrightarrow L_{-\ell},\, \sP_\rho(\ell)\big) \sup_{x\in [0,1]} \E_p\left[ X^{-k-r+r_0-1,0}_{-k}(u) \,;\, C \text{ open, } v \text{ not }\xleftrightarrow{L_{-k-r,\infty}(u)} L_\ell(v)  \mid  U_v=x \right],
 \label{eq:unsimplified2b}
 \geq 
 \bE'_p\left[ Z \mathbbm{1}\left(v \nxleftrightarrow{L_{-k-r-\ell,\infty}} S_{0,\infty}(\rho)\right)  \right]
 \end{multline}
 if $r \geq n_0$, where the second inequality is trivial. 
% On the other hand, if $r \geq n_0$ then for each $w \in L_{-k-1,-k+1}(v)$ the  events
% % \begin{multline} 
% $\{v \xleftrightarrow{L_{-k-r-\ell,-\ell}} w\} \cap \sE(v,B)$ and $
%  \{K'=A, \overline{K}'=B\}$ 
%  % \\=
% % \left(\{v \xleftrightarrow{L_{-k-r-\ell,-\ell}} w\} \cap \sE(v,B) \right) \circ \{K'=A, \overline{K}'=B\}.
%  % \end{multline} 
%  are independent as before, and we deduce that if
% % 
%  % Indeed, on the event on the left hand side, let $W$ be an open path from $v$ to $w$ in $L_{-k-r-\ell,-\ell}$, which is necessarily disjoint from $B$. Then $B$ is a witness for $\{K'=A, \overline{K}'=B\}$ while the union of $W$ with the set of all closed edges not contained in $B$ is a witness for $\{u \xleftrightarrow{L_{-k-r-\ell,-\ell}} w\} \cap \sE(u,B)$. Thus, we may apply Reimer's inequality a second time to deduce that if 
%  $r \geq n_0$ then
%  \begin{multline}
% \bE'_p\left[ Z \mathbbm{1}\left(v \nxleftrightarrow{L_{-k-r-\ell,\infty}} S_{0,\infty}(\rho)\right)  \right]
% \leq 
% \bE'_p\left[ Z \mathbbm{1}(\sE(v,B))  \right]
% \\
% =
% \bE'_p\left[ Z \mathbbm{1}\left(\sE(v,B)\right)  \mid  K'=A, \overline{K}'=B \right].
% \label{eq:unsimplified2}
% \end{multline}
Meanwhile, it follows from the definitions that, since $[v]=[v_0]$ and $-n_0 \leq \nlog \Delta(u,v) \leq -1$,
\begin{multline}
\sup_{x\in [0,1]} \E_p\left[ X^{-k-r+n_0+1,0}_{-k}(v_0) \mathbbm{1}\left(v_0 \nxleftrightarrow{L_{-k-r+n_0-1,\infty}(v_0)} L_{\ell-1}(v_0)\right) \mid U_{v_0}=x \right] 
\\
\leq \inf_{x\in [0,1]} \E_p\left[ \sum_{i=-1}^1 X^{-k-r+n_0,1}_{-k+i}(v_0)  \mathbbm{1}\left(v_0 \nxleftrightarrow{L_{-k-r+n_0,\infty}(v_0)} L_{\ell}(v_0)\right) \mid U_{v_0}=x \right] 
\\
\leq \bE'_p\left[ Z \mathbbm{1}\left(v \nxleftrightarrow{L_{-k-r-\ell,\infty}} S_{0,\infty}(\rho)\right)  \right],
\label{eq:unsimplified3}
\end{multline}
where the second inequality follows since $L_{-k-r+n_0,1}(v) \subseteq L_{-k-r,0}(u) = L_{-k-r-\ell,-\ell}$, $L_{-k-r+n_0,\infty}(v) \subseteq L_{-k-r,\infty}(u)=L_{-k-r-\ell,\infty}$, and $L_{\ell,\infty}(v) \supseteq L_{\ell-1,\infty}(u) \supseteq S_{0,\infty}(\rho)$.

Putting together \eqref{eq:unsimplified2b} and \eqref{eq:unsimplified3} and taking $m =r -n_0-1$ we deduce that  if $m\geq 0$ then
 \begin{multline}
\sup_{x\in [0,1]} \E_p\left[ X^{-k-m,0}_{-k}(v_0) \mathbbm{1}\left(v_0 \nxleftrightarrow{L_{-k-m-2,\infty}(v_0)} L_{\ell-1}(v_0)\right) \mid U_{v_0}=x \right] 
\\\preceq_p
\sum_{i=0}^{1+n_0}\bE'_p\left[  X^{-k-m-n_0-1-\ell,0}_{-k-\ell-i}(\rho) \mid  K'=A, \overline{K}'=B,\, \sP_\rho(k+\ell+m+n_0+1) \right].
\label{eq:unsimplified2}
\end{multline}
Since $(A,B) \in \cK'$ was arbitrary, we may average over the possible choices of $A$ and $B$ to obtain that if $m\geq 0$ then
 \begin{multline}
\sup_{x\in [0,1]} \E_p\left[ X^{-k-m,0}_{-k}(v_0) \mathbbm{1}\left(v_0 \nxleftrightarrow{L_{-k-m-2,\infty}(v_0)} L_{\ell-1}(v_0)\right) \mid U_{v_0}=x \right] 
\\\preceq_p
\myfrac{\sum_{i=0}^{1+n_0}\bE'_p\left[  X^{-k-m-n_0-1-\ell,0}_{-k-\ell-i}(\rho) \mid  \sP_\rho(k+\ell+m+n_0+1) \right]}{\bP'_p(\rho \leftrightarrow L_{-\ell} \mid \sP_\rho(k+\ell+m+n_0+1))}.
\label{eq:unsimplified2}
\end{multline}
Since this estimate holds no matter the value of $\rho$ and $U$, it follows that
 \begin{multline}
\sup_{v \in V, x\in [0,1]} \E_p\left[ X^{-k-m,0}_{-k}(v_0) \mathbbm{1}\left(v_0 \nxleftrightarrow{L_{-k-m-2,\infty}(v_0)} L_{\ell-1}(v_0)\right) \mid U_{v_0}=x \right] 
\\\preceq_p \inf_{v \in V, x \in [0,1]}
\myfrac{\sum_{i=0}^{1+n_0}\E_p\left[  X^{-k-m-n_0-1-\ell,0}_{-k-\ell-i}(v) \mid U_v=x,\,  \sP_v(k+\ell+m+n_0+1) \right]}{\P_p(v \leftrightarrow L_{-\ell} \mid U_v=x,\,\sP_v(k+\ell+m+n_0+1))}
% \label{eq:unsimplified2}
\end{multline}
for every $m \geq 0$. This is easily seen to imply the claim by taking $r_0=n_0+1$. \qedhere

\end{proof}

\subsection{The view from the peak at criticality}
\label{subsec:critical_peak}

In this subsection we implement the second step of the strategy sketched at the beginning of the section. That is, we prove a lower bound on the denominator appearing in the right hand side of \eqref{eq:simple_peak_comparison_overview}. This estimate extends \cref{lem:subcritpeaksurvival} to the critical case at the cost of an additional $o(k)$ error term in the exponential. Note that the proof does not give any explicit control of this error term.

\begin{lemma}
\label{lem:peaksurvival}
The estimate
\begin{equation}
\P_{p_c}\left(\rho \leftrightarrow L_{-k}(\rho),\, \sP_\rho \mid U_\rho = x \right) \succeq \nexp\left[-(\alpha_{p_c}-1)k +o(k) \right]
\end{equation}
holds for every $k\geq 0$ and $x\in [0,1]$.
\end{lemma}

As before, it follows by a simple finite-energy argument that there exists $v_0 \in V$ such that
\begin{equation}
\label{eq:peaksurvival_v0}
\P_{p_c}\left(v_0 \leftrightarrow L_{-k}(v_0),\, \sP_{v_0} \mid U_{v_0} = x \right) \succeq \nexp\left[-(\alpha_{p_c}-1)k +o(k) \right]
\end{equation}
for every $k\geq 0$ and $x\in [0,1]$.

\begin{figure}
\centering
\includegraphics[width=0.5\textwidth]{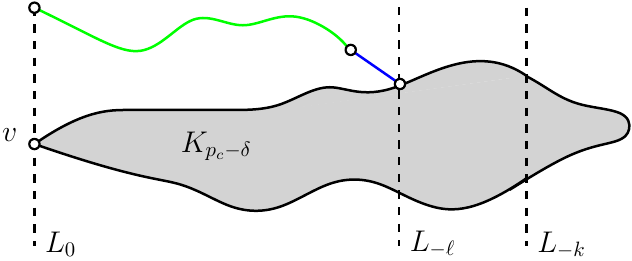}
\caption{Illustration of the proof of \cref{lem:peaksurvival}. If $v$ is the peak of its cluster $K_{p_c-\delta}$ in $G[p_c-\delta]$ (grey) but is not the peak of its cluster in $G[p_c]$, there must be a $(p_c-\delta)$-closed edge (blue) incident to $K_{p_c-\delta}$ that is $p_c$-open and whose other endpoint is connected to level zero by a $p_c$-open path (green) disjoint from $K_{p_c-\delta}$. Since there is no percolation at criticality, the conditional probability that any particular edge has this property tends to zero as the height of the edge tends to $-\infty$. On the other hand, \cref{lem:subcritpeaksurvival,prop:tamesecondmoment} imply that the expected number of points of $K_{p_c-\delta}$ in $L_{-\ell}$ conditioned on survival to level $-k$ is $O_\delta(1)$ when $\ell \leq k$ and is $\exp\bigl[-(\alpha_{p_c-\delta}-1)(\ell-k))+O_\delta(1)\bigr]$ when $\ell > k$. Together these facts imply that the conditional probability that $v$ is the peak of its cluster in $G[p_c]$ conditioned on the event that it is the peak of its cluster in $G[p_c-\delta]$ and that $K_{p_c-\delta}$ intersects $L_{-k}$ is $\exp\bigl[-o_\delta(k)\bigr]$. We can then conclude by applying \cref{lem:alphacontinuity} and \cref{lem:subcritpeaksurvival}.}
\end{figure}

\begin{proof}[Proof of \cref{lem:peaksurvival}]
It suffices to show that
% \[
$\P_{p_c}\left(\rho \leftrightarrow L_{-k}(\rho),\, \sP_\rho \right) \succeq_\eps \nexp\left[-(\alpha_{p_c}-1+\eps)k\right]$
% \] 
for every $\eps>0$ and $k\geq 0$. 
Let $\eps>0$ and, applying \cref{lem:alphacontinuity}, let $\delta=\delta_\eps >0$ be sufficiently small that $\alpha_{p_c-\delta} \leq \alpha_{p_c}+\epsilon/2$.
As in the proof of \cref{prop:tiltedmeanfieldlowerbound}, we couple $G[p_c-\delta]$ and $G[p_c]$ by letting every edge of $G$ be \textbf{open} with probability $p_c-\delta$ and, independently, \textbf{blue} with probability $\delta/(1-p_c-\delta)$. We write $\P$ and $\E$ for probabilities and expectations with respect to these random variables, together with an independent random root $\rho$ and independent random labels $U$. We write $K_{p_c-\delta}$ and $K_{p_c}$ for the connected component of $\rho$ in $G[p_c-\delta]$ and $G[p_c]$ respectively. We also write $\sP_{p_c}$ for the event that $\rho$ is the peak of $K_{p_c}$ and $\sP_{p_c-\delta}$ for the event that $\rho$ is the peak of $\sP_{p_c-\delta}$. When writing connectivity events, we use subscripts below arrows to denote connectivity in $G[p_c-\delta]$ and $G[p_c]$ as appropriate.

\cref{lem:subcritpeaksurvival} yields that
\begin{equation}
\label{eq:peaksurvival1}
\P\left(\rho \xleftrightarrow[p_c-\delta]{} L_{-k}(\rho), \sP_{p_c-\delta} \right) \succeq_\eps \nexp\left[-\big(\alpha_{p_c-\delta}-1\big)k\right],
\end{equation}
and combined with \cref{prop:tamesecondmoment} this implies that there exists a constant $C_\eps$ such that
\begin{equation}
\label{eq:peaksurvival2}
\E\left[ \Bigl|K_{p_c-\delta}(\rho) \cap L_{-\ell,0}(\rho)\Bigl| \; \mid \sP_\rho,\, \rho \xleftrightarrow[p_c-\delta]{} L_{-k}(\rho)\right] \leq  \begin{cases}
C_\eps & k\geq \ell\\
C_\eps \nexp\left[ -(\alpha_{p_c-\delta}-1)(\ell-k) \right] & \ell > k
\end{cases}
\end{equation}
for every  $k,\ell \geq 0$. Since there is no infinite cluster at $p_c$ almost surely \cite{timar2006percolation,Hutchcroft2016944}, we have that
\begin{equation}
\lim_{k\to\infty} \sup_{v\in V, x\in [0,1]}\P_{p_c}(v \leftrightarrow L_k(v) \mid U_v = x) =0,
\end{equation}
and so we may take $k_\eps <\infty$ sufficiently large that
\begin{equation}
\inf_{v\in V, x\in [0,1]} \P_{p_c}(v \nleftrightarrow L_{k,\infty}(v) \mid U_v = x) =
\inf_{v\in V, x\in [0,1]} \P_{p_c}(v \nleftrightarrow L_{k}(v) \mid U_v = x) \geq \nexp\left[-\frac{\eps}{2C_\eps}\right]
\end{equation}
for every $k\geq k_\eps$.

If $\rho$ is the peak of $K_{p_c-\delta}$ but not of $K_{p_c}$, then there must exist a vertex $u\in K_{p_c-\delta}$ that is connected to $L_{0,\infty}(\rho) \setminus K_{p_c-\delta}$ in $G[p_c]$ off of $K_{p_c-\delta}$ (i.e., without using any edge of the cluster $K_{p_c-\delta}$). Applying the Harris-FKG to the inequality to the intersection of the complements of these events, we obtain that 
\begin{multline}
\P\Big(\sP_{p_c} \mid \rho,\, K_{p_c-\delta}  \Big)\\ 
\geq \mathbbm{1}\left(\sP_{p_c-\delta}\right)\prod_{u \in K_{p_c-\delta}} \P\Big( u \text{ is not connected to $L_{0,\infty}(\rho)\setminus K_{p_c-\delta}$ in $G[p_c]$ off of $K_{p_c-\delta}$} \mid \rho,\, K_{p_c-\delta}\Big).
\label{eq:PeakFKG}
\end{multline}
Now, for each $u \in K_{p_c-\delta}$, we can bound from below the conditional probability appearing in the product on the right hand side of \eqref{eq:PeakFKG} either by the conditional probability that $u$ does not have any blue neighbours, which is at least $[(1-p_c)/(1-p_c-\delta)]^{\deg(u)}\geq (1-p_c)^{\deg(u)}$, or by the probability that $u$ is not connected to $L_{0,\infty}(v)$. Choosing which of these bounds to apply according to whether $v \in L_{-k_\eps,0}$ or $L_{-\infty,-k_\eps}$, we obtain that
\begin{multline}
% &\geq \prod_{u \in K_{p_c-\delta}(v)} \\
\P\Big(\sP_{p_c} \mid \rho,\, K_{p_c-\delta}, U  \Big)\\ 
% \geq 
% \prod_{u \in K_{p_c-\delta}(v)} \P\Big( u \text{ is not connected to $L_0(\rho)\setminus \{u\}$ in $G[p_c]$ off of $K_{p_c-\delta}$} \mid \rho,\, K_{p_c-\delta},\, \sP_{p_c-\delta}\Big)\\
\geq
\prod_{u \in K_{p_c-\delta} \cap L_{-k_\eps+1,0}(\rho)}\left(1-p_c\right)^{\deg(u)}
\prod_{u \in K_{p_c-\delta} \cap L_{-\infty,-k_\eps}(\rho)} \inf_{x\in [0,1]} \P_{p_c}\left( u \nleftrightarrow L_{0,\infty}(v) \mid U_u = x\right)
\vspace{-0.2em}
\end{multline}
and hence by definition of $k_\eps$ that there exists a constant $C$ such that
\begin{equation}
\P\Big(\sP_{p_c} \mid \rho,\, K_{p_c-\delta}, U  \Big)
\geq 
% \left(1-p_c\right)^{C|K_{p_c-\delta} \cap H^+_{0,-k_\eps}(\rho)|}
% \left(1-
\mathbbm{1}(\sP_{p_c-\delta})\cdot \nexp\left[-C|K_{p_c-\delta} \cap L_{-k_\eps+1,0}(\rho)|-\frac{\eps}{2C_\eps}|K_{p_c-\delta} \cap L_{-\infty,-k_\eps}|\right].
 % \frac{1}{2} 
% \left[\inf_{v\in V, x\in [0,1]} \P_{p_c}(v \nleftrightarrow L_{k} \mid U_v = x) \right]^{C_\eps k}
\end{equation}
Taking expectations over $K_{p_c-\delta}$ and $\rho$, using Jensen's inequality and applying the estimate \eqref{eq:peaksurvival2} we obtain that
\begin{equation}
\P\Big(\sP_{p_c} \mid \rho \xleftrightarrow[p_c-\delta]{} L_{-k}(\rho),\, \sP_{p_c-\delta}, U_\rho =x  \Big)
\\\geq   
\nexp\left[-C C_\eps k_\eps -\frac{\eps}{2}(k-k_\eps) - \frac{\eps}{2} C'_\eps \right] \succeq_\eps \nexp\left[ -\frac{\eps}{2}k \right],
\end{equation}
for every $x\in [0,1]$, 
where $C'_\eps\!=\sum_{k\geq 0} \nexp\left[ -(\alpha_{p_c-\delta}-1)k\right]$. 
The result follows by combining this inequality with~\eqref{eq:peaksurvival1}.
\end{proof}

% Let us also note the following easy consequence of \cref{lem:peaksurvival} and \cref{prop:subcritical}.

% \begin{lemma}
% \label{lem:downprob}
% Let $G$ be a connected, locally finite graph, and let $\Gamma \subseteq \Aut(G)$ be quasi-transitive and nonunimodular. Then 
% \begin{equation}\P_p\left( v \xleftrightarrow{L_{-\infty,0}(v)}H^-_{-t}(v)\right) \asymp \exp\left[-(\alpha_p-1)t +o_p(t)\right]\]
% for every $0<p \leq p_c$ and $t\geq 0$. 
% \end{lemma}

\subsection{Fractional moment estimates I: A bootstrapping toolkit}
% \subsection{Bootstrapping to a full-space fractional moment estimate}
\label{subsec:toolkit}

It remains to use the estimates proven in \cref{subsec:peak_comparison,subsec:critical_peak} to prove \cref{lem:bootstrap}. In this subsection, we develop some basic tools that will be used in this proof, which can be thought of as nonlinear versions of the tools used in \cref{sec:Fekete}. 
% Before proving \cref{lem:bootstrap}, we establish some preliminary lemmas that will be used in the proof. 
The first tool in our kit is a version of the equality $\E_p[X^{m,n}_k(\rho)] \asymp \nexp(-k)\E_p[X^{m-k,n-k}_{-k}(\rho)]$ that holds for fractional moments.

\newcommand{\HolderXa}{\left(X_{-k}^{m-k,n-k}(\rho)\right)}
\newcommand{\HolderXb}{\left(X_{k}^{m,n}(\rho)\right)}
\newcommand{\HolderXc}{\left(X_{0}^{m,n}(\rho)\right)}

\begin{lemma}[H\"older-MTP estimate]
\label{lem:HolderMTP}
The estimate
\begin{align}
\E_{p}\left[\HolderXb^{1-\eps} \right] \preceq \nexp\left[-(1-\eps)k \right]\, \E_p\left[ \HolderXa^{1-\delta}\right]^{1-\eps}
\E_{p}\left[\,\HolderXc^{(1-\eps)}\,\right]^\delta
\end{align}
holds for every $p\in [0,1]$, $k\in \Z$, $0\leq \delta\leq \eps \leq 1$, and $-\infty \leq m \leq n \leq \infty$.
\end{lemma}

The proof will apply the following form of H\"older's inequality, which we state here for clarity: If $X$ is a non-negative random variable and $Y$ is a positive random variable on the same probability space, then
% \[
% \E \left[XY\right] \geq \E\left[X^{1/p}\right]^p \E\left[Y^{-1/(p-1)}\right]^{-(p-1)}
% \]
\begin{equation}
\E\left[X^{1/\ell}\right] \leq \E \left[\frac{X}{Y}\right]^\frac{1}{\ell}\E\left[Y^{1/(\ell-1)}\right]^{\frac{\ell-1}{\ell}}
\end{equation}
for every $\ell>1$.

\begin{proof}[Proof of \cref{lem:HolderMTP}]

The claim is trivial when $\eps=1$ so we may assume that $0\leq \eps < 1$. 
% Fix $\eps>0$.
Applying the tilted mass-transport principle to the function
\begin{equation}
f\bigl(u,v,G[p],U\bigr) := \left(X_{-k}^{m-k,n-k}(u)\right)^{-\delta}\mathbbm{1}\Bigl(v \in L_{-k}(u) \text{ and } u \xleftrightarrow{L_{m-k,n-k}(u)} v \Bigr)
\end{equation}
yields that
\begin{align}
\E_{p}\left[ \HolderXa^{1-\delta}\right]
&\asymp \nexp(k)\,  \E_{p}\left[ \HolderXb \HolderXc^{-\delta}\right] \end{align}
for every $0\leq \delta <1 $, $-\infty \leq m \leq n \leq \infty$ and $k\in \Z$.
On the other hand, H\"older's inequality with $\ell=1/(1-\eps)$ implies that
\begin{equation}
\E_{p}\left[\HolderXb^{1-\eps}\right] \leq \E_{p}\left[\,\HolderXb\left(X_0^{m,n}(\rho)\right)^{-\delta}\,\right]^{1-\eps} \E_{p}\left[\,\HolderXc^{(1-\eps)\delta/\eps}\,\right]^\eps
\end{equation}
for every $0\leq \delta,\eps <1$, $-\infty \leq m \leq n \leq \infty$ and $k\in \Z$.
Combining these we obtain that
\begin{align}
\E_{p}\left[\HolderXb^{1-\eps} \right] &\preceq
\nexp\left[-(1-\eps)k\right]\, \E_p\left[ \HolderXa^{1-\delta}\right]^{1-\eps}
\E_{p}\left[\,\HolderXc^{(1-\eps)\delta/\eps}\,\right]^\eps
% & \preceq \nexp\left[-(1-\eps)k\right]\, \E_p\left[ \HolderXa^{1-\delta}\right]^{1-\eps}
% \E_{p}\left[\,\HolderXc^{(1-\eps)\delta/\eps}\,\right]^\eps,
\end{align}
for every $1>\delta,\eps>0$, $-\infty \leq m \leq n \leq \infty$ and $k\in \Z$. The claim follows by Jensen's inequality when $0\leq \delta \leq \eps <1$.
\end{proof}

Next, we have the following version of the extreme value path decomposition inequality (item $3$ of \cref{lem:pathdecomposition}) that holds for fractional moments. For each $-\infty \leq m \leq n \leq \infty$, $-\infty< k < \infty$, and $0\leq \eps \leq 1$ we define
\begin{equation}
\label{eq:fractional_E_def}
E_p^{m,n}(k;1-\eps) := \sup_{v\in V, x\in [0,1]}\E_p\left[ \left(X^{m,n}_k(v)\right)^{1-\eps} \mid U_v=x\right].
\end{equation}

\begin{lemma}[Up and down estimates]
\label{lem:fractionalupanddown}
The estimates
% \begin{enumerate}
		% \item \emph{(Up then down)}\textbf{.}
\begin{align}
\label{eq:upthendown}
E_p^{m,n}(k; (1-\eps)(1-\delta) ) &\leq \sum_{\ell = k\vee 0}^n E_p^{m,\ell}(\ell ; 1-\delta ) \left(E_p^{m-\ell,0}(k-\ell; 1-\eps) \right)^{(1-\delta)}\\
\intertext{and}
E_p^{m,n}(k; (1-\eps)(1-\delta) ) &\leq \sum_{\ell =m }^{k\wedge 0} E_p^{\ell,n}(\ell ; 1-\delta ) \left(E_p^{0,n-\ell}(k-\ell; 1-\eps) \right)^{(1-\delta)}
\label{eq:downthenup}
\end{align}
hold for every $p\in [0,1]$, $0\leq \eps,\delta \leq 1$ and $-\infty \leq m \leq k \leq n \leq \infty$
\end{lemma}

\begin{proof}[Proof of \cref{lem:fractionalupanddown}]
We prove \eqref{eq:upthendown}, \eqref{eq:downthenup} being similar. 
% Since 
% \begin{equation} \left(X_k^{m,n}(v)\right)^{(1-\eps)(1-\delta)} = \left (\sum_{\ell=k \wedge 0}^m X_{k}^{m,\ell}(v)-X_k^{m,\ell-1}(v)\right)^{(1-\eps)(1-\delta)}
% \leq \!\!\sum_{\ell=k \wedge 0}^m \left ( X_{k}^{m,\ell}(v)-X_k^{m,\ell-1}(v)\right)^{(1-\eps)(1-\delta)}
% \end{equation}
 It suffices to show that
\begin{equation}
\E_p \left[ \left (X_{k}^{m,\ell}(v)-X_k^{m,\ell-1}(v)\right)^{(1-\eps)(1-\delta)} \mid U_v=x \right] \\ 
\leq E_p^{m,\ell}(\ell;1-\delta)E_p^{m-\ell,0}(k-\ell;1-\eps)^{1-\delta}
\end{equation}
for every $\ell \geq m$, $v\in V$ and $x\in [0,1]$; the claim follows from this inequality by summing over $\ell$ (and using the fact that concave functions are subadditive).
Fix $v\in V$, $x\in [0,1]$, and $m\leq k \leq \ell$. Condition on $U_v=x$ and let $K'_0$ be the set of vertices that are connected to $v$ by an open path in $L_{m,\ell}(v)$ none of whose edges have both endpoints in $L_{\ell}(v)$. Condition on $K'_0$, let $N=|K'_0 \cap L_\ell(v)|$, and let $w_1,\ldots,w_N$ be an enumeration of $K'_0 \cap L_{\ell}(v)$. For each $1\leq i \leq N$, let $K_i'$ be the set of vertices that are connected to $w_i$ by an open path in $L_{m,\ell}(v) \setminus (\bigcup_{j=0}^{i-1} K_i')$. Note that if $u \in L_{k}(v)$ is connected to $v$ by an open path in $L_{m,\ell}(v)$ but not in $L_{m,\ell-1}(v)$, then there exists a unique $1 \leq i \leq N$ such that $v \in K'_i$. 
Moreover, for each $1\leq i \leq N$, the conditional law of $K_i'$ given $U$ and $K'_0,\ldots,K'_{i-1}$ is stochastically dominated by the cluster of $w_i$ in $L_{m,\ell}(v)$ in an independent copy of Bernoulli-$p$ percolation, conditioned only on $U$. Using this observation together with Jensen's inequality we obtain that
\begin{multline}
\E_p \Big[ \left (X_{k}^{m,\ell}(v)-X_k^{m,\ell-1}(v)\right)^{(1-\eps)(1-\delta)} \mid K_0',\, U_v=x \Big]\\
 % \hspace{5.5cm}
% \end{align}
% \begin{align}
% \hspace{5.5cm} 
% &
=\E_p\left[ \Big(\sum_{i=1}^N |K_i' \cap L_k(v)|\Big)^{(1-\eps)(1-\delta)} \mid K_0',\, U_v =x\right]\\
% &
\hspace{6cm}\leq \E_p\left[ \sum_{i=1}^N |K_i' \cap L_k(v)|^{1-\eps} \mid K_0',\, U_v =x\right]^{1-\delta}\\
% &
\leq E_p^{m-\ell,0}(k-\ell;1-\eps)^{1-\delta} N^{1-\delta},
\end{multline}
and taking expectations over $K_0'$ yields the result.
\end{proof}

Finally we have the following simple pair of estimates.

\begin{lemma}
\label{lem:bootstrapping_etc}
\hspace{1cm}
\begin{enumerate}
	\item
The estimate
\begin{align}
\label{eq:supaverage}
% \E_p\left[ \right]
E_p^{m,n}(k;1-\eps) &\preceq  \E_p\left[ \left(X^{m-1,n+1}_{k-1}(\rho)\right)^{1-\eps}+\left(X^{m-1,n+1}_{k}(\rho)\right)^{1-\eps}+\left(X^{m-1,n+1}_{k+1}(\rho)\right)^{1-\eps}  \right],
\end{align}
holds for every $p\in [0,1]$ and $0\leq \eps \leq 1$, and $-\infty \leq m \leq k \leq n \leq \infty$.
\item
The estimate
\begin{equation}
\label{eq:FKGshift}
E_p^{m+r,n+r}(k+r;1-\eps) \leq \nexp\left[ O_p(|r|)\right] E_p^{m,n}(k;1-\eps) 
\end{equation}
holds for every $p\in [0,1]$, $0\leq \eps \leq 1$, and $r\in \Z$, and every $-\infty \leq m \leq k \leq n \leq \infty$ such that $m \leq 0 \leq n$ and $m+r \leq 0 \leq n+r$.
% hold for every $-\infty \leq m \leq 0$, $0\leq n \leq \infty$ $m\leq k \leq n$, $r\geq 0$, $p\in [0,1]$, and $\eps \in [0,1]$, where $C_p$ is a constant depending on $p$ but not $\eps$.
\end{enumerate}
\end{lemma}

\begin{proof}
\eqref{eq:supaverage} follows by a similar proof to \cref{lem:bestworst}. \eqref{eq:FKGshift} follows by a similar argument to the proof of \cref{lem:probinfsup}.
\end{proof}

\subsection{Fractional moment estimates II: Completing the proof}
\label{subsec:bootstrap}
We now have all the ingredients in place to complete the proof of \cref{lem:bootstrap}. We begin with the following simple consequence of \cref{lem:peakcomparison,lem:peaksurvival}, which is similar to \cref{lem:bootstrap} except that it gives an estimate inside a slab rather than in the full space.

\begin{lemma}
\label{lem:startingpoint}
The estimate
\begin{equation}
E_{p_c}^{-k-r,0}(-k;1-\eps)
\preceq_\eps \nexp\left[ o(k) +O(r)\right]
\label{eq:startingpoint}
\end{equation}
holds for every $k,r\geq 0$ and $\eps\in (0,1]$.
\end{lemma}

We stress that the implicit bounds inside the exponential on the right hand side of \eqref{eq:startingpoint} do not depend on the choice of $\eps \in (0,1]$.

\begin{proof}[Proof of \cref{lem:startingpoint}]
% Since $\alpha_{p_c}\geq 1$ it suffices to show that
% \begin{equation}
% \E_{p_c}\left[ \left(X_{-k}^{-k-r,0}(v)\right)^{1-\eps} \mid U_v = x\right]
% \preceq_\eps \nexp\left[-(\alpha_{p_c}-1)k+ o(k) +O(r)\right].
% \label{eq:startingpoint'}
% \end{equation}
First note that $\sP_\rho(k+\ell+r)$ contains the event that none of the edges incident to $\rho$ are open. Thus, its probability is bounded below by a positive $p$-dependent constant and  we have that
\begin{equation}
\E_p\left[ X_{-k-\ell-i}^{-k-\ell-r,0}(\rho) \mid \sP_\rho(k+\ell+r)\right]
\asymp_p
\E_p\left[ X_{-k-\ell-i}^{-k-\ell-r,0}(\rho) \,;\, \sP_\rho(k+\ell+r)\right].
\end{equation}
for every $p\in [0,1)$ and $k,\ell,r,i \geq 0$. Let $r_0$ be the constant from \cref{lem:peakcomparison}. Applying the tilted mass-transport principle as in \eqref{eq:PeakMTP_overview} together with \cref{lem:probinfsup}, we deduce that
\begin{multline}
\sum_{i=0}^{r_0} \E_{p_c}\left[ X_{-k-\ell-i}^{-k-\ell-r-r_0,0}(\rho)\,;\;\sP_\rho(k+\ell+r+r_0) \right]
\preceq \sum_{i=0}^{r_0}  \P_{p_c}\left(\rho \xleftrightarrow{L_{-r+i-r_0}(\rho)} L_{k+\ell+i,\infty}(\rho)\right) \\
\preceq \nexp\left[ O(r) - (\alpha_{p_c}-1) (k+\ell) \right],
\end{multline}
for every $r \geq 1$ and $k,\ell \geq 0$. 
% Moreover, it follows from the definitions that
% \begin{multline}
% \sup_{v\in V, x\in [0,1]} \sum_{i=0}^{r_0} \E_{p_c}\left[ X_{-k-\ell-i}^{-k-\ell-r-r_0,0}(v)\,;\;\sP_v(k+\ell+r+r_0), U_v=x \right] \\\preceq
% \sum_{i=0}^{r_0+2} \E_{p_c}\left[ X_{-k-\ell-i+1}^{-k-\ell-r-r_0-1,0}(\rho)\,;\;\sP_\rho(k+\ell+r+r_0-1) \right]
% \preceq \nexp\left[O(r) - (\alpha_{p_c}-1) (k+\ell) \right]
% \end{multline}
% for every $r\geq n_0$ also.
% Together with \cref{lem:peakcomparison} and \cref{lem:peaksurvival} this yields that
Applying \cref{lem:peakcomparison} and letting $v_0$ be as in \eqref{eq:peaksurvival_v0}, we obtain that
\begin{multline} 
 \sup_{v\in V, x\in [0,1]} \E_{p_c}\left[X_{-k}^{-k-r,0}(v)\mathbbm{1}\left( v \nxleftrightarrow{L_{-r-k-2,\infty}(v)} L_{\ell-1}(v) \right) \mid U_{v} = x \right]\\
\preceq 
\myfrac{\sum_{i=0}^{r_0} \E_{p_c}\left[ X_{-k-\ell-i}^{-k-\ell-r-r_0,0}(v_0)\,;\;\sP_{v_0}(k+\ell+r+r_0) \right]}{\P_{p_c}\left(v_0 \leftrightarrow L_{-\ell}(v_0) \mid \sP_{v_0} \right)}
% \preceq \myfrac{\sum_{i=0}^{r_0} \E_{p_c}\left[ X_{-k-\ell-i}^{-k-\ell-r-r_0,0}(\rho)\,;\;\sP_{\rho}(k+\ell+r+r_0) \right]}{\P_{p_c}\left(v_0 \leftrightarrow L_{-k}(v_0) \mid \sP_{v_0} \right)}
\\\preceq
\nexp\left[O(r) -(\alpha_{p_c}-1)k+ o(\ell) \right],
\label{eq:maladjusted}
\end{multline}
% so that
% \begin{equation}
% \sup_{v\in V, x\in [0,1]} \E_{p_c}\left[X_{-k}^{-k-r,0}(v) \mathbbm{1}\left( v \nxleftrightarrow{L_{-r-k-1,\infty}(v)} L_{\ell-1}(v) \right) \mid U_{v} = x \right]
% \preceq_\eps \nexp\left[O(r) -(\alpha_{p_c}-1)k+ \frac{\eps}{2} \ell \right]
% \end{equation}
for every $r\geq 1$ and $k,\ell\geq 0$. By adjusting the implicit constants if necessary, we may take the bound \eqref{eq:maladjusted} to hold for every $r,k,\ell\geq 0$.

% Take $\ell = \lceil \log n + (\alpha_{p_c}-1)k \rceil$. 
Now, applying the union bound and Markov's inequality yields that
\begin{multline}
\P_{p_c}\left(X_{-k}^{-k-r,0}(v) \geq n \mid U_{v}=x \right)
 \\\leq \frac{1}{n}\E_p\left[X_{-k,}^{-k-r,0}(v) \mathbbm{1}\left( v \nxleftrightarrow{L_{-r-k-2,\infty}(v)} L_{\ell-1}(v) \right) \mid U_{v} = x \right] + \P_{p_c}\left( v \xleftrightarrow{L_{-r-k-2,\infty}(v)} L_{\ell-1}(v) \mid U_{v}=x \right)\\
\preceq \frac{1}{n}\nexp\left[ O(r) -(\alpha_{p_c}-1)k+ o(\ell) \right] +  \nexp\left[O(k+r) -\alpha_{p_c}\ell\right] 
 % \preceq_\eps \frac{1}{n^{1-\eps/2}} \nexp\left[-\Big(\alpha_{p_c}-1-\frac{\eps}{2}\Big)k + Cr \right],
\end{multline}
for every $v\in V$, $x\in [0,1]$, and $r,k,\ell \geq 0$, where we have used \eqref{eq:maladjusted} and \cref{lem:probinfsup} respectively to bound the two terms on the second line.
Taking $\ell = \lceil \log n + C k \rceil$ for a sufficiently large constant $C$, we deduce that
\begin{multline}
\sup_{v\in V,x\in[0,1]}\P_{p_c}\left(X_{-k}^{-k-r,0}(v) \geq n \mid U_{v}=x \right)\\
\preceq \frac{1}{n}\, \nexp\left[ O(r) -(\alpha_{p_c}-1)k+ o(k) + o(\log n)\right] +  \frac{1}{n} \,\nexp\left[O(r)\right] 
 % \preceq_\eps \frac{1}{n^{1-\eps/2}} \nexp\left[-\Big(\alpha_{p_c}-1-\frac{\eps}{2}\Big)k + Cr \right],
\end{multline}
and hence that
\begin{multline}
\sup_{v\in V,x\in[0,1]}\P_{p_c}\left(X_{-k}^{-k-r,0}(v) \geq n \mid U_{v}=x \right)\\
\preceq_\eps \frac{1}{n^{1-\eps/2}}\, \nexp\left[ O(r) -(\alpha_{p_c}-1)k+ o(k) \right] +  \frac{1}{n} \,\nexp\left[O(r)\right] 
 % \preceq_\eps \frac{1}{n^{1-\eps/2}} \nexp\left[-\Big(\alpha_{p_c}-1-\frac{\eps}{2}\Big)k + Cr \right],
\end{multline}
for every $r,k \geq 0$ and $\eps \in (0,1]$.
 % in the second inequality. 
Multiplying both sides by $n^{-\eps}$ and summing over $n$ completes the proof.
% \[
% \inf_{v\in V, x\in [0,1]} \E_{p_c}\left[ \left(X_{-k}^{-k-r,0}(v)\right)^{1-\eps} \mid U_{v} = x\right]
% \preceq_\eps \nexp\left[ -(\alpha_{p_c}-1)k +Cr + o_\eps (k) \right]
% \]
\end{proof}

In the remainder of the section, we apply the tools developed in \cref{subsec:toolkit} to bootstrap from the slab estimate of \cref{lem:startingpoint} to the full-space estimate \cref{lem:bootstrap}.
The next lemma provides in particular an `upwards' version of the estimate \eqref{eq:startingpoint}.

\begin{lemma}
\label{lem:midway}
The estimates
\begin{align}
E_{p_c}^{0,k}(k;1-\eps) &\preceq_{\eps} \nexp\left[-k+o_\eps(k) \right] && \text{and}
\label{eq:midway1}
\\
% \intertext{holds for every $0<\eps \leq 1$ and $k\geq 0$, and there exists a constant $C$ such that}
E_{p_c}^{-\infty,r}(0;1-\eps) &\preceq_\eps \nexp\left[ O(r)\right]
\label{eq:midway2}
% \preceq_\eps \nexp\left[ Ck \right]
\end{align}
hold for every $0<\eps\leq 1$ and $k,r\geq 0$.
\end{lemma}

% We stress that the constant $C$ in \eqref{eq:midway2} does not depend on the choice of $0<\eps\leq 1$.
Again, we stress that the implicit constant inside the exponential on the right hand side of \eqref{eq:midway2} does not depend on the choice of $\eps \in (0,1]$.

\begin{proof}
Let $0<\eps \leq 1$, and let $\delta = ( 1 \wedge C^{-1}) \eps \leq \eps$, where $C=C_{p_c}$ is the implicit constant prefactor of $|r|$ in \eqref{eq:FKGshift}. Applying \cref{lem:HolderMTP}, we have that
% \begin{align*}
% \E_{p_c}\left[\left(X_k^{-r,k+r}(\rho)\right)^{1-\eps}\right]
% \preceq_\eps 
% \nexp\left[-(1-\eps)k\right] \E_{p_c}\left[\left(X_{-k-r}^{-k-r,r}(\rho)\right)^{1-\delta}\right]^{1-\eps}\E_{p_c}\left[\left(X_0^{-r,k+r}(\rho)\right)^{(1-\eps)\delta/\eps}\right]^{\eps}
% \end{align*}
% and hence, by Jensen's inequality,
\begin{align}
\E_{p_c}\left[\left(X_k^{-r,k+r}(\rho)\right)^{1-\eps}\right]
% &\preceq_\eps  
% \nexp\left[-(1-\eps)k\right] \E_{p_c}\left[\left(X_{-k}^{-k-r,r}(\rho)\right)^{1-\delta}\right]^{1-\eps} \E_{p_c}\left[\left(X_0^{-r,k+r}(\rho)\right)^{(1-\eps)}\right]^{\delta}\\
&\preceq \nexp\left[-(1-\eps)k\right] E_{p_c}^{-k-r,r}(-k;1-\delta)^{1-\eps} E_{p_c}^{-r,k+r}(0;1-\eps)^\delta.
\end{align}
Applying the estimate \eqref{eq:FKGshift} from \cref{lem:bootstrapping_etc} to both terms on the right-hand side we obtain that
\begin{multline}
\E_{p_c}\left[\left(X_k^{-r,k+r}(\rho)\right)^{1-\eps}\right]
\\\preceq_\eps 
 \nexp\left[-(1-\eps)k+(1-\eps)Cr + \delta C(k+r)\right] E_{p_c}^{-k-2r,0}(-k-r;1-\delta)^{1-\eps+\delta},
\end{multline}
and applying \cref{lem:startingpoint} yields that
\begin{align}
\E_{p_c}\left[\left(X_k^{-r,k+r}(\rho)\right)^{1-\eps}\right]
&\preceq_{\eps} 
% \nexp\left[-(1-\eps)\alpha_{p_c}k  + C\delta k + o_{\eps}(k) \right]
% \preceq_\eps
\nexp\left[-(1-2\eps)k  + o(k) + O(r)\right].
\label{eq:eps_decreasing}
\end{align}
Since $0<\eps \leq 1$ was arbitrary and the left hand side of \eqref{eq:eps_decreasing} is a decreasing function of $\eps$, it follows that in fact
\begin{equation}
\E_{p_c}\left[\left(X_k^{-r,k+r}(\rho)\right)^{1-\eps}\right]
\preceq_{\eps} 
\nexp\left[-k  + o_{\eps}(k) + O(r)\right],
\end{equation}
and applying \eqref{eq:supaverage} we deduce that
 % obtain 
\begin{equation}
E_{p_c}^{-r,k+r}(k;1-\eps)
\preceq_{\eps} 
\nexp\left[-k  + o_{\eps}(k) + O(r)\right]
\label{eq:kawaii}
\end{equation}
for every $0 < \eps \leq 1$ and $k,r \geq 0$. 
In particular, this yields the estimate \eqref{eq:midway1}. 
Applying the estimates \eqref{eq:upthendown}, \eqref{eq:FKGshift}, and \eqref{eq:kawaii} we obtain that
\begin{align}
E_{p_c}^{-\infty,r}\left(0;(1-\eps)^2\right)  &\leq 
\sum_{k \geq0} E_{p_c}^{-k,r}(-k; 1-\eps) E_{p_c}^{0,k+r}(k; 1-\eps)^{ 1-\eps}
\\
&\leq  \sum_{k \geq0} \nexp[O(r)] E_{p_c}^{-k-r,0}(-k-r; 1-\eps) E_{p_c}^{0,k+r}(k; 1-\eps)^{ 1-\eps}
\\
 &\leq \sum_{k\geq0} \nexp\left[
-  (1-\eps) k + o_\eps(k+r) + O(r)
\right]
\preceq_\eps \nexp\left[O(r)\right].
\end{align}
% for every $\eps>0$ as claimed.
The estimate \eqref{eq:midway2} follows since $0<\eps\leq 1$ was arbitrary.
\end{proof}

Next, we prove a half-space version of \cref{lem:bootstrap} by interpolating between the estimates of \cref{lem:startingpoint} and of \cref{lem:midway}.
Once \cref{lem:3/4way} is proven, it will remain only to improve these half-space estimates to full-space estimates.

\begin{lemma}
\label{lem:3/4way}
The estimates
\begin{align}
E_{p_c}^{-\infty,0}(-k;1-\eps) &\preceq_\eps \nexp\left[ o_\eps(k)\right] && \text{and}\\
% \intertext{and}
E_{p_c}^{0,\infty}(k;1-\eps) &\preceq_\eps \nexp\left[ -k + o_\eps(k)\right]&&
\end{align}
hold for every $0<\eps\leq 1$ and $k\geq 0$.
\end{lemma}

% We are now ready to prove \cref{lem:bootstrap}. 
The proof will apply Lyapunov's interpolation inequality (a special case of H\"older's inequality), which we state here in full generality for clarity: If $X$ is a non-negative random variable, then 
\begin{equation}
\E\left[ X^{(1-\lambda)a + \lambda b} \right] \leq \E\Big[ X^a \Big]^{1-\lambda}\E\left[X^b\right]^\lambda
\end{equation}
for every $a,b\geq0$, and $\lambda \in [0,1]$.

\begin{proof}[Proof of \cref{lem:3/4way}]
Define
\begin{align*}
\cE_1 &= \left\{\eps \in (0,1] : E_{p_c}^{-\infty,0}(-k;1-\eps) \preceq_{\eps} \nexp\left[o_\eps(k)\right] \right\} && \text{ and}\\
\cE_2 &= \left\{\eps \in (0,1] : E_{p_c}^{0,\infty}(k;1-\eps) \preceq_{\eps} \nexp\left[-k+o_\eps(k)\right] \right\}.
\end{align*}
We wish to show that $\cE_1=\cE_2=(0,1]$.
	This is done via a bootstrapping procedure. Let $C_1$ and $C_2$ be the implicit constant prefactors of $r$ and $|r|$ from \eqref{eq:midway2} and \eqref{eq:FKGshift} respectively, let $C=C_1+C_2$, and let $\eta=\min\{1/2,1/4C\}$. It holds trivially that $1 \in \cE_1$, since 
	 \begin{equation}E_{p_c}^{-\infty,0}(-k,0)
	 =\sup_{v\in V,x\in [0,1]} \P_p\left(v\xleftrightarrow{L_{-\infty,0}(v)} L_{-k}(v)\right) \leq 1.
	 \end{equation}
	 We claim that the following hold:
	 \begin{enumerate}
	 		  \item
	 If $\eps \in \cE_1$ then $\eps'\in \cE_2$ for every $(1-\eta)\eps < \eps' \leq 1$, and
	  % $(1-\eta)\eps \in \cE_2$, and
	  \item
	  If $\eps \in \cE_2 \setminus \{1\}$ then $\eps' \in \cE_1$ for every $\eps < \eps' \leq 1$.
	  % \intertext{and}
	 \end{enumerate}
	 Once each of these are established, it will follow by induction that $((1-\eta)^{i},1] \subseteq \cE_1 \cap \cE_2$ for every $i\geq 0$, and hence that $\cE_1 = \cE_2 = (0,1]$ as desired.

	 For item 1, let $0<\eps\leq 1$ be such that $\eps\in \cE_1$, let $(1-\eta)\eps < \eps' \leq 1$ and let $0<\delta\leq 1$ be such that $\eps'=1-(1-\delta)(1-(1-\eta)\eps)$. Then \eqref{eq:upthendown} implies that
	 \begin{align}
	 E_{p_c}^{0,\infty}(k;1-\eps')&\leq \sum_{\ell\geq k} E_{p_c}^{0,\ell}(\ell;1-\delta) E_{p_c}^{-\ell,0}(k-\ell;1-(1-\eta)\eps)^{1-\delta}\\
	 &\leq \sum_{\ell\geq k} E_{p_c}^{0,\ell}(\ell;1-\delta) E_{p_c}^{-\infty,0}(k-\ell;1-(1-\eta)\eps)^{1-\delta}.
	 \end{align}
Since $1-(1-\eta)\eps = (1-2\eta)(1-\eps)+2\eta(1-\eps/2)$, we may apply Lyapunov's inequality to the second term in the sum to deduce that
	 \begin{equation}
	 E_{p_c}^{0,\infty}(k;1-\eps')\\
	 \leq \sum_{\ell\geq k} E_{p_c}^{0,\ell}(\ell;1-\delta)
	 E_{p_c}^{-\infty,0}(k-\ell; 1-\eps)^{(1-\delta)(1-2\eta)}
	  E_{p_c}^{-\infty,0}(k-\ell;1-\eps/2)^{2\eta (1-\delta)}.
	 \end{equation}
% The first term in the sum is controlled by 
Applying \eqref{eq:midway1} to control the first term in the sum, the assumption that $\eps \in \cE_1$ to control the second term, and \eqref{eq:midway2} and \eqref{eq:FKGshift} to control the third term, we obtain that
	 \begin{equation}
	 E_{p_c}^{0,\infty}(k;1-\eps')\\
	 \preceq_{\eps,\delta} \sum_{\ell\geq k} \nexp\left[-\ell+o_\delta(\ell) + o_\eps(k-\ell) + 2\eta(1-\delta) C(\ell-k) \right],
		 \end{equation}
and our choice of $\eta$ yields that
	 \begin{align}
	 E_{p_c}^{0,\infty}(k;1-\eps') 
	 &\preceq_{\eps,\delta} \sum_{\ell \geq k} \nexp\left[-\ell + \frac{1}{2}(\ell-k) + o_{\eps}(\ell-k) +o_\delta(\ell)  \right]
\nonumber
\\
	 &\asymp_{\eps,\delta}\sum_{r\geq 0} \nexp\left[-k - \frac{1}{2}r + o_{\eps}(r) +o_\delta(k+r)  \right]\preceq_{\eps,\delta} \nexp\left[-k+o_{\eps,\delta}(k)\right].
	 \end{align}
Since $\delta$ was chosen as a function of $\eps$ and $\eps'$, it follows that $\eps'\in \cE_2$ as claimed.

For item 2, let $\eps \in \cE_2$, let $\eps <\eps' \leq 1$, and let $0<\delta\leq 1$ be such that $\eps'=1-(1-\eps)(1-\delta)$. Then \eqref{eq:downthenup} implies that
	 \begin{align}
	 E_{p_c}^{-\infty,0}(-k;1-\eps') &\leq \sum_{\ell\geq k} E_{p_c}^{-\ell,0}(-\ell;1-\delta) E_{p_c}^{0,\ell}(\ell-k;1-\eps)^{1-\delta}
\nonumber
	 \\
	 &\leq 
	 \sum_{\ell\geq k} E_{p_c}^{-\ell,0}(-\ell;1-\delta) E_{p_c}^{0,\infty}(\ell-k;1-\eps)^{1-\delta}
	 \end{align}
	 for $k \geq 0$.
	 Applying \cref{lem:startingpoint} to control the first term in the sum and the assumption that $\eps \in \cE_2$ to control the second, we obtain that
 \begin{align}
	 E_{p_c}^{-\infty,0}(-k;1-\eps') 
	 &\preceq_{\eps,\delta}
	 \sum_{\ell\geq k} \nexp\left[ o_\delta(\ell)-(1-\delta)(\ell-k) + o_\eps(\ell-k) \right]\preceq_{\eps,\delta} \nexp\left[ o_{\eps,\delta}(k) \right].
	 \end{align}
	 % so that $1-(1-\eps)(1-\delta)\in \cE_2$ for every $\delta>0$ 
	Since $\delta$ was chosen as a function of $\eps$ and $\eps'$, it follows that $\eps'\in \cE_1$ as claimed.
	\end{proof}

	We are now ready to complete the proof of \cref{lem:bootstrap}.

\begin{proof}[Proof of \cref{lem:bootstrap}]
Let $0<\eps\leq1$ and let $0<\delta\leq \eps$. 
Applying \cref{lem:HolderMTP} and the estimate \eqref{eq:FKGshift} of \cref{lem:bootstrapping_etc} yields that
\begin{multline}
\E_{p_c}\left[ \left(X_{-k}^{-k-r,\infty}(\rho)\right)^{1-\eps}\right] \preceq 
\nexp\left[(1-\eps)k\right]E_{p_c}^{-r,\infty}(k;1-\delta)^{1-\eps}E_{p_c}^{-k-r,\infty}(0; 1-\eps)^\delta\\
\preceq \nexp\left[(1-\eps)k + \delta O(k+r) + O(r)\right]E_{p_c}^{0,\infty}(k+r;1-\delta)^{1-\eps}E_{p_c}^{0,\infty}(k+r;1-\eps)^{\delta}
\end{multline}
Applying \cref{lem:3/4way}, we obtain that
\begin{align}
\E_{p_c}\left[ \left(X_{-k}^{-k-r,\infty}(\rho)\right)^{1-\eps}\right] 
&\preceq_{\eps,\delta} \nexp\left[ \delta O(k+r) + O(r) + o_{\eps,\delta}(k+r) \right],
\end{align}
and since $0<\delta \leq \eps$ was arbitrary we obtain that
\begin{align}
\E_{p_c}\left[ \left(X_{-k}^{-k-r,\infty}(\rho)\right)^{1-\eps}\right] 
&\preceq_{\eps} \nexp\left[ o_\eps(k) + O(r) \right].
\end{align}
We then obtain from this and \eqref{eq:supaverage} that
\begin{equation}
\label{eq:5/6way1}
E_{p_c}^{-k,\infty}(-k;1-\eps) \preceq_\eps \nexp\left[ o_\eps(k)\right]
\end{equation}
for every $0<\eps \leq 1$ and $k\geq 0$.
A similar proof yields that the analogous `upwards' bound
\begin{equation}
\label{eq:5/6way2}
E_{p_c}^{-\infty,k}(k;1-\eps) \preceq_\eps \nexp\left[ -k + o_\eps(k)\right]
\end{equation}
holds for every $0<\eps \leq 1$ and $k\geq 0$.

 We now apply \eqref{eq:5/6way1} and \cref{lem:3/4way} together with \eqref{eq:downthenup} to deduce that
\begin{align}
E_{p_c}^{-\infty,\infty}\bigl(-k;(1-\eps)^2\bigr)&\preceq_\eps \sum_{\ell \geq k} E_{p_c}^{-\ell,\infty}(-\ell;1-\eps)E_{p_c}^{0,\infty}(\ell-k;1-\eps)^{1-\eps}
\nonumber
\\
&\preceq_\eps
\sum_{\ell\geq k} \nexp\left[ -(1-\eps)(\ell-k) + o_\eps(\ell) + o_\eps(\ell-k) \right]
\preceq_\eps \nexp\left[o_\eps(k)\right]
\end{align}
for every $k\geq 0$ and $0<\eps \leq 1$. 
% Since the left hand side of \eqref{eq:notJensen1} is decreasing in $\eps$, it follows that
% \begin{align}
% E_{p_c}^{-\infty,\infty}(-k;1-\eps)&
% \preceq_\eps \nexp\left[-k+o_\eps(k)\right]
% \end{align}
% for every $k\geq 0$ and $0<\eps\leq 1$.
 Similarly, applying \eqref{eq:5/6way2} and \cref{lem:3/4way} together with \eqref{eq:upthendown} yields that
\begin{align}
E_{p_c}^{-\infty,\infty}\bigl(k;(1-\eps)^2\bigr)&\preceq_\eps \sum_{\ell \geq k} E_{p_c}^{-\infty,\ell}(\ell;1-\eps)E_{p_c}^{-\infty,0}(k-\ell;1-\eps)^{1-\eps}\\
&\preceq_\eps
\sum_{\ell\geq k} \nexp\left[ -\ell + o_\eps(\ell) + o_\eps(\ell-k) \right]
\preceq_\eps \nexp\left[-k+o_\eps(k)\right]
\end{align}
for every $k\geq 0$ and $0<\eps \leq 1$. 
	Since $0<\eps\leq 1$ was arbitrary, this concludes the proof of \cref{lem:bootstrap}, and thus also the proof of our main results, \cref{thm:pcpt,thm:pcph,thm:pcpu}.
\end{proof}

\section{Critical exponents and the triangle condition}
\label{sec:exponents}

% \begin{proof}
% \subsection{Overview}

In this section we prove \cref{thm:triangle,thm:criticalexponents}. We begin by deducing \cref{thm:triangle} from \cref{thm:pcpt}.

\begin{lemma}
\label{lemma:diamondimpliestriangle}
Let $G$ be a locally finite graph and suppose that $\Gamma \subseteq \operatorname{Aut}(G)$ is transitive and nonunimodular.
Then 
\[\nabla_p(v) \leq \Big(\sup_{u\in V} \chi_{p,\lambda}(u)\Big)^3\] for every $v\in V$, $p\in [0,1]$ and $\lambda \in \R$. In particular, $\nabla_p<\infty$ for every $0\leq p < p_t$.
\end{lemma}

\begin{proof}
Since $\Delta(v,v)=1$, we have the   trivial inequality
\begin{equation}
\nabla_p(v) = \sum_{x,y \in V} \tau_p(v,x) \tau_p(x,y) \tau_p(y,v) \leq \sum_{x,y,z \in V} \tau_p(v,x) \tau_p(x,y) \tau_p(y,z) \Delta^\lambda(v,z)
\end{equation}
and using the cocycle identity $\Delta(v,z)=\Delta(v,x)\Delta(x,y)\Delta(y,z)$  yields that
\begin{equation} \nabla_p(v) \leq \sum_{x\in V} \tau_p(v,x)\Delta(v,x)^{\lambda} \sum_{y\in V} \tau_p(x,y)\Delta(x,y)^{\lambda} \sum_{z\in V} \tau_p(y,z) \Delta(y,z)^{\lambda} \leq \Big(\sup_{u\in V} \chi_{p,\lambda}(u)\Big)^3
\end{equation}
as claimed.
\end{proof}

\begin{remark}
In \cite[Theorem 2.9]{Hutchcroft2019Hyperbolic}, we prove that the matrix of connection probabilities $T_p(u,v)=\bP_p(u \leftrightarrow v)$ defines a bounded operator on $L^2(V)$ for every $p<p_t$. The consequences of this property are developed at length in \cite{Hutchcroft2019Hyperbolic} and \cite{1901.10363}. 
% It follows in particular that the \emph{open} triangle condition is satisfied for every $p<p_t$.
\end{remark}

From here, it remains to derive \cref{thm:criticalexponents} from \cref{thm:triangle} together with the estimates we derived in \cref{sec:Fekete}.
As stated in the introduction, and observed by Schonmann \cite{MR1833805}, most aspects of the proofs of \cite{MR762034,aizenman1987sharpness,MR1127713,MR923855,MR2551766,MR2779397} generalise unproblematically to quasi-transitive graphs satisfying the triangle condition at $p_c$. In the interest of space, we do not go through these parts of the proofs here. The reader is encouraged to consult the original papers, as well as \cite{grimmett2010percolation,heydenreich2015progress}. However, there are four points that require more serious attention, two of which have already been addressed in the literature and two of which we address  here. First, we have the two that have already been addressed:
\begin{enumerate}
\item Barsky and Aizenman's proof of the upper bounds of \eqref{exponent:volume} and \eqref{exponent:theta} use the \emph{open triangle condition} rather than the triangle condition as we have stated it. They showed that the two conditions are equivalent in the case of $\Z^d$ using Fourier analysis, but this proof does not generalise to other transitive graphs. Fortunately, however, the two conditions were shown to be equivalent at $p_c$ for all transitive graphs by Kozma \cite{MR2779397}, who used the theory of unbounded operators. 
% [In \cite{MR2779397} it is stated that the equivalence holds for all $p\in [0,1]$, but an inspection of the proof reveals that the assumption $p\leq p_c$ is also used.]
 Kozma's proof is easily generalised to the quasi-transitive case. 

\item Kozma and Nachmias's \cite{MR2551766} proof of the intrinsic radius upper bound in \eqref{exponent:intradius} assumes both unimodularity (implicitly in their proof of Lemma 3.2) and polynomial growth (in the deduction of Theorem 1.2, part $(i)$ from Lemma 3.2). However, the use of both assumptions are confined to the proof of a single estimate (Theorem 1.2, part $(i)$). Fortunately, Sapozhnikov \cite{sapozhnikov2010upper} found a very short and simple proof of this estimate that works for any bounded degree graph satisfying the upper bound of \eqref{exponent:susceptibility} uniformly over all its vertices, thus  rendering this problem unproblematic.   
\end{enumerate}

Secondly, we have the two issues that we must address ourselves. 
% The changes required to take these issues into account turn out to all be rather straightforward, but are not merely superficial as is the case for the other parts of the proofs.
% We now indicate what these points are and how we address them.

\begin{enumerate}
	\item Aizenman and Newman's  \cite{MR762034} proof of the upper bound of \eqref{exponent:susceptibility} and Nguyen's proof of the lower bound of \eqref{exponent:gap} both rely on the differential inequality
\begin{equation}
\label{eq:meanfield_diffineq_overview}
\frac{d}{dp}\chi_p \succeq \chi_p^2 \qquad \text{ as $p\uparrow p_c$}.
\end{equation}
	Aizenman and Newman's derivation of this inequality from the triangle condition contains two steps that do not generalize to our setting:
\begin{enumerate}
\item 	There is an implicit use of the (unimodular form of the) mass-transport principle in \cite[Eq. 6.4]{MR762034}.  (The fact that it does so was observed by Schonmann \cite{MR1833805}; see the discussion around eq.\ 3.14 of that paper.)  
\item The argument of \cite[Lemma 6.3]{MR762034} relies on some topological features of $\Z^d$ and does not generalize to arbitrary quasi-transitive graphs. (Indeed, it does not work on a tree.) 
\end{enumerate}

	\item Kozma and Nachmias's \cite{MR2551766} proof of the lower bound of \eqref{exponent:intradius} uses the asymptotics for the two-point function for critical high-dimensional percolation, due to Hara, van der Hofstad, and Slade \cite{MR1959796} and Hara \cite{MR2393990}. As such, it does not generalise to other quasi-transitive graphs satisfying the triangle condition. Moreover, the lower bound of \eqref{exponent:radius} is \emph{false} for $\Z^d$, where the correct high-dimensional exponent for the extrinsic diameter is $2$ instead of $1$ (see \cite{MR2748397} for an explanation of this disparity).
	% whose 
\end{enumerate}

% In fact, the required changes to take these issues into account are all very straightforward

Thus, in order to deduce \cref{thm:criticalexponents} from \cref{thm:triangle}, it suffices to prove the upper bound of \eqref{exponent:susceptibility} and the lower bounds of \eqref{exponent:gap}, \eqref{exponent:radius}, and \eqref{exponent:intradius}. While it is certainly possible to adapt the original proofs to our setting (and in particular to prove a differential inequality of the form \eqref{eq:meanfield_diffineq_overview}), this requires some work, and we are fortunate that we may instead apply \cref{thm:pcpt} and the results of \cref{sec:Fekete} to give very quick proofs in the nonunimodular setting. 

% In fact, we also generalise this differential inequality to the tilted susceptibility for $\lambda \in [0,1]$, thereby obtaining results on critical exponents for the tilted susceptibility near $p_c(\lambda)$ ([ref]).
% 
% However, having the differential inequality is required to obtain the lower bound for the gap exponent.)
% 
% 
 % The proof we give of the lower bound of \eqref{exponent:intradius} applies to any quasi-transitive graph satisfying the triangle condition at $p_c$, so that we obtain a generalisation of Kozma and Nachmias's intrinsic diameter exponent theorem to all such graphs.

% We also give an alternative, simpler proof of the upper bound of \eqref{exponent:susceptibility} and of both the upper and lower bounds of \eqref{exponent:radius} and \eqref{exponent:intradius} that is specific to the nonunimodular setting. The alternative proof of the upper bound of \eqref{exponent:susceptibility} does not use differential inequalities or the triangle condition at all, whereas the alternative proof of \eqref{exponent:radius} and \eqref{exponent:intradius} uses them only implicitly via an appeal to the cluster volume upper bound \eqref{exponent:volume}.
 % and does not use differential inequalities.
 % (However, this method does not apply to the tilted susceptibility when $\pcl{\lambda}=p_t$, and cannot be used to yield the lower bound on the gap exponent via the method of Nguyen \cite{MR923855}, which requires the differential inequality.)
  % \end{proof}

We begin with the upper bound of \eqref{exponent:susceptibility}. We in fact prove the following generalization of the result to the tilted susceptibility.

\begin{thm}
\label{thm:tiltedsusceptibilityexponent}
Let $G$ be a connected, locally finite graph, and let $\Gamma$ be a quasi-transitive nonunimodular subgroup of $\Aut(G)$. Then for every $\lambda \in \R$ such that $p_c(\lambda)<p_t$ we have that
\[
\chi_{p,\lambda}(v) \asymp_\lambda (p_c(\lambda)-p)^{-1} \qquad \text{ as $p \uparrow p_c(\lambda)$}
\]
for every $v\in V$.
\end{thm}

\begin{proof}[Proof of \cref{thm:tiltedsusceptibilityexponent}]
By symmetry, it suffices to consider the case $\lambda > 1/2$. The lower bound is provided by \cref{prop:tiltedmeanfieldlowerbound}. For the upper bound, \cref{cor:lambdatoalpha} and the assumption that $p_c(\lambda)<p_t$ guarantee that $\beta_p=\lambda$, and we have by \cref{prop:tamedecay} that
\begin{equation}
\E_{\pcl{\lambda}}\left[X_k^{-\infty,\infty}(v)\right] \preceq_\lambda \begin{cases} \nexp\left[-\lambda k  \right] &k \geq0\\
\nexp\left[(\lambda-1)k \right] & k < 0,
\end{cases}
\end{equation}
and hence in particular that
$\tau_p(x,y) \preceq_\lambda \Delta^\lambda(x,y) \wedge \Delta^\lambda(y,x)$ 
for every $0\leq p \leq p_c(\lambda)$ and $x,y \in V$. Using the inequality 
\begin{equation}
\log \tau_{p'}(x,y) \leq \frac{\log p'}{\log p}\log \tau_p(x,y),
\end{equation}
which holds in any graph \cite[Theorem 2.38]{grimmett2010percolation}, we deduce by calculus that there exists a constant $c=c_\lambda$ such that
\begin{equation}
\E_{\pcl{\lambda}-\eps}\left[X_k^{-\infty,\infty}(v)\right] \preceq_\lambda \begin{cases} \nexp\left[-(\lambda +c\eps) k  \right] &k \geq0\\
\nexp\left[(\lambda-1+c\eps)k \right] & k < 0.
\end{cases}
\end{equation}
The result follows from this together with \eqref{eq:XtoChi}.
\end{proof}

% \subsection{The radius and intrinsic radius}

% In this section we establish the upper and lower bounds of \eqref{exponent:radius} and \eqref{exponent:intradius}. 

\begin{corollary}
Let $G$ be a connected, locally finite graph, and let $\Gamma$ be a quasi-transitive nonunimodular subgroup of $\Aut(G)$. Then 
\begin{equation}
\chi^{(k+1)}_p(v)/\chi^{(k)}_p(v) \succeq (p_c-p)^{-2} \qquad \text{ as $p \uparrow p_c$}
\end{equation}
for every $v\in V$ and $k\geq 1$.
\end{corollary}

\begin{proof}
Holder's inequality implies that the left hand side is increasing in $k$, so that it suffices to consider the case $k=1$. An inequality of Durrett and Nguyen \cite[Section 5]{durrett1985thermodynamic}, which holds for all quasi-transitive graphs with $p_c<1$, states that
\begin{equation}
\label{eq:DNDiffIneq}
\frac{d}{dp} \chi_p(v) \preceq \sqrt{\chi_p(v) \chi_p^{(2)}(v)}
\end{equation}
for every $v\in V$ and $p_c/2 \leq p < p_c$. On the other hand, our assumptions together with \cref{thm:tiltedsusceptibilityexponent} yield that $\chi_p(v) \asymp (p_c-p)^{-1}$ for every $v\in V$ and $p_c/2 \leq p < p_c$, and it follows by the mean-value theorem that there exists a constant $C$ such that for every $v\in V$ and $\eps>0$ there exists $p \in [p_c-C\eps,p_c-\eps]$ such that
 $\frac{d}{dp} \chi_p(v) \geq \eps^{-2}$. Thus, it follows from \eqref{eq:DNDiffIneq} that 
  for every $v\in V$ and $\eps>0$ there exists $p \in [p_c-C\eps,p_c-\eps]$ such that
  \begin{equation}
\chi_p^{(2)}(v) \succeq \eps^{-2} \chi_p^2(v) \succeq \eps^{-3}.
  \end{equation}
  The claim follows since $\chi_p^{(2)}(v)$ is an increasing function of $p$.
\end{proof}

It remains only to prove the lower bounds of \eqref{exponent:intradius} and \eqref{exponent:radius}. A more general derivation of this lower bound is given in \cite[Section 3]{1901.10363}.

\begin{prop}
\label{prop:extrad}
Let $G$ be a connected, locally finite, quasi-transitive graph, and suppose that $\Aut(G)$ has a quasi-transitive nonunimodular subgroup. Then
\begin{equation}
\label{eq:extrad}
\bP_{p_c}\left(
\operatorname{rad}_\mathrm{int}(K_v) \geq n 
\right) \geq
\bP_{p_c}\left(
\operatorname{rad}(K_v) \geq n 
\right) \succeq n^{-1}
\end{equation}
for every $v\in V$ and $n\geq 1$. 
\end{prop}

\begin{proof}
The first inequality is trivial.
Let $\Gamma \subseteq \Aut(G)$ be quasi-transitive and nonunimodular, and let $(X_n^{-\infty,\infty})_{n\in \Z}=(X_n^{-\infty,\infty}(\rho))_{n\in \Z}$ be as in \cref{sec:Fekete}.  \cref{thm:pcpt} and \cref{cor:lambdatoalpha} imply that $\beta_{p_c}=1$, and \cref{prop:tamesecondmoment} implies that
\begin{equation}
\P_{p_c}\left(\operatorname{rad}(K_\rho) \geq n \right) \geq \P_{p_c}\left(X_{-n}^{\infty,\infty} > 0\right) \geq \frac{\E_{p_c}\left[X_{-n}^{-\infty,\infty}\right]^2}{\E_{p_c}\left[\left(X_{-n}^{-\infty,\infty}\right)^2\right]} \succeq \frac{1}{n},
\end{equation}
where the first inequality follows by definition of the uniform separating layer decomposition. The claim for general $v\in V$ follows by quasi-transitivity and the Harris-FKG inequality.
\end{proof}

\section{Remarks, examples, and open problems}
\label{sec:closing}

\subsection{Different automorphism groups on the tree}
\label{subsec:examplestree}

% \begin{example}
Let $T$ be the $d$-regular tree. The most obvious choice of a nonunimodular transitive subgroup of $T$ is the group $\Gamma$ consisting of those automorphisms of $T$ that fix some given end $\xi$ of $T$. We can trivially compute that
$\alpha_p = \beta_p = \log_{d-1}(1/p)$,
and it follows that
% may be verified via direct calculation
% that
\begin{equation} p_c(G,\Gamma,\lambda) = (d-1)^{-\max\{\lambda,1-\lambda\}}\end{equation}
for every $\lambda \in \R$.
% and hence that
In particular, 
$p_\mathrm{t}(G,\Gamma)=p_{1/2}(G,\Gamma)=(d-1)^{-1/2}$.
Moreover, for $p<\pcl{\lambda}$ we can compute the tilted susceptibility to be
\begin{equation}
\chi_{p,\lambda} = \frac{1-p^2}{(1-(d-1)^{1-\lambda}p)(1-(d-1)^\lambda p)}.
\end{equation}
(This can be done either by a direct counting argument or by solving a system of linear equations as is done in the next example, below.)
% \end{example}
Thus, we see that for $\lambda\neq 1/2$,
$\chi_{\pcl{\lambda}-\eps,\lambda}$ grows like $\eps^{-1}$ as $\eps\to0$, as stated in \cref{thm:tiltedsusceptibilityexponent}, while at $\lambda=1/2$ we have instead that
% \frac{1-(d-1)^{-2+2\lambda}}{(d-1)^{1-\lambda}-(d-1)^\lambda} \eps^{-1} +O(1) & \lambda >1/2
% \\
% 
\begin{equation}\chi_{p_{1/2}-\eps,1/2}=\frac{d-2}{d-1} \eps^{-2}.\end{equation}
% so that the exponent there is $2$ rather than $1$.
This shows that \cref{thm:tiltedsusceptibilityexponent} cannot be extended in general to the case $\lambda=1/2$. 
% & \lambda =1/2.

However, fixing an end is not the only way to get a nonunimodular automorphism group of $T$. Indeed, suppose that $d = 4$. We define a $(1,1,2)$-\textbf{orientation} of $T$ to be a (partial) orientation of the edge set of $T$ such that every vertex has one oriented edge emanating from it, two oriented edges pointing into it, and one unoriented edge incident to it. Fix one such orientation of $T$, and let $\Gamma'$ be the group of automorphisms of $T$ that preserve the orientation. It turns out that the behaviours of various tilted quantities are very different with respect to the two different groups $\Gamma$ and $\Gamma'$. 
 
We define the \textbf{height difference} $h(u,v)$ between two vertices $u$ and $v$ in $T$ to be the the number of oriented edges that are crossed in the forward direction minus the number of oriented edges that are crossed in the reverse direction in the unique simple path from $u$ to $v$ in $T$, and observe that the modular function can be expressed as 
$\Delta_{\Gamma'}(u,v)=2^{h(u,v)}$.
% The modular function $\Delta_{\Gamma'}$ is given by 

\newcommand{\chiup}{\chi_{p,\lambda}^+}
\newcommand{\chiacross}{\chi_{p,\lambda}^0}
\newcommand{\chidown}{\chi_{p,\lambda}^-}

%%%%%%%%%%%Tikz Figure

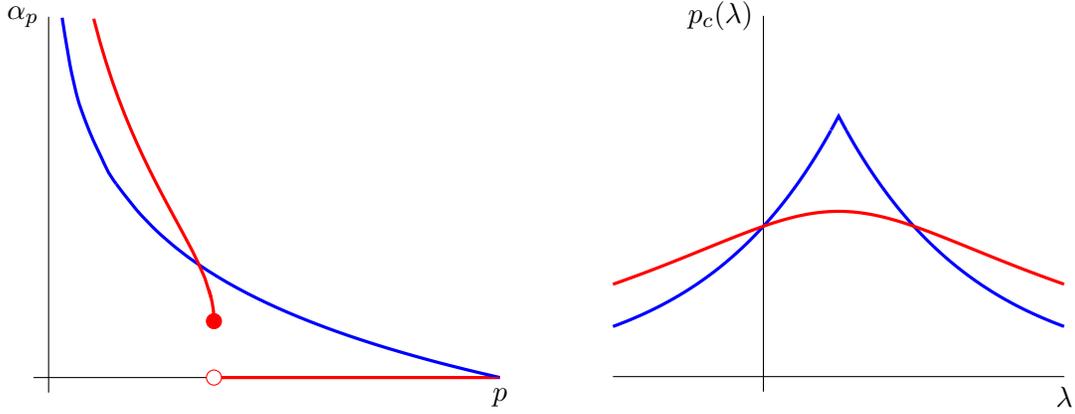
\begin{figure}
\centering
\begin{tikzpicture}
      \draw[-,thin] (-0.2,0) -- (6*0.366406859,0);
      \draw[-,red,very thick] (6*0.366406859,0) -- (6,0) node[below,black] {$p$};
      % \draw[-,thin] (6,0) -- (6.2,0);
      \draw[-,thin] (0,-0.2) -- (0,4.8) node[left,black] {$\alpha_p$};
      \draw[scale=6,domain=0.03:0.15,smooth,very thick,variable=\x,blue,samples=8] plot ({\x},{-0.25*ln(\x)/ln(3)});
      \draw[scale=6,domain=0.15:1,smooth,very thick,variable=\x,blue,samples=20] plot ({\x},{-0.25*ln(\x)/ln(3)});
      \draw[scale=6,domain=0.1:0.36,very thick,variable=\x,red,samples=20] plot ({\x},{0.25*log2((3*pow(\x,2)-\x+1+pow(9*pow(\x,4)-6*pow(\x,3)-pow(\x,2)-2*\x+1,0.5))/(2*\x))});
      \draw[scale=6,domain=0.355:0.3664,very thick,variable=\x,red,samples=30] plot ({\x},{0.25*log2((3*pow(\x,2)-\x+1+pow(9*pow(\x,4)-6*pow(\x,3)-pow(\x,2)-2*\x+1,0.5))/(2*\x))});
      \draw[scale=6,domain=0.366:0.3665,very thick,variable=\x,red,samples=30] plot ({\x},{0.25*log2((3*pow(\x,2)-\x+1+pow(max(9*pow(\x,4)-6*pow(\x,3)-pow(\x,2)-2*\x+1,0),0.5))/(2*\x))});
      \draw[-,red] (6*0.366406859,0) -- (6,0);
      \draw[red,fill=white] (6*0.366406859,0) circle [radius=0.1]; 
      \draw[red,fill=red] (6*0.366406859,0.5*6*0.25) circle [radius=0.1]; 
      % \draw[scale=0.5,domain=-3:3,smooth,variable=\y,red]  plot ({\y*\y},{\y});
\end{tikzpicture}
    \hspace{1cm}
\begin{tikzpicture}
      \draw[-,thin] (-2,0) -- (4,0) node[below] {$\lambda$};
      % \draw[-,red,very thick] (6*0.366406859,0) -- (6,0) node[below,black] {$p$};
      % \draw[-,thin] (6,0) -- (6.2,0);
      \draw[-,thin] (0,-0.2) -- (0,4.8) node[left,black] {$\pcl{\lambda}$};
      \draw[scale=2,domain=-1:0.45,smooth,very thick,variable=\x,blue,samples=100] plot ({\x},{3*pow(3,-max(\x,1-\x))});
      \draw[scale=2,domain=0.45:0.55,very thick,variable=\x,blue,samples=3] plot ({\x},{3*pow(3,-max(\x,1-\x))});
      \draw[scale=2,domain=0.55:2,smooth,very thick,variable=\x,blue,samples=100] plot ({\x},{3*pow(3,-max(\x,1-\x))});
      \draw[scale=2,domain=-1:2,smooth,very thick,variable=\x,red,samples=100] plot ({\x},{0.5*(pow(2,\x)+pow(2,1-\x)+1-pow(max(0,pow(pow(2,\x)+pow(2,1-\x)+1,2)-12),0.5))});
      % \draw[scale=0.5,domain=-3:3,smooth,variable=\y,red]  plot ({\y*\y},{\y});
    \end{tikzpicture}
    \caption{Comparison of $\alpha_p$ and $\pcl{\lambda}$ for the $4$-regular tree with respect to the automorphism group fixing an end (blue) and the automorphism group fixing a $(1,1,2)$-orientation (red). Note that the second figure is formed by reflecting the first around the line $\alpha_p=1/2$ and then rotating; this follows from \cref{cor:lambdatoalpha} and the fact that, in both examples, $p_c(\lambda)<p_t$ for every $\lambda \neq 1/2$. The two curves in the left figure intersect at $(p_c,1)$: this intersection must occur since $\pcl{0}=\pcl{1}=p_c$ does not depend on the choice of the automorphism group, and $\alpha_{p_c}=1$ for any choice of automorphism group.}
\end{figure}

%%%%%%%%%%% End TikZ figure

% Let $\chiup$, $\chiup$ be the expectation of $|K_v|_{v,\lambda}$  
We now compute $\chi_{p,\lambda}$ and $\pcl{\lambda}$ for every $p\in(0,1)$ and $\lambda \in \R$. 
Let $v$ be a vertex of $T$, and let $e^\uparrow$, $e^\rightarrow$, and $e^\downarrow$ be, respectively, an oriented edge emanating from $v$, an unoriented edge incident to $v$, and an oriented edge pointing into $v$. Let 
\begin{align*}
\chiup &= \sum_{u\in V} \P_p(u \leftrightarrow v \text{ off }e^\uparrow\;) \Delta^\lambda(v,u),\\
\chiacross &= \sum_{u\in V} \P_p(u \leftrightarrow v \text{ off }e^\rightarrow) \Delta^\lambda(v,u), \qquad \text{ and}\\
\chidown &= \sum_{u\in V} \P_p(u \leftrightarrow v \text{ off }e^\downarrow\;) \Delta^\lambda(v,u).
\end{align*}
% Observe that we have the identity
Considering the four parts of the cluster of a vertex that are connected to the vertex via the four different edges incident to the origin leads to the expression
\begin{equation}\chi_{p,\lambda}=1+2^{1-\lambda} p \chiup + p \chiacross +  2^\lambda p\chidown, \end{equation}
and similar reasoning allows us to write down the system of linear equations
\begin{align}
\begin{array}{rrrrrrrrr}
\chiup &= &1 &+ &2\cdot 2^{-\lambda}p \chiup &+ &p \chiacross  &&\\
\chiacross &= &1&+ &2\cdot 2^{-\lambda} p \chiup &&  &+ &2^\lambda p \chidown\phantom{.}  \\
\chidown &= &1 &+ &2^{-\lambda}p \chiup  &+ &p\chiacross  &+  &2^\lambda p \chidown.
\end{array}
\end{align}
Solving these equations leads to the expressions
\begin{align}
\pcl{\lambda} &= \frac{2^\lambda+2^{1-\lambda}+1 - \sqrt{(2^\lambda+2^{1-\lambda}+1)^2-12}}{6} & \lambda \in \R \\
\intertext{and}
\chi_{p,\lambda} &= \frac{1-p^2}{1 - (2^\lambda+2^{1-\lambda}+1)p+3p^2} & \lambda \in \R,\, 0 \leq p<\pcl{\lambda}. \label{eq:orientedtreechi}
% \end{align}
\intertext{In particular, we have that $p_c(T,\Gamma',\lambda) < p_\mathrm{t}(T,\Gamma')$ for every $\lambda \neq 1/2$, as we predict to hold in general (\cref{conj:pclambda}). Moreover, we have that $p_t(T,\Gamma')<p_t(T,\Gamma)$, so that the triangle condition holds at $p_t(T,\Gamma')$. The denominator of \eqref{eq:orientedtreechi} never has a double root, so that, in contrast to the previous example,}
% \begin{align}
 \chi_{\pcl{\lambda}-\eps,\lambda} &\asymp_{\lambda} \eps^{-1} & \lambda \in \R,\, \eps > 0 \end{align}
 for every $\lambda \in \R$. 
 % Since the triangle condition holds at $p_t(T,\Gamma')$, this behaviour can also be deduced from \cref{thm:tiltedsusceptibilityexponent}. 
Using \cref{cor:lambdatoalpha} and taking the inverse of our expression for $\pcl{\lambda}$, we also obtain that
\begin{align}\alpha_p &= \log_2\left(\frac{3p^2-p+1+\sqrt{9p^4-6p^3-p^2-2p+1}}{2p}\right) & p \leq p_t. \end{align}

This example also has the following surprising property.
 % that $\alpha_p$ jumps from $1/2$ to $0$ at the tiltability threshold $p_t$. 

\begin{prop}
Let $T$ be a $4$-regular tree and let $\Gamma$ be the group of automorphisms fixing some specified $(1,1,2)$-orientation of $T$ as above. Then
$p_t(T,\Gamma)=p_h(T,\Gamma)$, and $\alpha_{p}=0$ for all $p>p_t$.
\end{prop}

\begin{proof}
Let $p>p_t$. Since $\alpha_p$ is strictly increasing when it positive, we have that $\alpha_p<1/2$. Let $\eps>0$ be sufficiently small that 
$1-\alpha_{p_t-\eps}>\alpha_{p}$.
% \[\lim_{n\to\infty} \end{equation}
Let $\rho$ be a fixed root vertex of $T$, and let $X^{-k,0}_{-k}$ be the number of vertices $v$ with $\Delta(\rho,v)=2^{-k}$ that are connected to $\rho$ by a path in the slab $\{u\in V : 2^{-k} \leq \Delta(\rho,u) \leq 1 \}$ that is open in $T[p_t-\eps]$. 
 By \cref{prop:tamedecay,prop:tamesecondmoment}, there exist positive constants $c_1$ and $c_2$, depending on $\eps$, such that
\begin{equation}
\P_{p_t-\eps}\left(X^{-k,0}_{-k} \geq c_1 2^{(1-\alpha_{p_t-\eps}) k}\right) \geq c_2
\end{equation}
for every $k\geq0$. 
% for infinitely many $k$ almost surely on the event that the cluster of $\rho$ in $T[p_t-\eps]$ is infinite. 

% Now, observe that
For each vertex $v$ with $\Delta(\rho,v)=2^{-k}$ that is connected to $\rho$ in the slab $\{u \in V :2^{-k} \leq \Delta(\rho,u) \leq 1\}$, there are two paths of length three $\eta_1(v)$ and $\eta_2(v)$ that first go down one level, then across the unoriented edge, and then back up one level. Observe that for any two such vertices $u$ and $v$, the four associated paths are mutually disjoint, and their endpoints are in four distinct connected components of $L_{-k,\infty}$. This leads to the inequality
% Conditioned on $X_{-k}^{-k,0}$, for each vertex $v$ with 
\begin{align}
\P\left(\rho \leftrightarrow L_n \text{ in } T[p] \mid X^{-k,0}_{-k}\right) \geq 1-\left[1-p^3 \P\left(\rho \xleftrightarrow{L_{0,\infty}} L_{n+k} \text{ in $T[p]$}\right)\right]^{2 X^{-k,0}_k},
% \\
% \geq 
\end{align}
from which we deduce that
\begin{equation}
\P\left(\rho \leftrightarrow L_n \text{ in } T[p] \right) \geq \P\left(X^{-k,0}_{-k} \geq c 2^{(1-\alpha_{p_t-\eps})k}\right)\left(1-\left[1-p^3 2^{-\alpha_p(n+k)+o(n+k)}\right]^{c2^{(1-\alpha_{p_t-\eps})k}}\right).
\end{equation}
Taking the limit as $k\to\infty$ we obtain that
\begin{equation}
\P_p(\rho \leftrightarrow L_n) \geq \limsup_{k\to\infty} \P\left(X^{-k,0}_{-k} \geq c 2^{(1-\alpha_{p_t-\eps})k}\right).
\end{equation}
Since this bound is positive and does not depend on $n$, we deduce that $p>p_h$. Since $p>p_t$ was arbitrary, it follows that $p_t=p_h$.

It remains to show that $\alpha_p=0$ for $p>p_h$. A result of Tim\'ar  \cite[Theorem 5.5]{timar2006percolation} states that if $G$ is a connected, locally finite graph, $\Gamma$ is a transitive nonunimodular subgroup of $G$, and $G[p]$ has infinitely many heavy clusters almost surely, then there exists a slab such that the open subgraph of the slab contains an infinite cluster almost surely. This implies that if $p_h<p<p_u$ then $\alpha_p=0$, and since $p_u=1$ in our example this concludes the proof. 
\end{proof}

In our opinion, once $p_c<p_u$ is established, the most interesting future direction for research on percolation on nonamenable graphs is to understand the nature of the uniqueness/nonuniqueness transition. Solving \cref{Q1,Q2}, below, would be a good start towards such an understanding. 

\begin{question}
\label{Q1}
What are the asymptotics of the connection probabilities  
\[\bP_{p_h}\left(v \leftrightarrow L_k(v)\right) \quad \text{ and } \quad \bP_{p_h}\left(v \xleftrightarrow{L_{0,\infty}(v)} L_k(v)\right)\]
in this example? Do the same asymptotics hold for other pairs $(G,\Gamma)$ such that the triangle condition holds at $p_h$? Do they bound from below the corresponding connection probabilities for all pairs $(G,\Gamma)$ where $G$ is connected and locally finite and $\Gamma$ is quasi-transitive and nonunimodular?
\end{question}

% We believe that the behaviour of this example at $p_t$ may be shared by other nonunimodular graphs whose levels are `high-dimensional' in some sense, similarly to how the behaviour of critical percolation on `high-dimensional' is known to or believed to be similar the the behaviour of critical percolation on a tree. In particular, we ask the following question.

\begin{question}
\label{Q2}
Let $T_k$ be a $k$-regular tree, let $d\geq 1$ and consider the group of automorphisms $\Gamma$ of $T_k \times \Z^d$ that fix some specified end $\xi$ of $T$. 
\begin{enumerate}
	\item Is $p_t=p_h$? 
	\item Is $\nabla_{p_t}<\infty$? Is $\nabla_{p_h}<\infty$?
	\item What are the asymptotics of the connection probabilities  
\[\bP_{p_h}\left(v \leftrightarrow L_k(v)\right) \quad \text{ and } \quad \bP_{p_h}\left(v \xleftrightarrow{L_{0,\infty}(v)} L_k(v)\right)?\]
	\item
What is the behaviour of $\chi_{p_{1/2}-\eps,1/2}$ as $\eps\to0$? 
\item Is $\alpha_p$ continuous in $p$? 
\item For which of these questions does the answer depend on $d$? 
\end{enumerate}
\end{question}

\subsection{An example with $p_h<p_u<1$}
\label{subsec:examplesphpu1}

In this section, we construct an example of a connected, locally finite graph $G$ and a transitive, nonunimodular group $\Gamma \subseteq \Aut(G)$ such that $p_c(G)<p_h(G,\Gamma)<p_u(G)<1$. Although we expect many graphs to have this property, it had not previously been proven to hold in any example (see \cite{timar2006percolation}). With a little further work, the example can be modified to obtain an example with $\Gamma=\Aut(G)$.

First, note that the proof of \cref{prop:tiltedmeanfieldlowerbound} also yields the following anisotropic version of that proposition: Let $G=(V,E)$ be a connected, locally finite graph, let $\Gamma\subseteq \Aut(G)$ quasi-transitive, and let $\cO_E$ be a set of orbit representatives for the action of $\Gamma$ on $E$. If $\mathbf{p}:\cO_E\to[0,1]$ and $\eps:\cO_E\to [0,1]$ are such that $p_e+\eps_e \in [0,1]$ for every $e\in E$, then
\begin{equation}
\label{eq:inhomomeanfield}
\max_{v\in V}\chi_{\mathbf{p}+\mathbf{\eps},\lambda}(v) \leq \sum_{i\geq0} \left(\max_{e\in \cO_E} \frac{\mathbf{\eps}_e }{1-\mathbf{p}_e}\right)^{i}\left(\max_{v\in V} \sum_{u\sim v}\Delta^\lambda (v,u)\right)^{i} \left(\max_{v\in V} \chi_{\mathbf{p},\lambda}(v)\right)^{i+1}. 
\end{equation}
(Indeed, simply take each edge $e$ to be open with probability $p_e$ and blue with probability $\eps_e/(1-p_e)$ and run the proof as before.) In particular, this inequality holds even if $\mathbf{p}_e=0$ for some $e\in E$. 

Let $T$ be a $4$-regular tree. As in the previous subsection, let $\Gamma$ be the group of automorphisms of $T$ fixing some specified end $\xi$ of $T$, and let $\Gamma'$ be the group of automorphisms of $T$ fixing some specified $(1,1,2)$-orientation of $T$. The groups $\Gamma \times \Aut(\Z)$ and $\Gamma'\times \Aut(\Z)$ act as automorphisms groups of $T\times \Z$ by acting separately on each coordinate.

In the previous subsection we established that $p_t(T,\Gamma')=p_h(T,\Gamma') < p_t(T,\Gamma)$. Let $p_h(T,\Gamma')<p<p_t(T,\Gamma)$. It follows from \eqref{eq:inhomomeanfield} that there exists $\eps>0$ such that anisotropic percolation on $T\times \Z$ in which tree-edges are open with probability $p$ and $\Z$ edges are open with probability $\eps$ has finite $\Gamma$-tilted susceptibility for $\lambda=1/2$, and therefore has no $\Gamma$-heavy clusters. On the other hand, it clearly contains $\Gamma'$-heavy clusters since $p>p_h(T,\Gamma')$. 
We can mimic this anisotropic percolation by replacing tree edges with multiple edges in parallel: If $T^k$ is obtained from the $4$-regular tree by replacing each edge with $k$ parallel edges, then if $k$ is sufficiently large the have that 
\begin{equation}
p_h\left(T^k\times \Z,\; \Gamma'\times \Aut(\Z)\right) < p_t\left(T^k\times \Z,\; \Gamma \times \Aut(\Z)\right) \leq p_u\left(T^k\times \Z\right).\end{equation}
On the other hand, we know that $p_u(T^k\times \Z)<1$ for every $k\geq 1$ \cite{MR1064560,MR1622785,MR2286517}, and so we have obtained an example for which $p_c<p_h<p_u<1$. 
% With a little further work, it is possible to turn this example into one in which $p_h$ is defined with respect to the full automorphism group.

% \[
% \Gamma \times \Aut(\Z) = \{\lambda \}
% \]

\subsection{Further questions and conjectures}
\label{subsec:questions}

We tried for some time but were unable to prove the following conjecture. The corresponding statement for self-avoiding walks is true and is proven in \cite{1709.10515}.

\begin{conjecture}
\label{conj:pclambda}
Let $G$ be a connected, locally finite graph, and let $\Gamma \subseteq \Aut(G)$ be quasi-transitive and nonunimodular. Then the function $\lambda \mapsto \pcl{\lambda}$ is continuous and strictly increasing on $(-\infty,1/2]$. 
% \[\pcl{\lambda}(G,\Gamma) < p_t(G,\Gamma) \]
% for every $\lambda \neq 1/2$. Equivalently, $\alpha_{p_t}=1/2$.
\end{conjecture}

Note that it follows from the results of \cref{sec:Fekete} that $\lambda \mapsto p_c(\lambda)$ is continuous and strictly increasing on the set $\{ \lambda \in (-\infty,1/2) : p_c(\lambda)<p_t\}$.

Let us also state the following conjecture on the equality or inequality of $p_h$ and $p_u$. This conjecture is suggested by the work of Tim\'ar \cite{timar2006percolation}, which establishes the `only if' part of the conjecture. We expect this conjecture to be harder to prove than \cref{conj:pcpu}.

\begin{conjecture}
Let $G$ be a connected, locally finite graph, and let $\Gamma \subseteq \Aut(G)$ be quasi-transitive and nonunimodular. Then $p_h(G,\Gamma)<p_u(G)$ if and only if there exists $u\in V$ and $t_1,\ldots,t_n\in \R$ such that the quasi-transitive subgraph of $G$ induced by $\{v\in V: \Delta_\Gamma(u,v) \in \{t_1,\ldots,t_n\} \}$ is nonamenable.
\end{conjecture}

Another natural questions concerns the dependency of $p_h$ on the group $\Gamma$. 

\begin{question}
Let $G$ be a connected, locally finite graph. 
If $\Gamma \subseteq \Gamma' \subseteq \Aut(G)$ are quasi-transitive subgroups of automorphisms of $G$, is it the case that  $p_h(G,\Gamma)>p_h(G,\Gamma')$ if and only if $\Delta_{\Gamma}\neq \Delta_{\Gamma'}$?
\end{question}

% Finally, it would be interesting to extend our methods to analyze continuum versions of percolation processes in the nonunimodular setting. 

% \begin{conjecture}[Voronoi percolation]
% Let $M$ be a connected Riemannian manifold, and let $\Gamma$ be a transitive and nonunimodular supgroup of $\operatorname{Isom}(M)$. Then for every $\delta$, there exists $p\in (0,1)$ such that Poisson-Voronoi percolation on $M$ with density $\delta$ and retention probability $p$ has infinitely many light clusters almost surely.
% \end{conjecture}

% Natural examples of include hyperbolic spaces $\mathbb{H}^n$ (the group of isometries fixing some specified point in the ideal boundary is transitive and nonunimodular), as well as products $\mathbb{H}^n \times \R^m$ of hyperbolic spaces with Euclidean spaces. Similar conjectures concerning e.g.\ Poisson-Boolean percolation should also hold, and are equally interesting. It would also be interesting to investigate e.g.\ branched polymers in these spaces, perhaps using techniques related to those of \cite{1709.10515}.

\subsection*{Acknowledgments}
This work took place mostly while the author was a PhD student at the University of British Columbia, during which time he was supported by a Microsoft Research PhD Fellowship. We thank Nicolas Curien, Hugo Duminil-Copin, Matan Harel, Aran Raoufi, and Yinon Spinka for useful discussions, particularly during a visit by the author to the IHES in March 2017. We also thank Omer Angel, Gady Kozma, Russ Lyons, and Asaf Nachmias for several helpful discussions, and also thank Russ Lyons for catching some typos in a previous version. Finally, we thank the three anonymous referees and Pengfei Tang for their close reading and detailed comments, which have greatly improved the paper.

\addcontentsline{toc}{section}{Glossary of notation}

\section*{Glossary of notation}
\label{sec:Glossary}

Our use of asymptotic notation is described at the end of \cref{sec:background}.
\begin{enumerate}[leftmargin=3cm]
\item[$p_c$, $p_u$, $p_h$] The critical probability, uniqueness threshold, and heaviness transition. Defined in \cref{sec:Intro} and \cref{subsec:intro_heavy} respectively.
\item[$\bP_p, \bE_p$] Probabilities and expectations taken with respect to the law of Bernoulli-$p$ bond percolation. Defined in \cref{subsec:intro_criticalexponents}.
\item[$\P_p, \E_p$] Probabilities and expectations taken with respect to the joint law of Bernoulli-$p$ bond percolation, the random root $\rho$. Later in the paper, this notation is used for the probabilty space that also includes the random variable $U$ used to define the uniform separating layers decomposition. Defined in \cref{subsec:MTP,subsec:layersdef}.
\item[$\nabla_p(v)$] The triangle diagram. Defined in \cref{subsec:intro_criticalexponents}.
\item[$\Delta$] The modular function. Defined in \cref{sec:Intro} (transitive case) and \cref{subsec:MTP} (general case).
\item[{$[v]$}] The orbit under $\Gamma$ of the vertex $v$. Defined in \cref{subsec:MTP}.
\item[$|A|_{v,\lambda}$, $|K_v|_{v,\lambda}$] The tilted volume $|A|_{v,\lambda} := \sum_{a \in A} \Delta^\lambda(v,a)$ of the set of vertices $A$, where the set $A$ is often taken to be the cluster $K_v$. Defined in \cref{subsec:intro_tilted}.
% \item[$|K_v|_{v,\lambda}$] The tilted volume of the cluster of $v$. Defined in \cref{subsec:intro_tilted}.
\item[$\chi_p(v)$, $\chi_{p,\lambda}(v)$] The susceptibility and tilted susceptibility. Defined in \cref{subsec:intro_tilted}.
\item[$p_c(\lambda)$, $p_t$]  The critical parameter associated to the tilted susceptibility and the tiltability threshold, respectively. Defined in \cref{subsec:intro_tilted}.
\item[$A \circ B$] The disjoint occurrence of $A$ and $B$. Defined in \cref{subsec:background_BK}.
\item[$M_{p,\lambda,h}$, $\chi_{p,\lambda,h}$] The tilted magnetization and truncated tilted susceptibility. Defined in \cref{sec:AizBar}.
\item[$\bP_{p,\lambda,h}$, $\bE_{p,\lambda,h}$] Probabilities and expectations taken with respect to the joint law of Bernoulli-$p$ bond percolation and the tilted ghost field. Defined in \cref{sec:AizBar}.
\item[$\P_{p,\lambda,h}$, $\E_{p,\lambda,h}$] Probabilities and expectations taken with respect to the joint law of Bernoulli-$p$ bond percolation, the tilted ghost field, and the random root $\rho$. Defined in \cref{sec:AizBar}.
\item[$\nexp$, $\nlog$] The normalized exponential and logarithm. Defined in \cref{subsec:layersdef}.
\item[$S_{s,t}(v)$] The slab $S_{s,t}(v)=\{u \in V : s \leq \log \Delta(v,u) \leq t\}$. Defined in \cref{subsec:layersdef}.
\item[$L_n(v)$, $L_{m,n}(v)$] Layers and slabs in the uniform separating layers decomposition. Defined in \cref{subsec:layersdef}.
\item[$U$] The random variable used to define the uniform separating layers decomposition. Defined in \cref{subsec:layersdef}.
\item[$X^{m,n}_k(v)$] The number of points in $L_k(v)$ connected to $v$ by an open path in $L_{m,n}(v)$. Defined in \cref{subsec:tiltable_overview}.
\item[$\alpha_p$] The exponential decay rate associated to the probability of crossing a large slab in the upward direction. Defined in \cref{subsec:probdecay}.
\item[$\beta_p$] The exponential decay rate associated to the expected number of points at the top of a large slab. Defined in \cref{subsec:expdecay}.
\item[$E_p^{m,n}(k)$] The supremum $\sup_{v\in V,x\in [0,1]} \E_p\left[X^{m,n}_k(v) \mid U_v=x\right]$. Defined in \cref{subsec:decomps}.
\item[$E_p^{m,n}(k;1-\eps)$] The supremum $\sup_{v\in V,x\in [0,1]} \E_p\left[\left(X^{m,n}_k(v)\right)^{1-\eps} \mid U_v=x\right]$. Defined in \cref{subsec:toolkit}.
\item[$\sP_v$, $\sP_v(k)$] The event that $v$ is the peak (unique highest point) of its cluster in either the whole space or in the slab $L_{-k,0}$, respectively. Defined in \cref{subsec:subcritical_peak} and \cref{subsec:peak_comparison}, respectively.
\end{enumerate}

\addcontentsline{toc}{section}{References}

\footnotesize{
  \bibliographystyle{abbrv}
  \bibliography{unimodularthesis}
  }

\end{document}

%% file: header.tex
\usepackage{amsmath,amssymb,amsfonts,amsthm}
\usepackage{color}

\usepackage[colorlinks=true,linkcolor=blue,citecolor=blue,urlcolor=blue,pdfborder={0 0 0}]{hyperref}
\usepackage{cleveref}

\usepackage{caption}
\usepackage{thmtools}

\usepackage{mathrsfs}

\crefname{theorem}{Theorem}{Theorems}
\crefname{thm}{Theorem}{Theorems}
\crefname{lemma}{Lemma}{Lemmas}
\crefname{lemmae}{Lemma}{Lemmas}
\crefname{lem}{Lemma}{Lemmas}
\crefname{remark}{Remark}{Remarks}
\crefname{prop}{Proposition}{Propositions}
\crefname{defn}{Definition}{Definitions}
\crefname{corollary}{Corollary}{Corollaries}
\crefname{conjecture}{Conjecture}{Conjectures}
\crefname{question}{Question}{Questions}
\crefname{chapter}{Chapter}{Chapters}
\crefname{section}{Section}{Sections}
\crefname{figure}{Figure}{Figures}

\theoremstyle{plain}
\newtheorem{thm}{Theorem}[section]

\newtheorem{lemma}[thm]{Lemma}

\newtheorem{lem}[thm]{Lemma}

\newtheorem{corollary}[thm]{Corollary}

\newtheorem{prop}[thm]{Proposition}
\newtheorem{conjecture}[thm]{Conjecture}

\newtheorem{question}[thm]{Question}
\theoremstyle{definition}

\theoremstyle{remark}
\newtheorem{remark}[thm]{Remark}
\newtheorem{lemmae}[thm]{Lemma}

\numberwithin{equation}{section}

%% file: commands.tex
\renewcommand{\P}{\mathbb P}
\newcommand{\E}{\mathbb E}

\newcommand{\R}{\mathbb R}
\newcommand{\Z}{\mathbb Z}
\newcommand{\N}{\mathbb N}

\newcommand{\cE}{\mathcal E}

\newcommand{\cG}{\mathcal G}

\newcommand{\cK}{\mathcal K}

\newcommand{\cO}{\mathcal O}

\newcommand{\sA}{\mathscr A}
\newcommand{\sB}{\mathscr B}

\newcommand{\sE}{\mathscr E}

\newcommand{\sP}{\mathscr P}

\newcommand{\sT}{\mathscr T}